\documentclass[11pt]{article}

\usepackage[a4paper]{geometry}
\usepackage{soul}  

\newcommand{\bdelta}{\boldsymbol{\rho}}
\newcommand{\brho}{\boldsymbol{\rho}}
\newcommand{\bsigma}{\boldsymbol{\sigma}}
\newcommand{\bff}{\mathbf{f}}
\newcommand{\bd}{\mathbf{d}}
\newcommand{\bfe}{\mathbf{e}}
\newcommand{\bb}{\mathbf{b}}

\newcommand{\bu}{\mathbf{u}}
\newcommand{\cS}{\mathcal{S}}
\newcommand{\cG}{\mathcal{G}}
\newcommand{\cP}{\mathcal{P}}

\usepackage{graphicx}
\usepackage{epstopdf}   
\usepackage{amsmath,amsfonts,amsthm,amssymb}
\usepackage{xspace,mathtools,lmodern,microtype}
\usepackage{mathrsfs}
\usepackage{trfsigns}
\usepackage{color}
\usepackage{ifpdf}
\usepackage{cite}
\usepackage{tikz}
\usepackage{xcolor}
\usepackage{float}
\usepackage{caption}
\usepackage{bigints}

\usepackage{lmodern}
\usepackage{psfrag}
\usepackage{pifont}
\usepackage{wrapfig}

\usepackage{titlesec}
\usepackage{mwe}
\usepackage{subcaption}

\usepackage[utf8]{inputenc}
\DeclareUnicodeCharacter{200E}{ }

\usepackage{color}
\usepackage[hypertexnames=false]{hyperref}
\usepackage{nameref}
\usepackage{cleveref}

\definecolor{myblue}{rgb}{0,0,0.6}
\hypersetup{colorlinks=true,
linkcolor=myblue,citecolor=myblue,filecolor=myblue,urlcolor=myblue}

\usepackage{abstract}

\usepackage{caption}
\usepackage{listings}
\newtheorem{theorem}{Theorem}[]
 
 \newtheorem{definition}[theorem]{Definition}
 \newtheorem{lemma}[theorem]{Lemma}
 \newtheorem{proposition}[theorem]{Proposition}

 \newtheorem{notation}[theorem]{Notation}
\newtheorem{discussion}[theorem]{Discussion}

 \newtheorem{remark}[theorem]{Remark}

\usepackage{algorithm}
\usepackage{algpseudocode}
\usepackage{algpascal}
\usepackage{setspace}

\usepackage{array,multirow}

\newcommand{\cR}{{\mathcal R}}
\newcommand{\ri}{{\rm i}}
\newcommand{\hm}{{m}}
\newcommand{\bm}{\boldsymbol{m}}
\newcommand{\tm}{\widetilde{m}}
\newcommand{\tbm}{\widetilde{\boldsymbol{m}}}
\newcommand{\bp}{\boldsymbol{p}}
\newcommand{\beps}{\boldsymbol{\varepsilon}}
\newcommand{\blambda}{\boldsymbol{\lambda}}

\numberwithin{equation}{section}
\numberwithin{theorem}{section}

\newcommand{\ben}{\begin{enumerate}}
\newcommand{\een}{\end{enumerate}}
\newcommand{\be}{\begin{equation}}
\newcommand{\ee}{\end{equation}}
\newcommand{\bdm}{\begin{displaymath}}
\newcommand{\edm}{\end{displaymath}}
\newcommand{\bea}{\begin{eqnarray}}
\newcommand{\eea}{\end{eqnarray}}

\def\={&=& }

\newcommand{\beq}{\begin{equation}}
\newcommand{\eeq}{\end{equation}}
\newcommand{\beqr}{\begin{eqnarray}}
\newcommand{\eeqr}{\end{eqnarray}}
\newcommand{\bef}{\begin{figure}}
\newcommand{\enf}{\end{figure}}
\newcommand{\bec}{\begin{center}}
\newcommand{\enc}{\end{center}}

\newcommand{\eps}{\varepsilon}

\usetikzlibrary{decorations.pathmorphing}
\usetikzlibrary{decorations.markings}
\usetikzlibrary{arrows,decorations.markings}
\usetikzlibrary{arrows.meta}
\tikzset{snake it/.style={decorate, decoration=snake}}
\usetikzlibrary{positioning,automata}

\tikzstyle{decision} = [diamond, draw, fill=blue!20, 
    text width=4.5em, text badly centered, node distance=3cm, inner sep=0pt]
\tikzstyle{block} = [rectangle, draw, fill=blue!5!white, 
    text width=18em, text centered, rounded corners, minimum height=8em, thick]
    \tikzstyle{blocks} = [rectangle, draw, fill=blue!5!white, 
    text width=8.8em, text centered, rounded corners, minimum height=6.5em, thick]
    \tikzstyle{blockf} = [rectangle, draw, fill=blue!5!white, 
    text width=6em, text centered, rounded corners, minimum height=3em, thick]
     \tikzstyle{blockc} = [rectangle, draw, fill=blue!5!white, 
    text width=18em, text centered, rounded corners, minimum height=8em, thick]
     \tikzstyle{blockm} = [rectangle, draw, fill=blue!5!white, 
    text width=18em, text centered, rounded corners, minimum height=8em, thick]
    \tikzstyle{blockr} = [rectangle, draw, fill=red!10!white, 
    text width=10em, text centered, rounded corners, minimum height=4em, thick]
    \tikzstyle{smallblock} = [rectangle, draw, fill=blue!5!white, 
    text width=18em, text centered, rounded corners, minimum height=8em, thick]
\tikzstyle{line} = [draw, -latex']
\tikzstyle{cloud} = [draw, ellipse,fill=red!20, node distance=3cm,
    minimum height=2em]

     \tikzstyle{blocksred} = [rectangle, draw, fill=red!5!white, 
    text width=8.8em, text centered, rounded corners, minimum height=6.5em, thick]

  \tikzstyle{blocklarge} = [rectangle, draw, fill=blue!5!white, 
    text width=15em, text centered, rounded corners, minimum height=8em, thick]
\tikzstyle{blocklargered} = [rectangle, draw, fill=red!5!white, 
    text width=15em, text centered, rounded corners, minimum height=8em, thick]

 \tikzstyle{blocklargeh} = [rectangle, draw, fill=blue!5!white, 
    text width=11em, text centered, rounded corners, minimum height=6em, thick]
\tikzstyle{blocklargeredh} = [rectangle, draw, fill=red!5!white, 
    text width=11em, text centered, rounded corners, minimum height=6em, thick]  
  
\usetikzlibrary{automata,positioning}
\usepackage[export]{adjustbox}

\usepackage{wrapfig}

\title{{Optimising seismic imaging design parameters via bilevel
learning}}
\author{Shaunagh Downing,    Silvia Gazzola,  Ivan G. Graham,  Euan ~A.~Spence  \\[1ex]
  \small{\tt S.Downing@bath.ac.uk, S.Gazzola@bath.ac.uk, 
     I.G.Graham@bath.ac.uk, E.A.Spence@bath.ac.uk}\\[1ex]
  Department of Mathematical Sciences, University of Bath, Bath, BA2 7AY, UK
}

\date{
    \today
}

\begin{document}

\maketitle

\noindent
\begin{abstract}
  Full Waveform Inversion (FWI) is a standard algorithm in seismic imaging. {It solves the inverse problem of computing} a model of the physical properties of  the earth's
  subsurface by minimising the misfit between actual  measurements of scattered seismic waves
  and numerical predictions of these, with the latter obtained  by solving the (forward) wave equation.
The   implementation of FWI requires the a priori choice of a number of ``design parameters'', such as the positions of sensors for the actual measurements and one (or more) regularisation weights.
In this paper we describe a  novel 
algorithm for determining these design parameters automatically from a set of training images, using a (supervised) bilevel learning approach. In our algorithm, the upper level objective function measures the quality of the reconstructions of the training images,  where  the reconstructions are obtained by solving  the lower level optimisation problem  -- in this case FWI. Our algorithm employs  (variants of)  the BFGS quasi-Newton method to perform the optimisation at each level,   and thus requires the repeated solution of the forward problem -- here taken to be the Helmholtz equation.
 This paper focuses on the implementation of the  algorithm. {The novel contributions are: (i)  an adjoint-state method for the efficient computation of the upper-level gradient; (ii) a complexity analysis for the bilevel algorithm, which counts the number of  Helmholtz solves needed and shows this number is independent of the number of design parameters optimised;  (iii) an effective  preconditioning strategy for iteratively
solving
the linear systems required at each step of the bilevel algorithm; (iv) a smoothed extraction process for point values of the discretised wavefield, necessary for ensuring a smooth upper level objective function.
The algorithm also uses an extension to the
bilevel setting of  classical frequency-continuation strategies, helping avoid convergence to  spurious stationary points. } 
The advantage of our algorithm is demonstrated on a problem derived from the standard Marmousi test problem.
\end{abstract}

\section{Introduction}

\textit{Seismic Imaging} is the process of  computing  a structural image of the interior of a body (e.g., a part of the earth's subsurface),  from \textit{seismic data} (i.e., measurements of artificially-generated waves that have propagated within it).  Seismic imaging is used widely 
in  searching  for mineral deposits 
and  archaeological sites underground, or to acquire geological information; see, e.g. \cite[Chapter 14]{sheriff1995exploration}. 
The generation of   seismic data
typically requires fixing a configuration of \textit{sources} (to generate the waves)
and \textit{sensors} (to measure the scattered field).  
The subsurface properties  being reconstructed   (often called \textit{model parameters}) can include  density, velocity or moduli of elasticity.  
There are various reconstruction methods  -- the method we choose for this paper is Full Waveform Inversion (FWI).  FWI is widely used in geophysics and has been applied also  in  $\mathrm{CO}_2$ sequestration (e.g., \cite{brown2009monitoring}) and  in medical imaging (e.g., \cite{guasch2020full, lucka2021high}).

\smallskip 

\noindent  {\bf Motivation for application of  bilevel learning.}\  
The quality of the result produced by FWI depends on the choice of  various \textit{design parameters}, 
 related both to the experimental
 set-up (e.g. placement of  sources or sensors) and to the design of the inversion algorithm  (e.g., the choice of FWI objective function, regularization strategy, etc.). In  exploration seismology, careful planning of seismic surveys is essential in achieving cost-effective acquisition and processing, as well as high quality data. Therefore in this paper we explore the potential for optimisation of these design parameters using a bilevel learning  approach, driven by a set of pre-chosen  `training models'. 
 Although there are many design parameters that could be optimised, we  restrict attention here to 
 `optimal' sensor placement and `optimal' choice of regularisation weight. 
However the general principles of our approach  apply more broadly. 
  
  There are many applications of FWI where  optimal sensor placement can be important.
  For example,  carbon sequestration  involves first characterising a candidate site via seismic imaging,
  and then monitoring the site over a number of years to ensure that the storage is performing
  effectively and safely. 
  Here 
  accuracy of images is imperative and optimisation of sensor locations could be very useful in
  the long term.
  The  topic of the present paper is thus of practical interest, but also fits with the growth of contemporary 
  interest in bilevel learning in other areas of  inverse problems. A related  approach has  been used to learn sampling patterns for MRI \cite{sherry2020learning}. Reviews of bilevel optimisation in general can be found, for example,  in \cite{crockett2021bilevel, dempe2018bilevel, DeZe:20}.

  The
  application of bilevel learning in seismic imaging does not appear to have had much attention  in the literature. An exception is \cite{haber2003learning}, where the regularisation functional is optimised using a supervised learning approach on the upper-level and the
  model is reconstructed on the lower-level using a simpler linear forward operator (see \cite[Equation 12 and Section 4]{haber2003learning}).  To our knowledge, the current paper is the first  to study bilevel learning in the context of FWI.

   We optimise the design parameters by exploiting \textit{prior information} in the
form of \textit{training models}. 
In practical situations,  such prior information may be available due to 2D surveys, previous drilling or exploratory wells. {Methods for deriving ground truth models from seismic data are outlined in \cite{jones2012building}. }

We use these training models  to \textit{learn} `optimal' design parameters, meaning  that these give the best reconstructions of the training models, {according to some quality measure,} over all   possible choices of design parameters.  In the
bilevel learning framework, the upper level objective involves the misfit between the training models and their reconstructions, obtained  via FWI,  and the lower level is  FWI itself.

\smallskip

{Although different to  the bilevel learing approach considered here, the application  of more general experimental design techniques  in  the context of FWI has a much wider literature.  For example,  in \cite{HuDe:13} an improved method for detecting steep subsurface structures,  based on seismic interferometry,  is described.
  In \cite{LoMoTe:15} the problem of identification of position and properties of a seismic source in
  an elastodynamic model is considered, where the number of parameters to be identified is small relative to the number of observations. The problem is formulated in a  Bayesian setting, seeking the posterior probability
  distributions of the unknown parameters, given prior information.   An algorithm for computing the
  optimal number and positions of receivers is presented, based on maximising the expected  information gain in the solution of the inverse problem.  
  In \cite{MaNuetal:17} the  general tools of Optimised Experimental Design are reviewed and
  applied to optimse the benefit/cost ratio in various applications of FWI, including examples in seismology and medical tomography.   Emphasis is placed on determining experimental designs which maximise the size of the resolved model space, based on an analysis of the spectrum of the approximate Hessian of the model to observation map.
  The cost of implementing such design techniques is further addressed in the more recent paper
  \cite{KrEdMa:21}, where optimising the spectral properties is replaced by a  goodness measure based on the determinant of the approximate Hessian.    }


  \smallskip
  
 \noindent{\bf Contribution of this Paper.}\  
 This paper formulates a bilevel learning problem for optimising design parameters in FWI,
 and proposes a solution  using  quasi-Newton methods at both upper and lower level.
We derive a novel formula for the gradient of the upper level objective function with respect to the design parameters, and analyse the complexity of the resulting algorithm 
in terms of the  number of forward
solves needed. In this paper we work in the frequency domain, so  the forward problem is the Helmholtz equation.
The efficient running of the algorithm depends on several implementation techniques:~a  bilevel frequency-continuation technique, which helps avoid stagnation in spurious stationary points, and a novel extraction process, which ensures that the computed wavefield is a smooth function of sensor  position,
irrespective of the numerical grid used to solve the Helmholtz problem.  
Since the gradient of the upper level objective function involves the inverse of the Hessian of the lower-level objective function,
Hessian systems have to be solved at each step of the bilevel method; we present an effective   preconditioning technique for solving these systems via Krylov methods.

{While the number of Helmholtz solves required by the bilevel algorithm can be
substantial, we emphasise that the learning process  should be regarded as an off-line phase in the reconstruction process, i.e., it is computed once and then the output (sensor position, regularisation parameter) is used in subsequent data acquisition and associated standard FWI computations.}

Finally, we apply our novel bilevel algorithm to an inverse problem involving the reconstruction of a smoothed version of the Marmousi model. These  show that the  design parameters  obtained using  the  bilevel  algorithm provide better  FWI reconstructions (on test problems lying outside the space of training models) than those reconstructions  obtained with {\em a priori} choices of design parameters. A specific list of
contributions of the paper are given as items (i) -- (iv) in the abstract. 

\smallskip

\noindent {\bf Outline of paper.} \  In Section \ref{sec:Formulation}, we discuss the formulation of the bilevel problem, while Section \ref{sec:bilevel}  presents our  approach for solving   it, including its  reduction to a single-level problem, the derivation of the upper-level gradient formula,  
and  the   complexity analysis. {Section  \ref{sec:Numerical} presents numerical results for the
  Marmousi-type  test problem followed by a short concluding section. Some implementation details  are given in the  Appendix (Section \ref{app:Details}).}

\section{Formulation of the  Bilevel Problem}
\label{sec:Formulation}

Since one of our ultimate aims is to optimise sensor positions,
we formulate the FWI objective function in terms of  wavefields that solve the continuous  (and not discrete) Helmholtz problem, thus ensuring that the  wavefield depends smoothly on sensor position.
This is different from many papers about FWI,  where the wavefield is postulated to be the solution of  a discrete system; see, e.g., \cite{aghamiry2021full, HeQinglong2020INmb, MetivierLudovic2011A2ni, metivier2017full, van_Leeuwen_2010, operto_2009}. 
  When,   in practice, we work at the discrete level,  special measures are taken to ensure that
  the smoothness with respect to sensor positions is preserved as well as possible -- see  Section 
  \ref{subsec:restrict}. 
  
\subsection{The Wave Equation in the Frequency Domain}
While the FWI problem can be formulated using a  forward problem defined by any wave equation in either the time or frequency domain, we focus here  on the acoustic wave equation in the frequency domain, i.e.,  the Helmholtz equation. 
We  study this in a bounded  domain {$\Omega\subset \mathbb{R}^d$, $d = 2,3$} with boundary $\partial \Omega$ and with classical impedance first-order absorbing boundary condition (ABC). However it  is straightforward to extend to more general domains and boundary conditions (e.g. with obstacles, {perfectly-matched layer (PML)}  boundary condition, etc.).       

In this paper the model to be recovered in FWI is taken to be the `squared-slowness' (i.e.,  the inverse of the velocity squared),   specified by a vector of parameters  $\bm=(m_1,\dots,m_M) \in \mathbb{R}^M_+\linebreak[4]:= \{ \bm \in \mathbb{R}^M: m_k > 0, \, k = 1, \ldots, M\}$. We  assume that  $\bm$
determines a function on the domain $\Omega$ through a relationship of the form:  
 \begin{align} \label{disc_to_cont} m(x)=\sum_{k=1}^M m_k \beta_k(x),\end{align}
 for some basis functions $\{\beta_k\}$, assumed to have local support in $\Omega$. For example, we may choose
 $\beta_k$ as nodal finite element basis functions with respect to a mesh defined on $\Omega$ and then $\bm$ contains the nodal values.
 The simplest case is where {$\Omega \subset \mathbb{R}^2$} is a rectangle, discretised by a uniform {rectangular grid (subdivided into triangles)}  and $\beta_k$ are the continuous piecewise linear basis, is used in the experiments in this paper.
 
\begin{definition}[Solution operator and its adjoint]
  \label{def:solnop}
  For a given model $\bm$ and frequency $\omega$, 
  we define  the solution operator  $\mathscr{S}_{\bm, \omega}$ 
  by requiring 
  \begin{align}  
    \left(\begin{array}{l}u\\ u_b\end{array}\right) = {\mathscr{S}}_{\bm, \omega} \left(\begin{array}{l}f\\ f_b\end{array}\right) \quad \iff \quad \left\{ \begin{array}{rl} - (\Delta + \omega^2 \hm ) u & = f \quad \text{on} \quad \Omega ,  \\ (\partial  /\partial n - \ri \sqrt{m} \omega ) u & = f_b \quad \text{on} \quad \partial \Omega , \\
                                                                                                                                                         u_b & = u\vert_{\partial \Omega} \end{array} \right. , 
\label{forward}
  \end{align}
  for all $(f, f_b)^\top  \in L^2(\Omega) \times L^2(\partial \Omega)$ (where $L^2$ denotes square-integrable functions).  We also define the adjoint solution operator $\mathscr{S}^*_{\bm, \omega}$ by 
 \begin{align}  
       \left(\begin{array}{r}v\\ v_b\end{array}\right) = \mathscr{S}^*_{\bm, \omega}  \left(\begin{array}{l}g\\ g_b\end{array}\right) \quad \iff \quad \left\{ \begin{array}{rl} - (\Delta + \omega^2 \hm ) v & = g \quad \text{on} \quad \Omega \\ (\partial  /\partial n + \ri \sqrt{m} \omega ) v & = g_b \quad \text{on} \quad \partial \Omega, \\
v_b &= v\vert_{\partial \Omega}                                                                                                              \end{array} \right. ,
 \label{adjoint}
 \end{align}
 for all $(g, g_b)^\top \in L^2(\Omega) \times L^2(\partial \Omega)$.
\end{definition} 

  \begin{remark}\label{rem:firstarg}
 (i)  The solution operator $\mathscr{S}_{\bm, \omega}$ returns a vector with   two components, one being the solution of a Helmholtz problem on the domain $\Omega$ and
    the other being its restriction to  the boundary $\partial \Omega$.
    Pre-multiplication of this vector with the {(row)} vector $(1,0)$  just returns the solution on the domain.

    (ii)  When $\Omega$ has a Lipschitz boundary the solution operators are well-understood mathematically, and it can be shown that they both map   $L^2(\Omega) \times L^2(\partial \Omega)$ to the Sobolev space  $H^1(\Omega) \times H^{1/2}(\partial \Omega)$, but we do not need that theory here. \end{remark}

    We denote the $L^2$ inner products on $\Omega$ and $\partial \Omega$ by
    $( \cdot , \cdot)_{\Omega}$, $( \cdot , \cdot)_{\partial \Omega}$ and also 
    introduce the inner product on the product  space:
    \begin{align*}
\left(\left(\begin{array}{l}f\\ f_b\end{array}\right),  \left(\begin{array}{l}g\\ g_b\end{array}\right)\right)_{\Omega \times \partial \Omega} = (f,g)_{\Omega} + (f_b, g_b)_{\partial \Omega}. 
\end{align*}
Then,  integrating by parts twice (i.e., using Green's identity), one can easily obtain the property:
\begin{align} \label{Green}
\left(\mathscr{S}_{\bm, \omega} \left(\begin{array}{l}f\\ f_b\end{array}\right),  \left(\begin{array}{l}g\\ g_b\end{array}\right)\right)_{\Omega \times \partial \Omega} = \left(\left(\begin{array}{l}f\\ f_b\end{array}\right),  \mathscr{S}^*_{\bm, \omega} \left(\begin{array}{l}g\\ g_b\end{array}\right)\right)_{\Omega \times \partial \Omega} . 
  \end{align}

  Considerable simplifications could be obtained by assuming that  $m$ is constant on $\partial \Omega$; this is a natural assumption used  in theoretically justifying the absorbing boundary
    condition in  \eqref{forward} and \eqref{adjoint},  or for  more sophisticated ABCs such as a PML.
    However in some of the literature (e.g. \cite[Section 6.2]{VaHe:16})  $m$ is allowed to vary on $\partial \Omega$, leading to problems that depend nonlinearly on $m$ on $\partial \Omega$, as in \eqref{forward} and \eqref{adjoint}; we therefore 
    cover
  the most general case here.  
  
 We consider below wavefields generated by sources; i.e., solutions of the Helmholtz equation with the right-hand side a delta function.
  
  \begin{definition}[The delta function and its derivative]\label{def:part_delta}
    For any  point $r \in \Omega$ we define the delta function $\delta_r$ by
    $$(f, \delta_r) = f(r),$$
for all $f$ continuous in a neighbourhood of $r$.      Then, for $l = 1, \ldots , d$, we define the generalised function $\frac{\partial}{\partial x_l}(\delta_r)$ by 
 \begin{align*}
   \left(f, \frac{\partial}{\partial x_l}(\delta_r) \right)  :=  - \left(\frac{\partial f}{\partial x_l}, \delta_r \right), 
 \end{align*}
 for all $f$ continuously differentiable in a neighbourhood of $r$. \end{definition}

\subsection{The lower-level problem}
The lower-level objective  of our bilevel problem is the classical FWI objective function:
\begin{align}
\phi(\bm, \cP, \alpha)= \frac12 \sum_{s \in \mathcal{S}} \sum_{\omega \in \mathcal{W}} \Vert \beps(\bm,\cP, \omega, s) \Vert_2^2   + \frac{1}{2} \bm^\top \Gamma(\alpha,\mu) \bm . 
\label{phi}
\end{align}
{In the notation for $\phi$,}  we distinguish three {of its}  independent variables:   (i)  $\bm \in \mathbb{R}^M$ denotes the model (and the lower level problem  consists of   minimising $\phi$ over all such $\bm$); (ii) $\cP = \{p_j: j = 1, \ldots, N_r\}$  denotes the set of $N_r$ sensor positions, with each $p_j \in \Omega$,  and  (iii) $\alpha$ is a regularisation parameter.  
In \eqref{phi},  {$\phi$ also depends on other parameters but we do not list these as independent variables}:     $ \mathcal{S}$ is a finite set of  source positions, $ \mathcal{W}$ is a finite set of  frequencies, $\Gamma(\alpha,\mu)$ is a real symmetric positive semi-definite regularisation matrix (to be defined below). {We assume throughout that sensors cannot coincide with sources.} Moreover, $\Vert \cdot \Vert_2$ denotes  the usual Euclidean norm on $\mathbb{C}^{N_r}$ and  $\beps\in \mathbb{C}^{N_r}$ is the vector of    ``data misfits'' at  the $N_r$ sensors, defined by    
 \begin{align} 
\label{resd}
   \beps(\bm, \bp,\omega, s)  = \textbf{d}(\bp,\omega,s) - \mathcal{R}(\cP) u(\bm,\omega, s) \in \mathbb{C}^{N_r}, 
   \end{align} 
where $\textbf{d} \in \mathbb{C}^{N_r}$ is the data,  
$u(\bm, \omega,s)$ is the wavefield obtained by solving the Helmholtz equation with model $\bm$, frequency $\omega$, source $s$ and zero impedance data, i.e.,
\begin{align} \label{defu}  u(\bm, \omega,s) = \mathscr{S}_{\bm,\omega}\left(\begin{array}{l}\delta_s\\0 \end{array} \right), \end{align}
and $\mathcal{R}(\cP)$ is the restriction operator, which  evaluates  the wavefield at sensor positions, i.e., 
 \begin{align}
   \mathcal{R}(\cP) u &= \left[ u(p_{1}), u(p_{2}), \ldots, u(p_{N_r}) \right]^\top  =  \left[ (u, \delta_{p_1}) , (u, \delta_{p_{2}}), \ldots, (u, \delta_{p_{N_r}}) \right]^\top  \in \mathbb{C}^{N_r}. 
         \label{Rv}
 \end{align}
We also need the adjoint operator $\mathcal{R}(\cP)^*$ defined by  
      \begin{align}
       \mathcal{R}(\cP)^* \boldsymbol{z} = \sum_{j = 1}^{N_r} \delta _{p_j} z_j, \quad \text{for} \quad \boldsymbol{z} \in \mathbb{C}^{N_r} .
         \label{Rstar}
      \end{align}
      It is then easy to see that, with $\langle \cdot, \cdot \rangle$ denoting the Euclidean inner product on $\mathbb{C}^{N_r}$,
      \begin{align}  
      \langle \mathcal{R}(\cP)u , \boldsymbol{z} \rangle = (u, \mathcal{R}(\cP)^* \boldsymbol{z})_\Omega .
      \label{I1} 
      \end{align} 

       \paragraph*{Regularisation}
       The general form of the regularisation matrix in  \eqref{phi}  is 
       \begin{align} \Gamma(\alpha, \mu) = \alpha \mathsf{R} + \mu I \label{Gamma}
       \end{align}
       where $I$ is the $M \times M$ identity,  $\mathsf{R}$ is an $M \times M$ real  positive semidefinite matrix that approximates the action of the negative Laplacian on the model space and $\alpha, \mu$ are positive parameters to be chosen.
       In the computations in this paper,  $\Omega\subset \mathbb{R}^2$  is a rectangular domain discretised by a rectangular grid with $n_1$ nodes in the horizontal ($x$) direction and $n_2$ nodes in the vertical ($z$) direction, in  which case  we  make the particular choice
       \begin{align} \label{ThisGamma}
         \mathsf{R}  = D_x^T D_x + D_z ^T D_z,
       \end{align}
       where $D_x:=D_{n_1} \otimes I_{n_2} $,  $D_z:= I_{n_1} \otimes D_{n_2}$,
          $\otimes$ is the  Kronecker product and $D_n$ is  the  difference matrix
       \begin{align*} 
D_n&= (n-1)\left(  \begin{matrix}
1 & -1 &  & & & &\text{\rm \large 0}\\
  & 1 & -1 & & & &\\
  &  & \ddots & \ddots & & & &\\
  & & & & 1& -1 &\\
 \text{\rm \large 0} & & & & & 1 &-1
 \end{matrix}
                                   \right) \in \mathbb{R}^{\left(n-1 \right) \times n},
\end{align*}

For general domains and discretisations, the relation \eqref{disc_to_cont} could be exploited to provide the matrix $\mathsf{R} $  by defining $\mathsf{R}_{k,k'}  = \int_\Omega \nabla \beta_k \cdot \nabla \beta_{k'}$, so that $\bm^T \mathsf{R}  \bm = \Vert \nabla m
\Vert_{L^2(\Omega)}^2$.  

In this paper,  $\alpha$ and   (some of) the coordinates  of the points in $\cP$  are designated {\em design parameters}, to  be found by optimising the upper level objective   $\psi$  (defined  below).
We  also tested algorithms that included $\mu$ in  the list of design parameters, 
but these failed to substantially improve the reconstruction of $\bm$.
However  the choice of a small fixed   $\mu>0$ (typically of the order of $10^{-6}$)  ensured the stability of the algorithm in practice. Nevertheless,   $\mu$ plays an important role in the theory, {since large enough $\mu$
ensures the positive-definiteness of the Hessian of $\phi$ and hence strict convexity of $\phi$.}
The inclusion of the term $\mu \Vert \bm \Vert_2^2$ in the regulariser
typically appears in FWI
theory and practice,  
 sometimes  in the more general form   $\mu ||\bm-\bm_0||_2^2$, where $\bm_0$ is a `prior model'; for example see \cite{asnaashari2013regularized}, \cite[Section 3.2]{tarantola2005inverse} and \cite[Equation 4]{aghamiry2021full}. 

  \subsection{Training models and the bilevel problem }
  \label{subsec:training}

  The main purpose of this paper is to show that, given a carefully chosen set $\mathcal{M}'$ of training models, one can learn good  choices of design parameters, which then provide an FWI algorithm with enhanced performance in more general applications.
  The good  design parameters are found by minimising the misfit between the  ``ground truth''   training
  models  and their FWI reconstructions. Thus,  we  are applying FWI in the special situation
  where the data $\mathbf{d}$ in \eqref{resd} is synthetic,  given by
  $\mathbf{d}(\bm',\omega , s) = \cR(\cP)u(\bm', \omega, s)$,   and so, {for each training model
    $\bm' \in \mathcal{M}'$}, 
    we rewrite
  $\phi$ in \eqref{phi} as:
  \begin{align}
    \phi(\bm, \cP, \alpha, \bm')& = \frac12 \sum_{s \in \mathcal{S}} \sum_{\omega \in \mathcal{W}} \Vert \beps(\bm,\cP, \omega, s, \bm') \Vert_2^2   + \frac{1}{2} \bm^\top \Gamma(\alpha,\mu) \bm ,
                                  \label{newphi}
  \end{align}
     \text{with}        \begin{align}  \beps(\bm, \cP,\omega, s, \bm')      & := \mathcal{R}(\cP) (u(\bm',\omega, s) - u(\bm,\omega, s)) , 
\label{resdbi}
                        \end{align}
                        where we have now added $\bm'$ to the independent variables of $\beps$ to emphasise  dependence on $\bm'$.
                        
 Then, letting  $\bm^{\rm FWI}(\cP,\alpha, \bm'$) denote a minimiser (over all models $\bm$) of $\phi$ (given by  \eqref{newphi}, \eqref{resdbi})),
   for each training model $\bm' \in \mathcal{M'}$, sensor position $\cP$ and regularisation  parameter $\alpha$, 
   the upper level objective function is defined to be  
 \begin{align}
\psi(\cP, \alpha)&:=\frac{1}{2N_{m'}}\sum_{ \bm' \in \mathcal{M'}}||\bm'-\bm^{\rm FWI}(\cP,\alpha,  \bm')||_2^2,
\label{SOobj} 
 \end{align}
 {where $N_{\bm'}$ denotes the number of training models in $\mathcal{M}'$.}
 
\begin{definition}[General bilevel problem] \label{def:bilevel} \ With $\psi$  defined by \eqref{SOobj} and $\phi$  defined by \eqref{newphi}:
\begin{align}
  \rm{Find} \quad  \cP_{\min}, \alpha_{\min} &=\underset{\cP, \alpha}{\mathrm{argmin}} \, \, \, \psi (\cP, \alpha), \label{upper}\\
 \textrm{subject to }\, \, \bm^{\rm FWI}(\cP,\alpha, \bm' ) & \in  \underset{\bm}{\mathrm{argmin}}\,\, \,  \phi (\bm, \cP,\alpha,\bm') \quad \mbox{ for each } \bm' \in \mathcal{M'}\label{lower}
\end{align}
\end{definition}

{Using the theory of the Helmholtz equation  it can be shown that
  the solution $u(\bm,\omega,s)$ depends continuously on $\bm$ in any domain which does not include the sensors
  (details of this are in  \cite[Section 3.4.3]{thesis}).  Since   $\phi$  is non-negative, it follows that}  $\phi$ has at least one minimiser with respect to  $\bm$.     However  since $\phi$ is not necessarily convex {(we do not know in practice if the value of $\mu$ chosen guarantees this)}, the $\mathrm{argmin}$ function in \eqref{lower} is potentially multi-valued.  This leads to an ambiguity in the defintion of $\psi$ in \eqref{upper}. To deal with this, 
  we replace \eqref{lower} by its first-order optimality condition (necessarily satisfied by any minimiser in $\mathbb{R}_+^M$).  While this in itself does not guarantee a unique solution at the lower level, it does  allow us to compute the gradient of $\psi$  with respect to any coordinate of the points in $\cP$ or with respect to   $\alpha$,  {\em under the assumption of uniqueness at the lower level}.
 {Such an approach  for  dealing with the  non-convexity of the lower level problem
is widely used in bilevel learning -- see also sources cited in the recent review \cite[Section 4.2]{crockett2021bilevel} (where it is called `the Minimizer Approach') and \cite{Matthias23, sherry2020learning}. 
Definitions \ref{def:bilevel} and the following Definition \ref{def:reduced} are  equivalent when $\mu$ is sufficiently large.}

\begin{definition}[Reduced  single level problem]\label{def:reduced}   \begin{align}
\textrm{Find} \quad  &\cP_{\min}, \alpha_{\min}=\underset{\cP, \alpha}{\mathrm{argmin}} \,\psi (\cP, \alpha)  \nonumber \\
\mbox{subject to} \quad &\nabla \phi(\bm^{\rm FWI}(\cP,\alpha, \bm'),  \cP, \alpha, \bm')=\mathbf{0} \quad \mbox{for each } \bm' \in \mathcal{M'},
\label{single}
  \end{align}
  where  $\nabla \phi$  denotes the gradient of $\phi$ with respect to  $\bm$. 
\end{definition} 

More generally, the  literature on bilevel optimisation contains several approaches to deal with  the non-convexity of the lower level problem.
For example in \cite[Section II]{sinha}, the  `optimistic' (respectively  `pessimistic') approaches are discussed, which 
means that one fixes  the lower-level minimiser as being one  providing the smallest (respectively largest)    value of  $\psi$. However it is not obvious how to implement  this scheme in practice.  

    We now denote the gradient and Hessian of $\phi$ with respect to $\bm$ by $\nabla \phi$ and $H$ respectively.
     {These are both functions of $\bm, \cP$,  $\alpha$  and $\bm'$ although
     and  $H$ is independent of $\bm'$,  so we write $\phi = \phi(\bm, \cP, \alpha, \bm')$ and  $H=H(\bm, \cP, \alpha)$.}
     An  explicit  formula for $\nabla \phi$ is given in the following proposition.
     It involves the operator
       $\cG_{\bm, \omega}$ on $L^2(\Omega)\times  L^2(\partial \Omega)$ defined  by
     \begin{align*} 
     \cG_{\bm, \omega} \left(\begin{array}{l} v\\ v_b\end{array} \right)  := \left(\begin{array}{l}\omega^2 v \\ \frac{\ri \omega}{2} \left(\frac{v_b} {\sqrt{m}} \right)  \vert_{\partial \Omega}\end{array} \right).
     \end{align*} 
   \begin{proposition}[{Derivative of $\phi$ with respect to $m_k$}]\label{prop:elem} 
Let $\Re$ denote the real part of a complex number. Then
     \begin{align}
  \frac{\partial \phi }{\partial m_k}(\bm, \cP, \alpha, \bm') \ & = \ - \Re  \sum_{s \in \mathcal{S}} \sum_{\omega \in \mathcal{W}} \left\langle \mathcal{R}(\cP) \frac{\partial u }{\partial m_k}(\bm,  \omega, s) , \beps(\bm, \cP, \omega, s, \bm') \right\rangle   + \Gamma(\alpha,\mu) \bm,    
                                                                  \label{graddirect_here}\end{align}
    and
    \begin{align}                                                                                                                         
  \left( \begin{array}{l} \frac{\partial u }{\partial m_k}\\\frac{\partial u }{\partial m_k}\vert_{\partial \Omega}\end{array} \right) & =  \, \mathscr{S}_{\bm, \omega}  \cG_{\bm, \omega} \left(\begin{array}{l} \beta_k u(\bm, \omega,s)\\ \beta_k u(\bm, \omega,s)\vert_{\partial \Omega} \end{array} \right) .    \label{oneMgrad2a}\end{align}
\end{proposition}
\begin{proof}  \eqref{graddirect_here} follows from  differentiating \eqref{newphi}.   \eqref{oneMgrad2a}  is obtained by differentiating \eqref{defu} 
 and using   \eqref{disc_to_cont}. \end{proof} 
 
 \begin{remark} \label{rem:note}In what follows, it is useful to note that, for any $\bsigma=(\sigma_1,\dots,\sigma_M) \in \mathbb{R}^M$,   using 
 the first row of \eqref{oneMgrad2a} and the  linearity of $\mathscr{S}_{\bm, \omega} $ and $\cG_{\bm, \omega}$, we obtain
 \begin{align} \sum_{k=1}^M \sigma_k \frac{\partial u(\bm, \omega,s) }{\partial m_k} = (1,0)\,  \mathscr{S}_{\bm, \omega} \cG_{\bm, \omega} (\sigma u(\bm, \omega,s)),  \quad \text{where} \quad   \sigma := \sum_{k = 1}^M \sigma_k \beta_k.      \label{oneMgrad2b}\end{align}
\end{remark}

\section{Solving the Bilevel Problem}
\label{sec:bilevel}

We apply  a quasi-Newton  method to solve the Reduced   Problem in Definition  \ref{def:reduced}. To implement this,   we need formulae for the derivative  of $\psi$ with respect to (some subset)  of the coordinates $\{p_{j,\ell}: j = 1, \ldots, N_r, \ \ell = 1,\ldots, d\}$
{of the points in $\mathcal{P}$,} as well as the parameter $\alpha$. {In the optimisation we may choose to constrain some of these coordinates, for example if the sensors lie in a well or on the surface.} In deriving these formulae, we use  the fact that  $\psi$ is a $C^1$ function of these variables;  this can be proved (for sufficiently large $\mu$) using the Implicit Function Theorem. {More precisely, the equation \eqref{single} can be thought of as a system of $M$ equations determining $\bm^{\rm FWI}$ as a $C^1$ function of the parameters in $\mathcal{P}$ and/or $\alpha$. The positive definiteness of the Hessian of $\phi$ allows an application of the Implicit Function Theorem to this system. This argument also justifies the formula \eqref{dmdpinter} used below. More details are in   
  \cite[Corollary 3.4.30]{thesis}.}  

  \subsection{Derivative of $\psi$ with respect to position coordinate  $p_{j,\ell}$  } 
\label{sect:gradient}

 The formulae  derived in Theorems \ref{sogradt1} and \ref{sogradta} below   involve the solution $\brho$ of the system \eqref{rho} below.   {In \eqref{rho} the system matrix is the  Hessian of $\phi$ and the right-hand side is  given by   the discrepancy between the training
 model $\bm'$ and its FWI reconstruction $\bm^{\rm FWI}(\cP, \alpha, \bm')$.} The existence of  $\brho $ is guaranteed by the following proposition.     
 \begin{proposition} \label{lem_sogradt}
   Provided 
   $\mu$ is sufficiently large, then,  for any   collection of sensors  $\cP$, regularisation parameter   $\alpha > 0 $, and   $ \bm, \bm' \in \mathbb{R}^M$,  the Hessian
   $H(\bm, \cP , \alpha) $
  is non-singular and 
  there is a unique  $\bm^{\rm FWI}(\cP,\alpha, \bm')\in \mathbb{R}^M$ satisfying
   \eqref{single}. {From now on,  we abbreviate this}   by writing  
   \begin{align}
   \bm^{\rm FWI}=\bm^{\rm FWI}(\cP,\alpha, \bm'). 
    \label{eq:abbrev} \end{align} 
\end{proposition}
\begin{proof}
  The result follows because  $H$  is symmetric and  positive definite
  when  $\mu$ sufficiently large.
\end{proof}

Under the conditions of Proposition \ref{lem_sogradt},   the linear   system
  \begin{align} \label{rho}
 H(\bm^{\rm FWI}, \cP, \alpha)\, \brho =  \bm'-\bm^{\rm FWI}  
\end{align}
has a unique solution  $\brho = \brho(\cP,\alpha, \bm')  \in \mathbb{R}^M$,
and we can  define the corresponding
function $\rho = \rho(\cP,\alpha, \bm')$ on $\Omega$ by
\begin{align} \label{rho_fn}
                                                \rho  = \sum_{k = 1}^{M} \rho_k \beta_k .
\end{align}

\begin{theorem}[Derivative of $\psi$ with respect to $p_{j,\ell}$]
    \label{sogradt1}
    If $\mu$ is sufficiently large,  then, for   $ j = 1, \ldots, N_r$,  and
    $\ell = 1, \ldots , d$,   $\partial \psi/\partial p_{j,\ell}$  exists and can be written
  \begin{align}  \frac{\partial \psi }{\partial p_{j,\ell}} (\cP, \alpha)\  =  
   \  & \frac{1}{N_{m'}}\sum_{\bm' \in \mathcal{M'}}\sum_{s \in \mathcal{S}}\sum_{\omega \in \mathcal{W}} \Re \,  a_\ell(\bm^{\rm FWI},\omega,s, \bm'; p_j) . \label{gradALT} 
  \end{align}
  Here, for each $\ell$, $a_\ell(\bm^{\rm FWI},\omega,s, \bm';p_j)$ denotes the  evaluation of the function
$a_\ell(\bm^{\rm FWI},\omega,s, \bm')$  at the point $p_j$, where   
  \begin{align} a_\ell(\bm^{\rm FWI},\omega,s, \bm') & = 
      \tau(\bm^{\rm FWI}, \omega,s, \bm') \left(\frac{\partial u}{\partial x_\ell} (\bm^{\rm FWI},\omega,s) - \frac{\partial u}{\partial x_\ell} (\bm',\omega,s)\right)\nonumber \\
    &  \qquad+ \frac{\partial \tau}{\partial x_\ell} (\bm^{\rm FWI}, \omega, s, \bm')
      \left( u (\bm^{\rm FWI},\omega,s) - u (\bm',\omega,s)\right), 
      \label{gradALT1}  \end{align}
 and     $\tau$ is given  by 
    \begin{align} \label{tau_def}
      \tau(\bm^{\rm FWI}, \omega,s, \bm') := (1,0)\,  \mathscr{S}_{\bm^{\rm FWI},\omega} \cG_{\bm^{\rm FWI}, \omega}(\rho u),      \end{align} where  $u = u(\bm^{\rm FWI}, \omega,s)$ {and the function $\rho = \rho(\cP,\alpha, \bm')$ is given by}    \eqref{rho_fn},  \eqref{rho}. {(For the meaning of the notation in \eqref{tau_def}, recall Remark  \ref{rem:firstarg} (i).) }
   \end{theorem}

     \begin{notation}\label{not:conv}                                                            To simplify notation,             in several  proofs we assume only one training model $\bm'$, one source $s$ and one frequency $\omega$, in which case we drop the summations over these variables. In this case we also omit the appearance of $s, \omega$ in the lists of independent variables. \end{notation}

\begin{proof}
Adopting the convention in Notation \ref{not:conv},
  our  first step is to differentiate (\ref{SOobj}) with respect to each  $p_{j,\ell}$,  to obtain, recalling $\bm \in \mathbb{R}^M_+$ {and that $\langle \cdot, \cdot \rangle$ denotes the Euclidean inner product,}  
\begin{align}
\frac{\partial \psi}{\partial p_{j,\ell}}(\cP, \alpha) = -  \left \langle    \frac{\partial \bm^{\rm FWI}}{\partial p_{j,\ell}},  \bm'- \bm^{\rm FWI} \right \rangle, 
\label{sogradstep1}
\end{align}
where, {as in \eqref{eq:abbrev}}, $\bm^{\rm FWI}$ is an abbreviation for  $\bm^{\rm FWI}(\cP, \alpha, \bm')$.
 To find an expression for  the first argument in the inner product in \eqref{sogradstep1}, we differentiate  \eqref{single} with respect to
 $p_{j,\ell}$ and use the chain rule  to obtain
\begin{align}
  &H \frac{\partial \bm^{\rm FWI}}{\partial p_{j,\ell}} =- \frac{\partial \nabla \phi}{\partial p_{j,\ell}},
    \label{dmdpinter}
\end{align}
where, to improve readability {and, analogous to the shorthand notations adopted before,}  we have avoided explicitly writing  the dependent variables of
$\nabla \phi = \nabla \phi(\bm^{\rm FWI}, \cP, \alpha, \bm')$ and  $ H = H(\bm^{\rm FWI}, \cP, \alpha)$.   

  Then, combining \eqref{sogradstep1} and \eqref{dmdpinter} and using the symmetry of $H$ and the definition of $\brho$  in \eqref{rho}, we obtain
\begin{align}
  \frac{\partial \psi}{\partial p_{j,\ell}}(\cP, \alpha)  \ & = \   \left \langle
                                                        \frac{\partial \nabla \phi}{\partial p_{j,l}},  \brho
                                                        \right \rangle
  \ =\  \sum_{k=1}^M \rho_k    \left(\frac{\partial^2 \phi}{\partial p_{j,\ell} \, \partial m_k}\right), 
\label{sogradstep2}
\end{align}
where we used the fact that $\brho = \brho(\cP, \alpha, \bm')\in \mathbb{R}^M$. Recall also that $\partial^2 \phi / \partial p_{j,\ell} \partial m_k$ is evaluated at $(\bm^{\rm FWI}, \cP,\alpha,\bm')$.   

Then, to simplify \eqref{sogradstep2}, we differentiate    \eqref{graddirect_here} with respect to $p_{j,\ell}$ and then  use    \eqref{I1} to obtain, for any $\bm, \cP, \alpha$ (and recalling Notation \eqref{not:conv}), 
 \begin{align}
\hspace{-2cm}\left(\frac{\partial^2  \phi}{ \partial p_{j,\ell} \, \partial m_k}\right)(\bm, \cP, \alpha, \bm')&=- \Re\,  \frac{d  }{d p_{j,\ell}}
\left \langle \mathcal{R}(\cP)\frac{\partial u}{\partial m_k}(\bm)   , \boldsymbol{\varepsilon}(\bm,\cP,\bm')\right \rangle  
                                                                                                                 \nonumber \\
&=- \Re \, \frac{d }{d p_{j,\ell}}
\left(  \frac{\partial u}{\partial m_k}(\bm), \mathcal{R}(\cP)^*\boldsymbol{\varepsilon}(\bm,\cP, \bm')\right)_\Omega  \nonumber\\
&=- \Re\, 
                                                                                                                                              \left(  \frac{\partial u}{\partial m_k}(\bm), \frac{d}{d p_{j,\ell}}\bigg(\mathcal{R}(\cP)^*\boldsymbol{\varepsilon}(\bm,\cP,\bm')\bigg)   \right)_\Omega   . \label{omeMgrad2}
 \end{align}
  Hence, evaluating \eqref{omeMgrad2} at $\bm = \bm^{\rm FWI}$,   combining this  with  \eqref{sogradstep2} and then using  \eqref{oneMgrad2b}, we have   
 \begin{align} \frac{\partial \psi}{\partial p_{j,\ell}} (\cP, \alpha)
   &= -\Re \left(\sum_k \rho_k(\cP,\alpha,\bm') \frac{\partial u}{\partial m_k}(\bm^{\rm FWI}), \frac{d}{d p_{j,\ell}}\bigg(\mathcal{R}(\cP)^*\boldsymbol{\varepsilon}(\bm^{\rm FWI},\cP, \bm') \bigg)   \right)_\Omega   \nonumber  \\
   &=- \Re   \left( (1,0) \, \mathscr{S}_{\bm^{\rm FWI},\omega} \cG_{\bm^{\rm FWI}, \omega} \bigg(\rho(\cP,\alpha,\bm') u (\bm^{\rm FWI})\bigg)    , \frac{d}{d p_{j,\ell}}\bigg(\mathcal{R}(\cP)^*\boldsymbol{\varepsilon}(\bm^{\rm FWI},\cP, \bm')\bigg)   \right)_\Omega.    \nonumber
 \end{align}
Now, using the definition of $\tau = \tau(\bm^{\rm FWI}, \bm')$ in \eqref{tau_def} we obtain   
 \begin{align}
\frac{\partial \psi}{\partial p_{j,\ell}} (\cP, \alpha)  &=- \Re  \left( \tau, \frac{d}{d p_{j,\ell}}
   \bigg(\mathcal{R}(\cP)^* \boldsymbol{\varepsilon}(\bm^{\rm FWI},\cP, \bm') \bigg)     \right)_\Omega.
                                                                           \label{oneMgrad3} 
  \end{align}

 To finish the proof, we 
 note that the operator $d/dp_{j,\ell}$,  appearing in \eqref{oneMgrad3},  denotes the {\em total}  derivative with respect to $p_{j,\ell}$.
Recalling   \eqref{resdbi} and \eqref{Rv}, we have   
\begin{align*}
\beps_{j'}(\bm^{\rm FWI},  \cP, \bm')&= u(\bm';p_{j'})-u(\bm^{\rm FWI};p_{j'}), \quad \text{for} \quad  j' = 1, \ldots, N_r. 
\end{align*}
Therefore, by \eqref{Rstar},
 \begin{align*}
 \cR(\cP)^*\beps(\bm^{\rm FWI}, \cP, \bm')= \sum_{j'=1}^{N_r} \left(u(\bm';p_{j'}) - u(\bm^{\rm FWI};p_{j'})  \right)\delta_{p_{j'}}.
 \end{align*}
 and thus 
 \begin{align*}  \frac{d}{d p_{j,l}}\left(\mathcal{R}(\cP)^*\beps(\bm^{\rm FWI},\cP, \bm')\right)
 & =                 
        \left(\frac{\partial u}{\partial x_l} (\bm';p_j) - \frac{\partial u}{\partial x_l} (\bm^{\rm FWI};p_j)\right)  \delta_{p_j}    \nonumber \\
  & \mbox{\hspace{1cm}} +  (u (\bm';p_j) -  u(\bm;p_j)) \frac{\partial}{\partial x_l} (\delta_{p_j}).
 \end{align*}
 Recalling Definition \ref{def:part_delta} and substituting this into \eqref{oneMgrad3} yields the result \eqref{gradALT}, \eqref{gradALT1} {(with $N_{\bm'} = 1$)}. 
\end{proof}

  \subsection{Derivative  of $\psi$ with respect to regularisation parameter  $\alpha$} 

 \begin{theorem}
   \label{sogradta}
   Provided $\mu$ is sufficiently large, $\partial \psi/\partial \alpha$ exists and is given by the formula 
\begin{align}
  \frac{\partial \psi}{\partial \alpha} (\cP, \alpha) = \frac{1}{N_{m'}}\sum_{\bm' \in \mathcal{M'}} (\bm^{\rm FWI})^\top
  \,
 {\mathsf{R}} \,    \bdelta 
 \label{sograda} \end{align}
where $\bm^{\rm FWI} = \bm^{\rm FWI}(\cP, \alpha, \bm')$ and $\bdelta = \bdelta(\bm^{\rm FWI},  \cP, \alpha, \bm')$ are as given  in Proposition  \ref{lem_sogradt} and \eqref{rho},
and {$\mathsf{R}$} is as in \eqref{ThisGamma}.
\end{theorem}

\noindent
 \begin{proof}
  The steps follow the proof of Theorem \ref{sogradt1}, but are simpler,  and again we assume only one 
    training model $\bm'$. 
First we  differentiate (\ref{SOobj}) with respect to $\alpha$ to obtain 
\begin{align}
\frac{\partial \psi(\cP, \alpha)}{\partial \alpha}= -  \left \langle    \frac{\partial  \bm^{\rm FWI}}{\partial \alpha},  \bm'- \bm^{\rm FWI} \right \rangle.
\label{sogradstep1a}
\end{align} 
Then we  differentiate \eqref{single} with respect to $\alpha$ to obtain, analogous to \eqref{dmdpinter}, 
\begin{align}
H \frac{\partial \bm^{\rm FWI}}{\partial \alpha} \ =\ -   \frac{\partial \nabla \phi}{\partial \alpha}. \label{sysa}
\end{align}
Differentiating  \eqref{graddirect_here} with respect to $\alpha$,    
{using \eqref{Gamma},}  and then substituting the {result} into   the right-hand side of \eqref{sysa}, we obtain    
 \begin{align}
   H \,
   \frac{\partial \bm^{\rm FWI}}{\partial  \alpha}  &\ =\ - {\mathsf{R}} \, \bm^{\rm FWI}. \label{dmda}
 \end{align}  
 Substituting (\ref{dmda}) into (\ref{sogradstep1a}) and recalling the definition of $\bdelta$ in \eqref{rho}   gives \eqref{sograda} {(with $N_{\bm'} = 1$)}.  \end{proof}

\medskip 
\noindent Algorithm \ref{alg:box} summarises the steps involved in computing the derivatives of the upper level objective function. {Here $\mathcal{M}'$ denotes the set of training models and $\mathcal{M}^{\rm FWI}$ denotes the set of their FWI reconstructions. }

                                                                                                      \begin{algorithm}[h!]
      \setstretch{1.15}                                                                             \caption{Derivative of $\psi$ with respect to $\alpha$ and  $p_{j,\ell}$  }
 \medskip
 
 \begin{algorithmic}[1]
   \State \textit{Inputs:} $\cP$, \ $\alpha$, \ $\mathcal{M'}$, \
  $\mathcal{M}^{\rm FWI}:= \{\bm^{\rm FWI}(\cP, \alpha, \bm'): \bm' \in \mathcal{M}'\}$ (lower level solutions), \ $j,\ell$

\State {\bf For each} $\bm' \in \mathcal{M'}$ (letting $\bm^{\rm FWI}$ denote $\bm^{\rm FWI}(\cP
  , \alpha, \bm')$):
\State \indent {\bf Solve} (\ref{rho}) for $\bdelta = \bdelta( \cP, \alpha, \bm')$

\State \indent  {\bf For each} $\omega \in \mathcal{W}$, $s \in \mathcal{S}$:  \hfill ($\star$) 

\State \indent \quad  {\bf Compute} $u(\bm^{\rm FWI}, \omega, s) = (1,0) \mathscr{S}_{\bm^{\rm FWI}, \omega} (\delta_s)   $; 
\State \indent \quad {\bf Compute} $\tau(\bm^{\rm FWI}, \omega, s,  \bm')$ by \eqref{tau_def}; 
\State \indent \quad {\bf Compute} $a_{\ell}(\bm^{\rm FWI}, \omega,s,\bm')$  by \eqref{gradALT1} and evaluate at $p_j$.
\State \indent {\bf End}
\State {\bf End}
\State  \textit{Output 1:} Compute $\partial \psi / \partial \alpha $ by \eqref{sograda}. 
\State \textit{Output 2: } Compute $\partial \psi / \partial p_{j,\ell}$  by (\ref{gradALT})
\end{algorithmic}
 \label{alg:box}
\end{algorithm}

\begin{remark}[Remarks on Algorithm \ref{alg:box}]\label{rem:box}
The computation of \textit{Output 1} does not require the inner loop over $\omega$ and $s$ marked ($\star$) {in Algorithm \ref{alg:box}}.

For each $s, \omega, \bm'$, Algorithm \ref{alg:box}  requires two 
Helmholtz  solves, one for $u$ and one for $\tau$. While $u$ would already be available as the wavefield arising in  the lower level problem, the computation of $\tau$  involves data determined by  $\rho u$, where $\rho = \rho(\cP, \alpha, \bm')$
     is given by \eqref{rho} and \eqref{rho_fn}.   
                                                                                                       
 The system  \eqref{rho}, which has to be solved for $\bdelta$, has system matrix $H(\bm^{\rm FWI}(\cP, \alpha, \bm'), \cP, \alpha)$, which is real symmetric. As is shown in Discussion \ref{disc:Hess},  
  matrix-vector multiplications with $H$ can be done very efficiently  and  so we solve   \eqref{rho}
  using an iterative method.
  Although positive definiteness of  $H$ is only guaranteed for $\mu$ sufficiently large
 (and such $\mu$ is not in general known), here we used  the preconditioned conjugate gradient method and found it to be very effective in all cases. Details are given in Section \ref{sect:PCG}.  

  Analogous Hessian systems arise in the application of the  truncated Newton method
  for the lower level problem (i.e., FWI).   In \cite[Section 4.4]{metivier2017full} the conjugate gradient method was also applied to solve these, although this was replaced by  Newton or steepest descent directions if
  the Hessian became indefinite. 
\end{remark}

\subsection{Complexity Analysis in Terms of the Number of PDE Solves}
\label{sec:complexity}

  To  assess the complexity of the {proposed} bilevel algorithm  in terms of the number of Helmholz solves needed (arguably the most computationally intensive part of the algorithm), we  introduce the  notation:
\begin{itemize}
\item $N_{\rm upper}$ = number of  iterations needed to solve the upper level optimisation problem. 
\item $N_{\rm lower}$ = average number of  iterations needed to solve the lower level (FWI) problem.  (Since the number needed will vary as the upper level iteration progresses,  we work with the average here.) 
\item $N_{\rm CG}$ = average number of conjugate gradient iterations used to solve (\ref{rho})
\item $N_{\rm data} := N_s * N_{\omega} * N_{m'}$ = the product of the number of sources, the number of frequencies and the number of training models =  the total amount of data used in the
  algorithm.  
\end{itemize}

\par The total cost of solving the bilevel problem may then be broken down as follows.  \begin{itemize}
\item[A] {\bf Cost of computing $\{\bm^{\rm FWI}(\cP,\alpha,\bm'): \bm' \in \mathcal{M}'\}$.} \ Each iteration of FWI requires two Helmholtz solves for each $s$ and $\omega$ (see Section  \ref{subsec:gradphi})
 and this is repeated for each $\bm'$, so the total cost is     $2N_{\rm lower}N_{\rm data} $ Helmholtz solves. 
\item[B]  {\bf Cost of updating $\cP, \alpha$}.  \ To solve the systems \eqref{rho} for each $s, \omega$ and $\bm'$ via the  conjugate gradient method we need, in principle, to do four Helmholtz solves for each matrix-vector product with $H$ (see Section \ref{subsec:MVHess}).
  However two of these ($u$ and the adjoint solution $\lambda$ defined by \eqref{lambda}) have already been counted in Point A.
  So the total number of solves needed by the CG method is  $2N_{\rm CG}N_sN_{\omega}N_{\bm'} = 2N_{\rm CG}N_{\rm data}$. 
  After this has been done, for  each  $s, \omega, \bm'$  one more Helmholtz solve is needed to compute \eqref{tau_def}. So the total cost of one update to $\cP, \alpha$ is $(2 N_{\rm CG} + 1)N_{\rm data}$ 
\end{itemize}
Summarising, we obtain the following result. 

\begin{theorem} \label{thm:cost}  The total cost of solving the bilevel sensor optimisation problem with a gradient-based optimisation method in terms of the number of PDE solves is 
\begin{align*}
\mathrm{Number \ of \ Helmholtz  \ Solves}= N_{\rm upper}\left( 2N_{\rm lower} + 2N_{\rm CG} + 1   \right)N_{\rm data} .
\end{align*}
\end{theorem}
  When using line search with the gradient-based optimisation method (as we do in practice),
  there is an additional cost factor of the number of line search iterations, but we do not include that here.
  
  While the cost reported in Theorem \ref{thm:cost} could be quite substantial, we note that it is completely independent of the number of parameters (in this case sensor coordinates and  regularisation parameters) that we choose to optimise. Also, as we see in Section \ref{sect:par}, the algorithm is highly parallelisable over training models, and experiments suggest that in a suitable parallel environment the factor  $N_{m'}$ will not appear in $N_{\rm data}$.

  \section{Application to a Marmousi problem}
\label{sec:Numerical}


{In this section we evaluate the performance of our bilevel  algorithm by applying it  to a variant
  of the Marmousi model on a rectangle in $\mathbb{R}^2$.
  Figure \ref{fig:bilevel} summarises the structure of our algorithm.
To give more detail, for each model $\bm'$ in  a pre-chosen set of training models $\mathcal{M}'$, initial guesses $\cP_0$ and $\alpha_0$ for the design parameters 
$\cP$ and $\alpha$, respectively,  are
  fed into   the lower-level optimisation  problem. This lower-level problem  is solved by a quasi-Newton method (see Section \ref{subsec:Quasi} for more detail), stopping when 
  the norm of the gradient of the lower level objective function
  is sufficiently small. This yields  {a new set of} models
  $\mathcal{M}^{\rm FWI} := \{\bm^{\rm FWI}(\cP_0, \alpha_0, \bm'): \bm'\in\mathcal{M}' \}$ .
The  set  $\mathcal{M}^{\rm FWI}$ is   then an input for the upper level optimsation problem, which in turn yields a new $\cP$ and $\alpha$ to be inputted as design parameters in the lower level problem. This
  process is iterated to eventually  obtain  $\alpha_{\rm opt}, \cP_{\rm opt}$ and corresponding reconstructions of the training models.}   
{More details about the numerical implementation   
are given in the Appendix (Section \ref{app:Details}).}

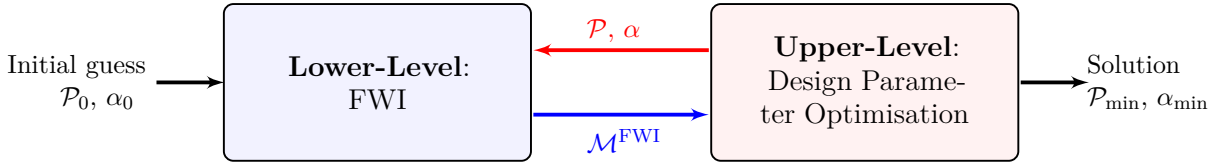
\begin{figure}[htp]
\scalebox{0.9}{
\begin{tikzpicture}[scale=0.95]
\node[text width=3.5cm] (ic) at (-4.6,4) { \normalsize Initial guess  \\ \normalsize $\qquad\cP_0$, $\alpha_0$}; 
\path [line, ultra thick] (-4.15,4) --  (-3.1,4);
\node [blocklargeh] (lower) at (-0.75,4) {\large \textbf{Lower-Level}:\\
FWI\\
};
\node [blocklargeredh] (upper) at (6.75,4) {\large \textbf{Upper-Level}:\\
Design Parameter Optimisation\\ 
};
\path [line, blue, ultra thick] (1.65,3.5) --  (4.35,3.5);
\node [blue,text width=1cm] at (3,3.1) {$\mathcal{M}^{\rm FWI}$}; 
\path [line, red, ultra thick] (4.35,4.5) -- (1.65,4.5);
\node [red,text width=1cm] at (3,4.8) {$\cP,\, \alpha$}; 
\path [line, ultra thick] (9.15,4) -- (10.2,4);
\node[text width=1cm] at (10.7,4) {\normalsize Solution\\ \normalsize $   \cP_{\min}, \, \alpha_{\min}$}; 
\end{tikzpicture}
}
\caption{Overall Schematic of the Bilevel Problem. \label{fig:bilevel}} 
\end{figure}
\noindent



{Since we want to assess the performance of our sensor positions and regularisation weight optimisation algorithm,  rather than the choice of regularisation, in the lower level problem we only consider FWI equipped with Tikhonov regularisation, with regularisation term as in \eqref{Gamma}. Since it is well-known that such kind of regularisation tends to oversmooth model discontinuities, the Marmousi model adopted here has been slightly smoothed by applying a Gaussian filter  horizontally and vertically (see Figure \ref{fig:Exp2GT}). We stress, however, that the algorithm we propose could also be combined with other regularisation techniques more suitable for non-smooth problems, such as total variation and total generalised variation, as used, for example, in \cite{aghamiry2021full, esser2018total, tgv}.}
\begin{figure}[h!]
\centering
\includegraphics[scale=1]{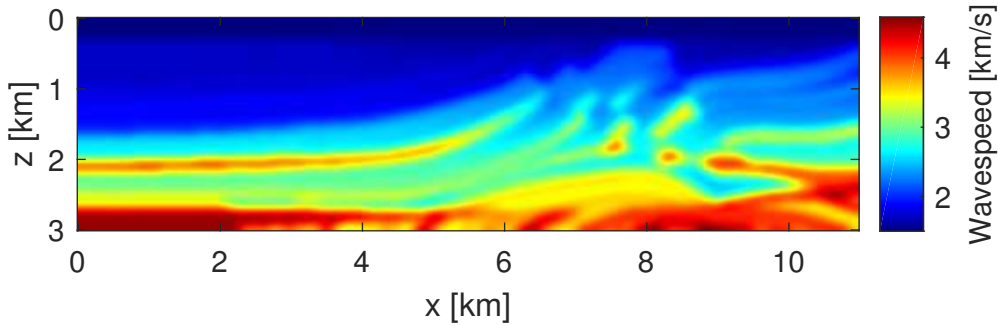}
\caption{Smooth Marmousi model. \label{fig:Exp2GT}}
\end{figure} 

{We discretised the model in Figure \ref{fig:Exp2GT} (which is 11km wide and 3km deep)   using  a $440 \times 121$ rectangular grid, yielding a
grid spacing of 25m in both the $x$ (horizontal) and $z$ (vertical) directions. Analogous to the experimental setup adopted in \cite[Section 4.3]{haber2003learning} and \cite[Section 5.1]{haber2008numerical}, we split this  horizontally into five slices of equal size (each with 10648 model parameters); see the first row of Figure \ref{fig:reconstr}. In our  experiments  we shall choose four of these slices as training models
  for the bilevel optimisation and reserve the fifth slice for testing the optimised parameters. Thus the bilevel algorithm will be implemented on a rectangular domain of width 2.2km and depth 3km, so that  $M= 10648$ is the dimension of the model space, as in \eqref{disc_to_cont}.} 

{On this domain  we   consider a cross-well  transmission set-up, where  sources are positioned uniformly along a vertical line (borehole) near its  left-hand edge and sensors are placed (iniially randomly) along a vertical bore-hole near the right-hand edge.  
  The positions of the sensors  are  optimised, with the constraint that they should always
  lie within
    the  borehole.  Here we give results only for the case of 5 sources and sensors.
    In the case of 5 fixed sources,  their positions and those of the sensors (before and after optimisation)   are illustrated   in
  Figure \ref{fig:all5}.  Here their positions are imposed on  Slice 2 of the Marmousi model.} 



{For each training model $\bm'$, the initial guess  for
  computing $\bm^{\rm FWI} (\cP_0,\alpha_0,\bm')$   in the lower-level problem  was taken to be $m = 1/c^2$, where $c$ is a smoothly vertically varying wavespeed, increasing with depth, horizontally constant,  and containing no information on any of the structures in the training models.  

In practice the algorithm in Figure \ref{fig:bilevel} was implemented within the  
  framework of bilevel frequency continuation (as described in Algorithm \ref{allg:fcb1} of Section \ref{app:Details}).
   Here the frequencies used were 0.5Hz, 1.5Hz, 3Hz, 6Hz, split into groups [(0.5), (0.5, 1.5), (1.5, 3), (3,6)]. The choice of increments and overlap of groups   was  motivated by results reported  in\cite{sirgue2004efficient}, \cite{BrOpVi:09}.}
  
  {In the cases where we have
    optimised  the regularisation weight $\alpha$, its initial value ($\alpha=10$) was kept fixed in the first three frequency groups of the frequency continuation scheme, and its value was optimised only in the final frequency group.}
  
  {At each iteration of the bilevel algorithm, the lower-level problem was  solved to a tolerance of $\|\nabla \phi\|_2 \leq 10^{-10}$. Preconditioned CG,  with preconditioner given by \eqref{P2} was used to solve the Hessian system \eqref{rho} (required for  the upper-level gradient computation), with initial guess
  chosen as the vector,  each entry of which is $1$; the CG iterations were terminated when a Euclidean norm relative residual reduction of  $10^{-15}$ was achieved. In the optimisation at the upper level, for each frequency group,  the iterations were terminated when  any one of the following conditions was satisfied:  (i) the infinity norm of the 
  gradient (projected onto the feasible design space) 
  was smaller than $ 10^{-10}$,   (ii) the updates to $\psi$ or the optimisation parameters stalled or (iii) 50 iterations  were reached.
  Note that condition (i) is similar to the condition proposed in  \cite[Equation (6.1)]{byrd1995limited}.}

{    To compare the    reconstructions of a given model numerically,   we computed several statistics as follows:
    For the $j$th parameter of a given model, its  relative percentage error (RE(j)) and the  mean 
      mean of these values  (MRE)  are  defined by: 
  \begin{align}
\text{MRE} := M^{-1} \sum_{j=1}^M \text{RE}(j), \quad \text{where} \quad \text{RE}(j) := \left| \frac{\text{Reconstruction}(j) - {\text{Ground Truth}(j)} }{{\text{Ground Truth}(j)}}  \right| \times 100.   
\label{relerr}
  \end{align} }

  {  The  {\em Structural Similarity Index (SSIM)} between the ground truths and the optimised reconstructions. This is a quality metric commonly used in imaging \cite[Section III B]{wang2004image}. A good similarity between images is indicated by values of SSIM that are close to 1.
    The SSIM values were computed using the \texttt{ssim} 
function in Matlab's Image Processing Toolbox.} 

{  Finally the Improvement Factor is defined to be the improvement in the value of the objective
  function $\psi$ of the upper level problem after optimisation, i.e.,
  $$ \text{IF}  = \frac{\psi(\alpha_0,\cP_0)}{\psi(\alpha_{\rm opt},\cP_{\rm opt})}, $$
    where $\alpha_{\rm opt}$ and $\cP_{\rm opt}$ are the optimised design parameters. (The scenarios below include the cases where either $\alpha$ or $\cP$ is  optimised, or both.) }

\subsection{Benefit of jointly optimising over sources' locations and weighting parameter.}

{We start by showing the benefit of optimising over both source locations and regularisation weight.
  To do this, we  compare the results so obtained with those  obtained by optimising solely over the source locations and solely over the regularisation parameter.

  We  ran our bilevel learning algorithm using slices 1, 2, 3, and 5 as training models, and 
  slice 4 for testing.
  When testing, we added   $1\%$ Gaussian white noise {(corresponding to 40dB SNR)} to the synthetic data.
}

{  In Table \ref{tab:jointOpt},  we report the values of  MRE, SSIM 
and IF   for the
reconstructions using  (i) the  unoptimised design  parameters,  together with those obtained by (ii) optimising with respect to $\alpha$ with  $\cP$  fixed and randomly placed in the well,  
(iii) optimising with respect to  $\cP$ keeping  $\alpha = 10$ fixed;  and finally (iv) optimising  $\cP$ and  $\alpha$  simultaneously. The random choice of sensors in  (ii) is motivated by the evidence {\em a priori}  that sensor placement benefits from random positioning, see \cite{hennenfent2008simply}. } 

 \begin{table}[h!]
 \footnotesize
 \begin{center}
\begin{tabular}{| c | c | c |c|c | c | c | c | c| c | c | c|} 
 \hline
slice & \multicolumn{2}{c|}{(i) unoptimised} & \multicolumn{3}{c|}{(ii) $\alpha$-optimised} & \multicolumn{3}{c|}{(iii) $\mathcal{P}$-optimised} &  \multicolumn{3}{c|}{(iv) $\mathcal{P},\alpha$-optimised}\\\hline
 & MRE & SSIM & MRE & SSIM & IF & MRE & SSIM &IF& MRE & SSIM & IF \\\hline
  1 & 4.06 & 0.82 & 3.86 & 0.82 &1.06 & 3.02 & 0.87 & 4.03 &   2.45  & 0.89 & 4.92 \\\hline
 2 &  4.45 & 0.76 & 4.28  &0.77 & 1.03  & 3.30 & 0.81 &3.40    &  3.03 & 0.83 &3.87 \\\hline
 3 & 4.73  & 0.78  & 4.58 & 0.79 & 1.03 &  2.77  &0.85  & 5.42 &  2.45 &0.86 & 6.55 \\\hline
 \textbf{4} & \textbf{7.37} & \textbf{0.67}  & \textbf{7.01}  & \textbf{0.68}  & \textbf{1.22} &  \textbf{5.60}  & \textbf{0.72}  & \textbf{4.77}  & \textbf{4.92} & \textbf{0.76} & \textbf{7.56} \\\hline
 5 & 7.31& 0.68 &5.75  &0.72  & 2.55  &  3.78 &  0.84 & 13.80 &  3.51 &0.86 &18.33\\\hline
\end{tabular}
\end{center}
\caption{{Values of Mean Relative percentage error (MRE), SSIM and IF  for  different choices of the FWI design parameters. Slices 1, 2, 3 and 5 were used for training, while slice 4 was used for testing and is highlighted in bold face}.}\label{tab:jointOpt}
\end{table}
{
The results show steady  improvement of all accuracy indicators  as we proceed through the options (i)-(iv). 
In particular we note the strong superiority of strategy (iv). Here the optimised design parameters obtained by learning on slices 1,2,3 and 5 yield an imrovement factor of more than 7 when tested on  slice 4 (even though in this case 1\% white noise was added).    

{
We highlight two points regarding Strategies (iii) and (iv). First, given $\mathcal{P}$, it is relatively cheap to optimise with respect to $\alpha$ -- see Remark \ref{rem:box} (i.e., Strategy (iv) has essentially the same computational cost as Strategy (iii)). Second, since $\alpha$ is tailored to the fit-to-data, which in turn depends on $\mathcal{P}$, it makes sense to change $\alpha$ when $\mathcal{P}$ is changed.
}


In Figure \ref{fig:reconstr}, we display, for each slice,  in row 1: the ground truth, in row 2:  the reconstructions of the training
and testing models, using  the  unoptimised parameters  (random sensor positions and $\alpha = 10$) and finally in row 3:    the reconstructions using $\cP, \alpha$  optimised design parameters.}

\begin{figure}[h!]
\begin{tabular}{cccccc}
 \hspace{-0.7cm}  & 
\hspace{-0.7cm} slice 1 (train) & 
\hspace{-1.1cm} slice 2 (train) & 
\hspace{-1.1cm} slice 3 (train) & 
\hspace{-1.1cm} slice 4 (test) & 
\hspace{-1.1cm} slice 5 (train)\\
 \hspace{-0.7cm}  \parbox[t]{2mm}{\multirow{3}{*}[7em]{\rotatebox[origin=c]{90}{exact}}}  &
 \hspace{-0.45cm} \includegraphics[scale=0.37]{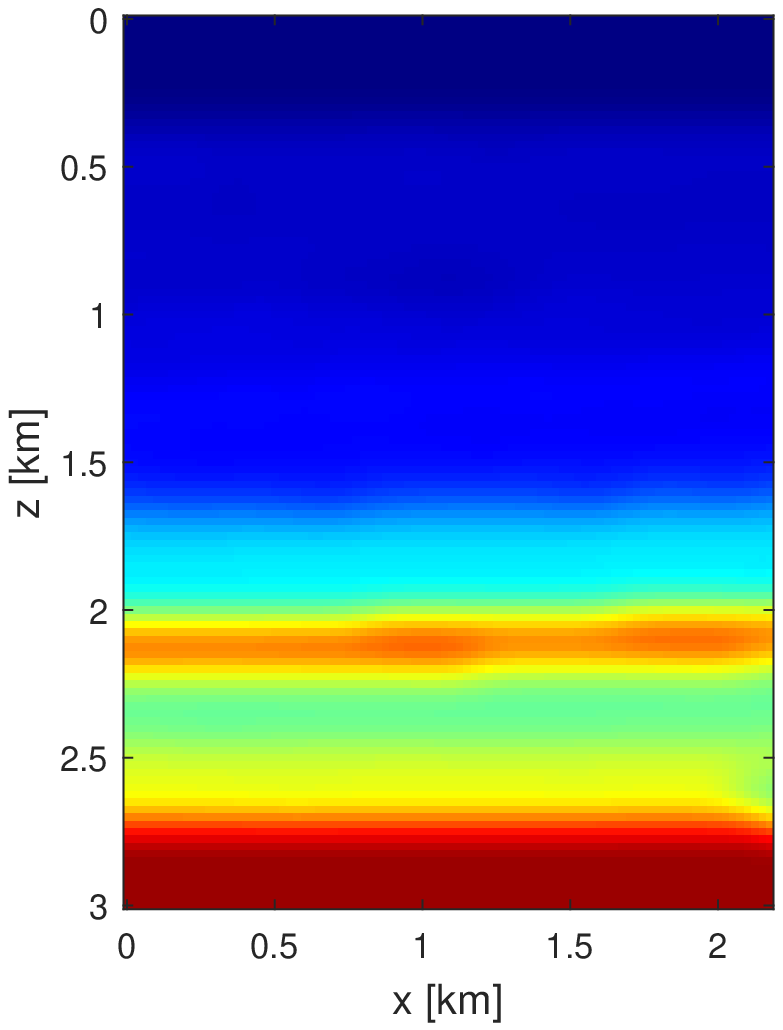} &
 \hspace{-0.7cm} \includegraphics[scale=0.37]{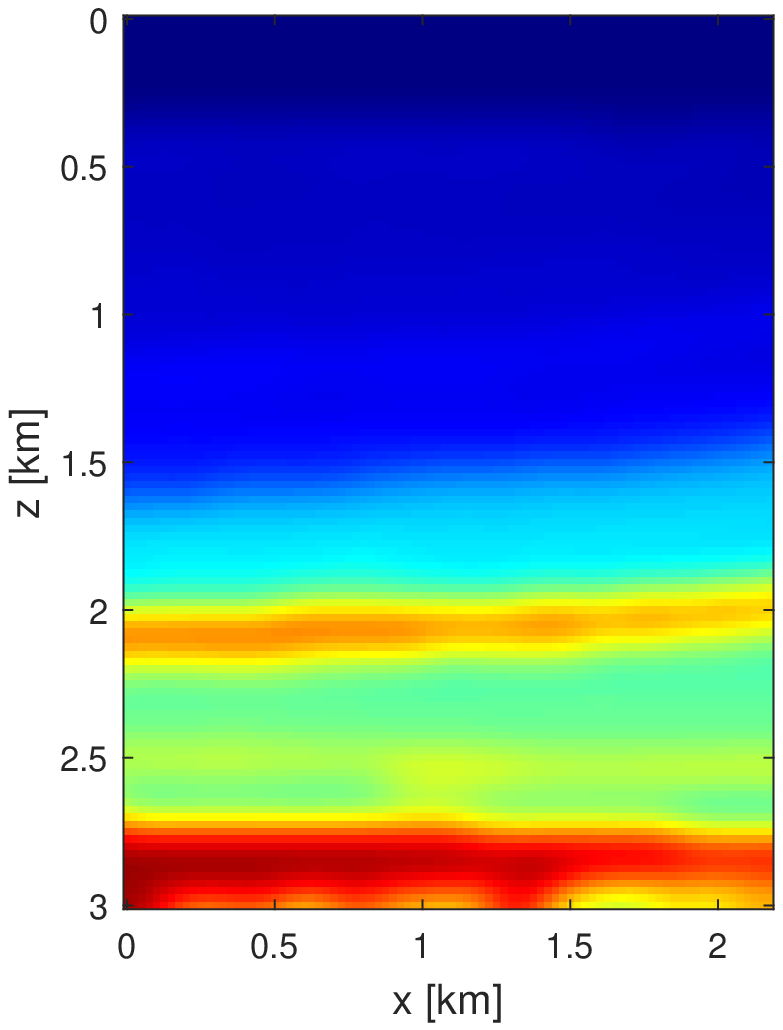} &
 \hspace{-0.7cm} \includegraphics[scale=0.37]{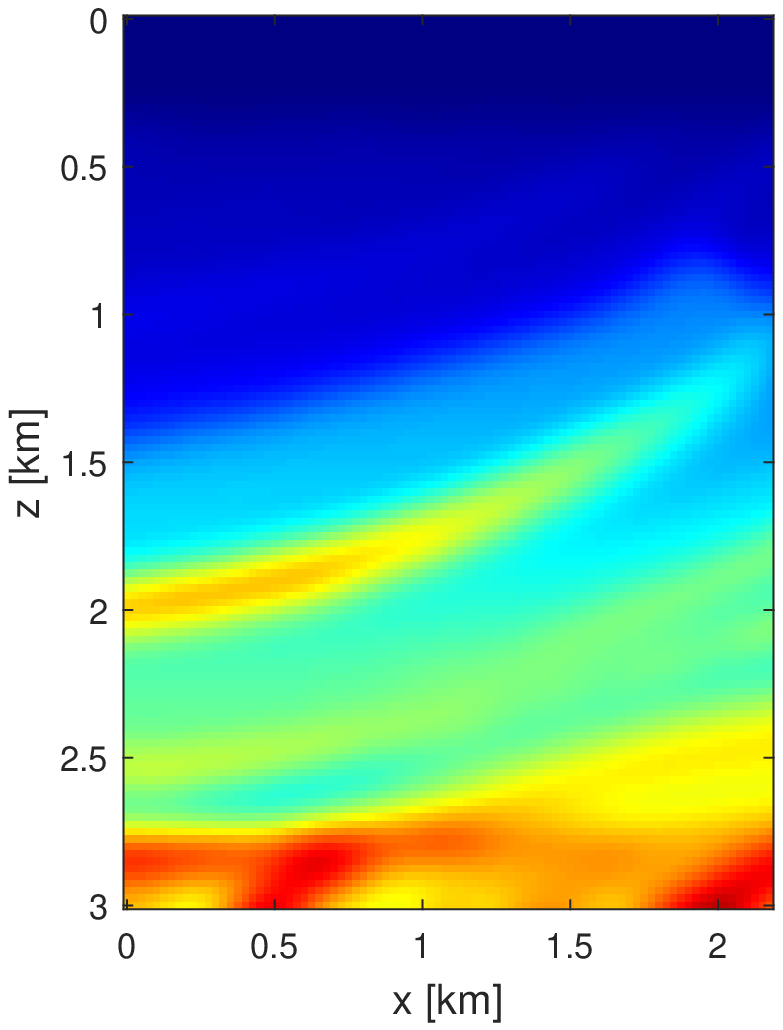} &
 \hspace{-0.7cm} \includegraphics[scale=0.37]{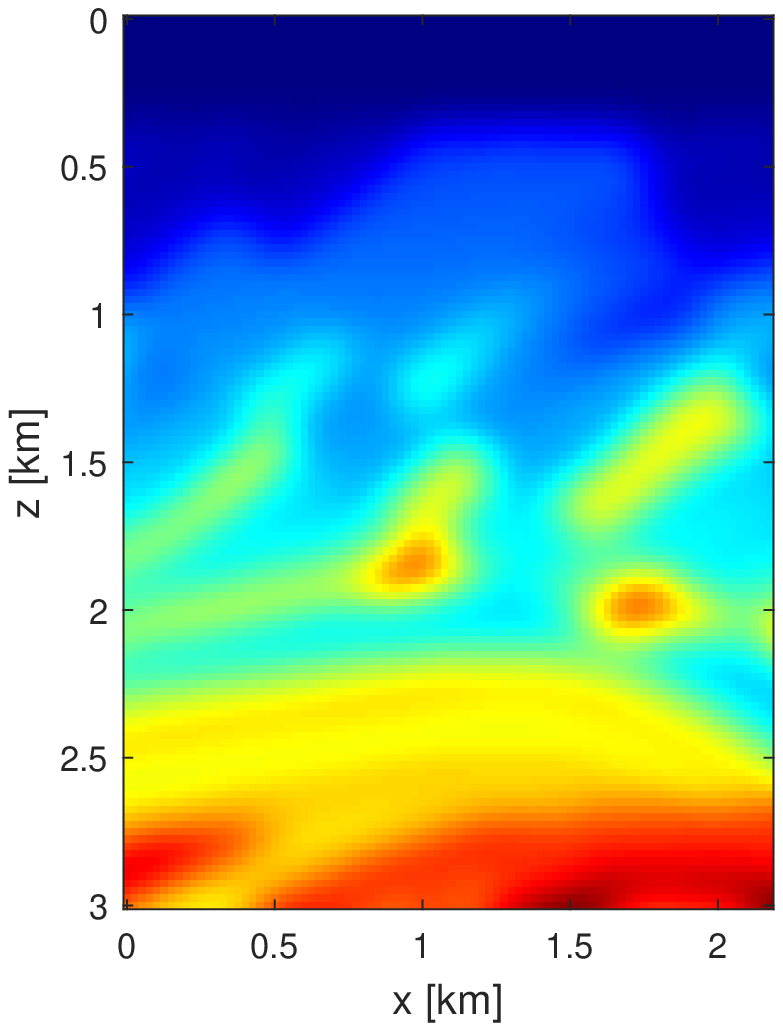} &
 \hspace{-0.7cm} \includegraphics[scale=0.37]{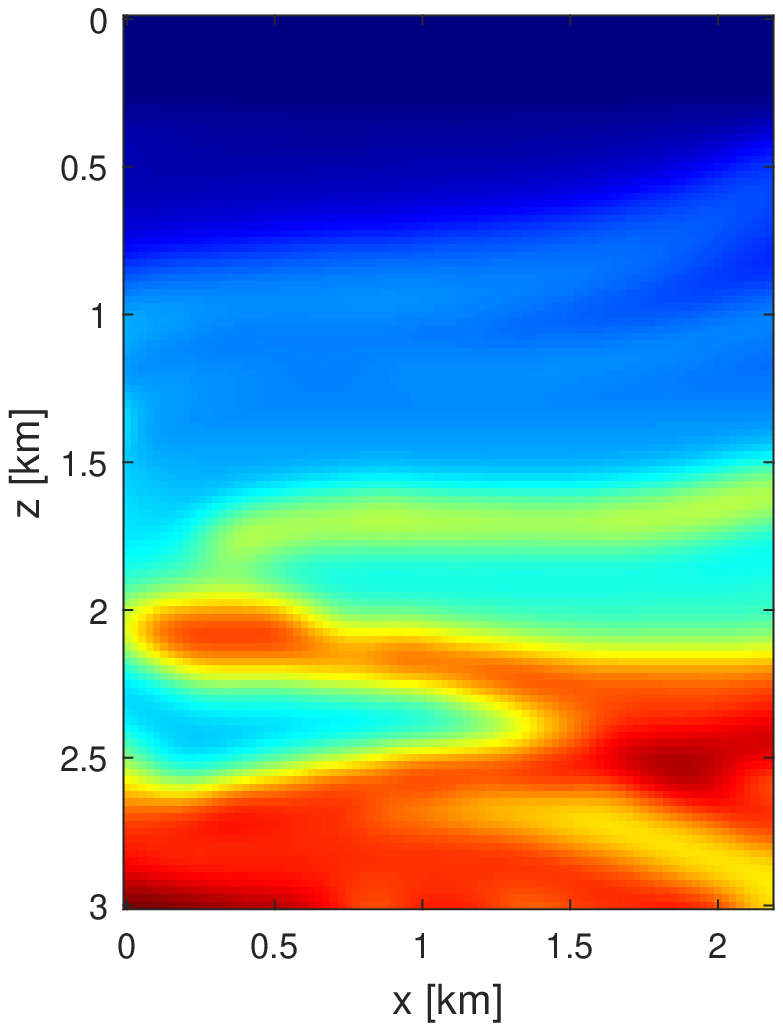}\\
 \hspace{-0.7cm}  \parbox[t]{2mm}{\multirow{3}{*}[8em]{\rotatebox[origin=c]{90}{unoptimised}}}  &
 \hspace{-0.7cm} \includegraphics[scale=0.28]{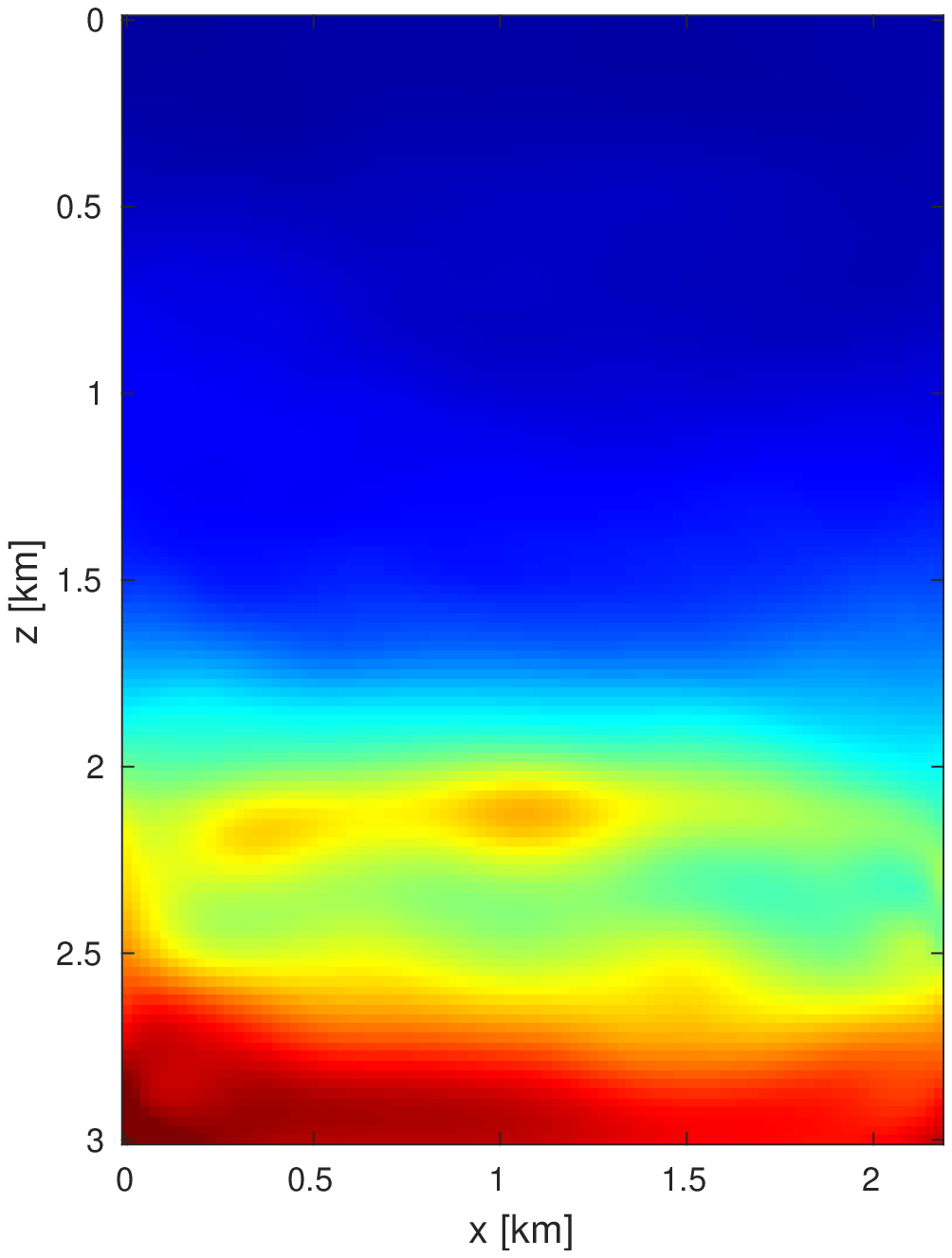} &
 \hspace{-1.0cm} \includegraphics[scale=0.28]{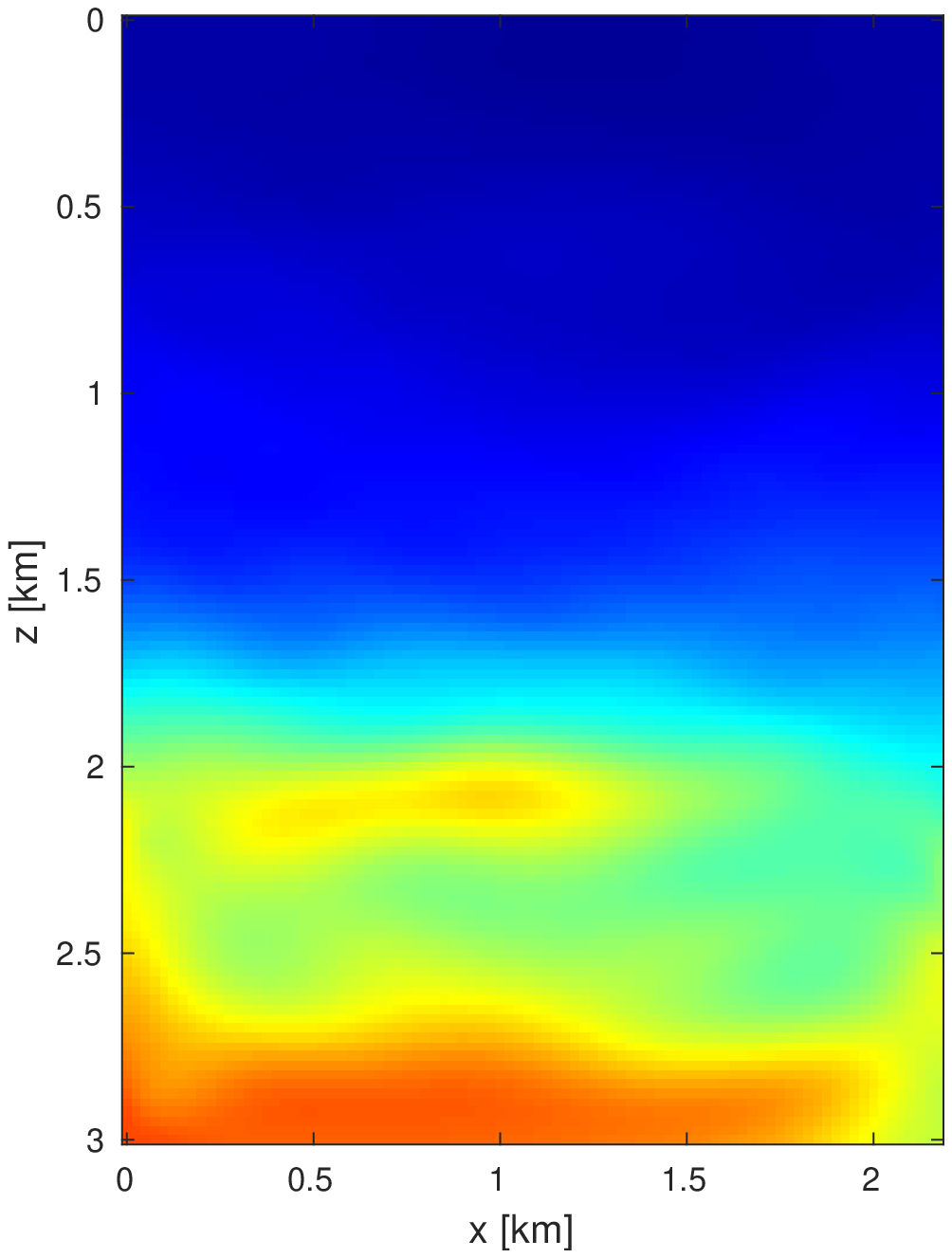} &
 \hspace{-1.0cm} \includegraphics[scale=0.28]{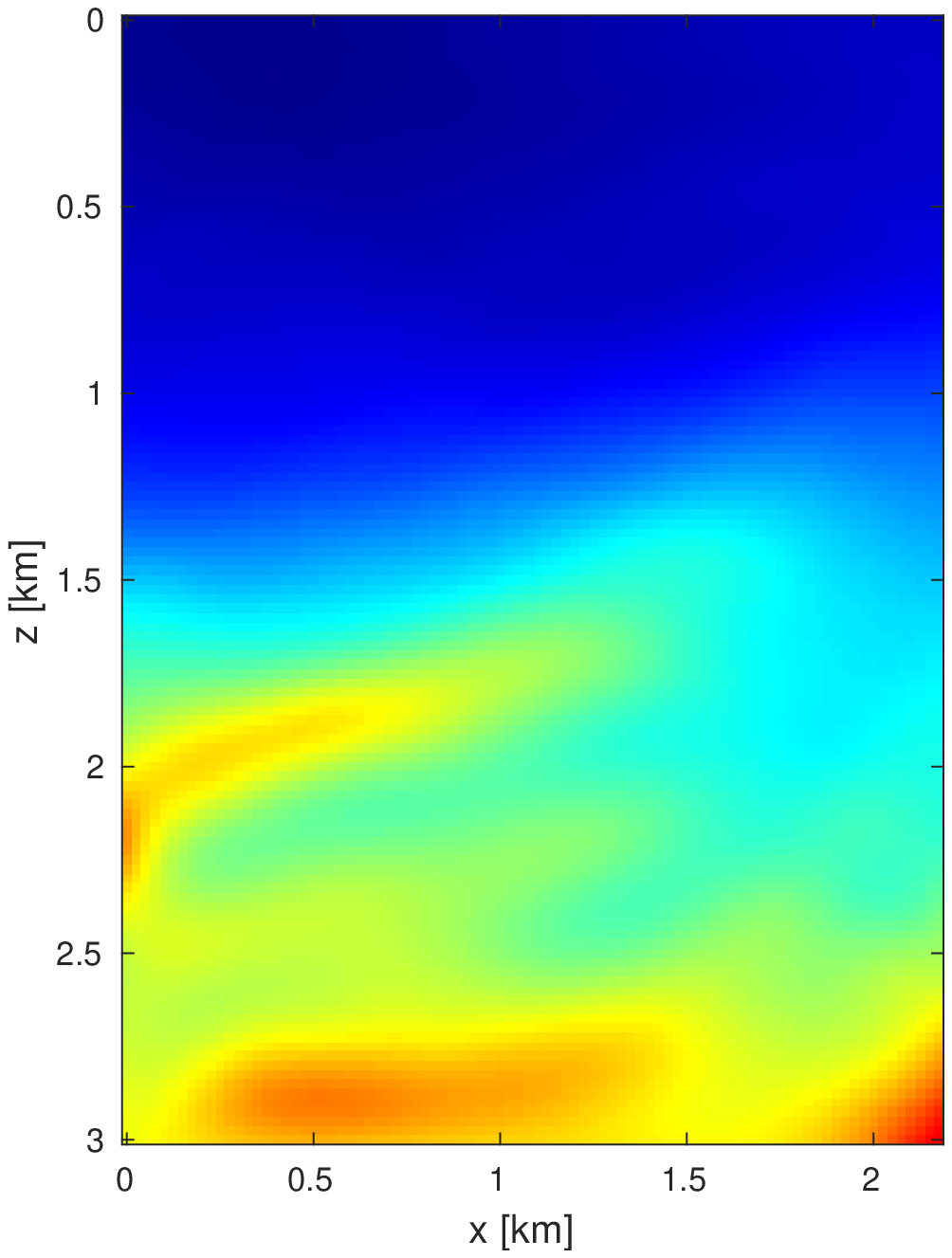} &
 \hspace{-1.0cm} \includegraphics[scale=0.28]{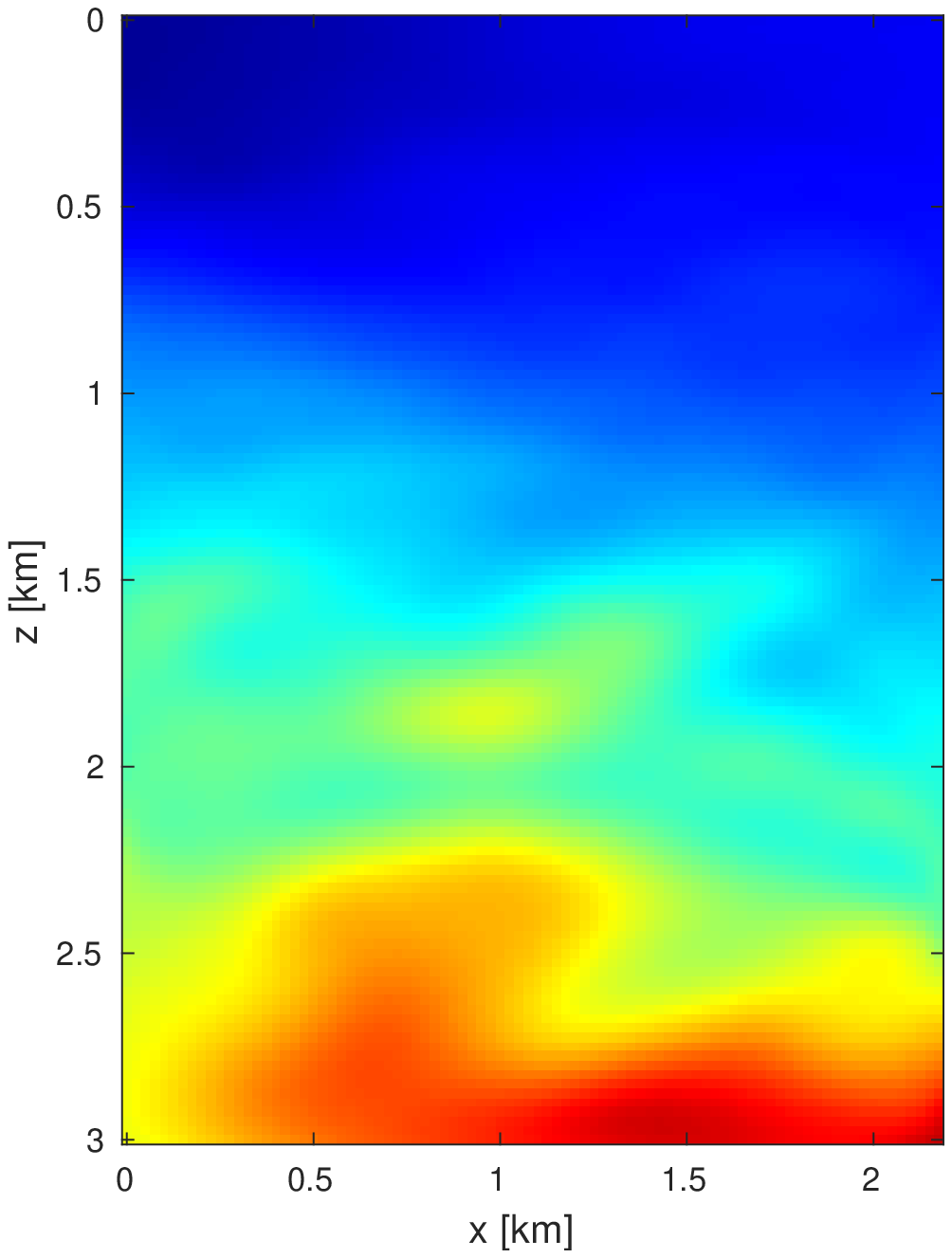} &
 \hspace{-1.0cm} \includegraphics[scale=0.28]{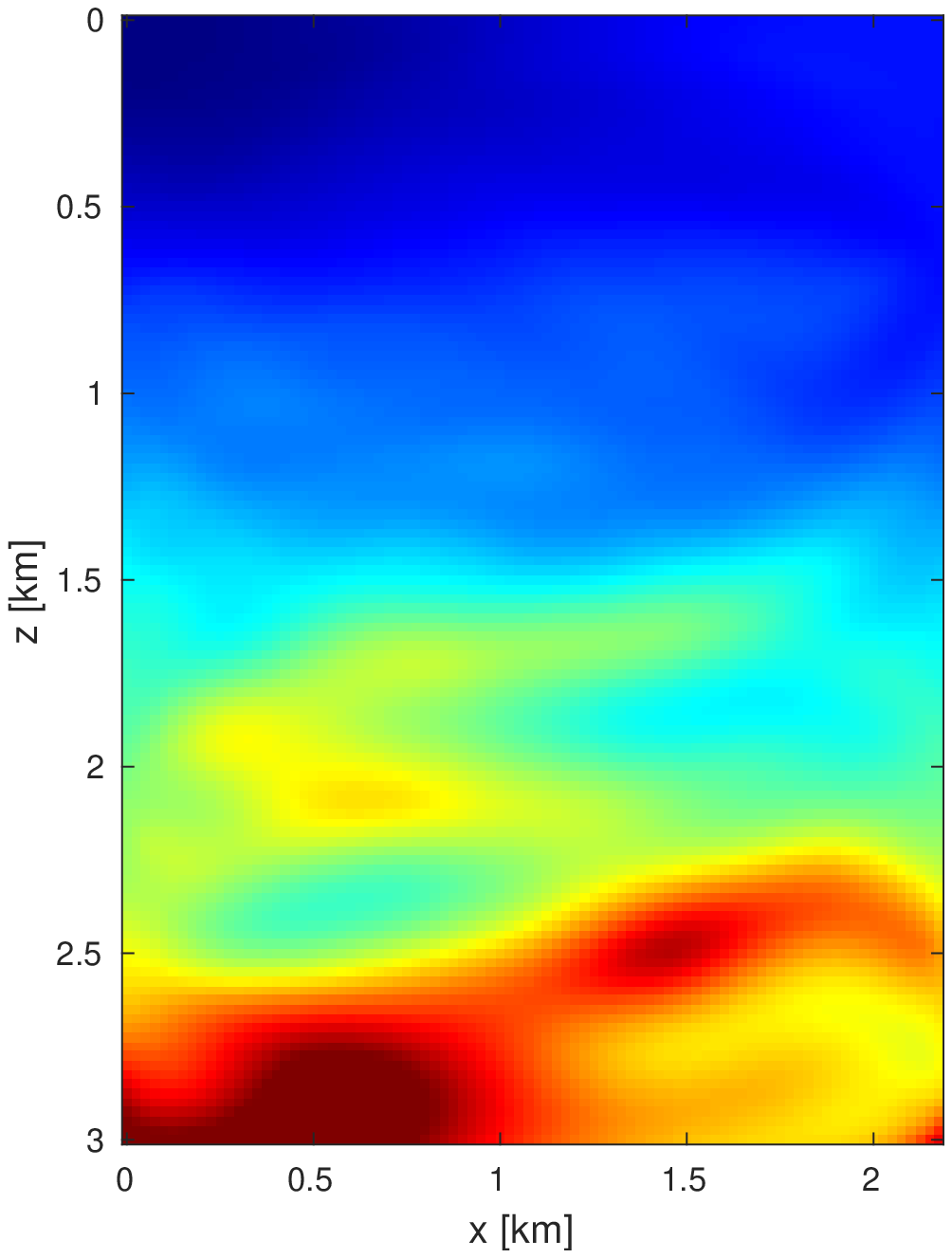}\\
  \hspace{-0.7cm}  \parbox[t]{2mm}{\multirow{3}{*}[8em]{\rotatebox[origin=c]{90}{$\mathcal{P},\alpha$-optimised}}}  &
 \hspace{-0.7cm} \includegraphics[scale=0.28]{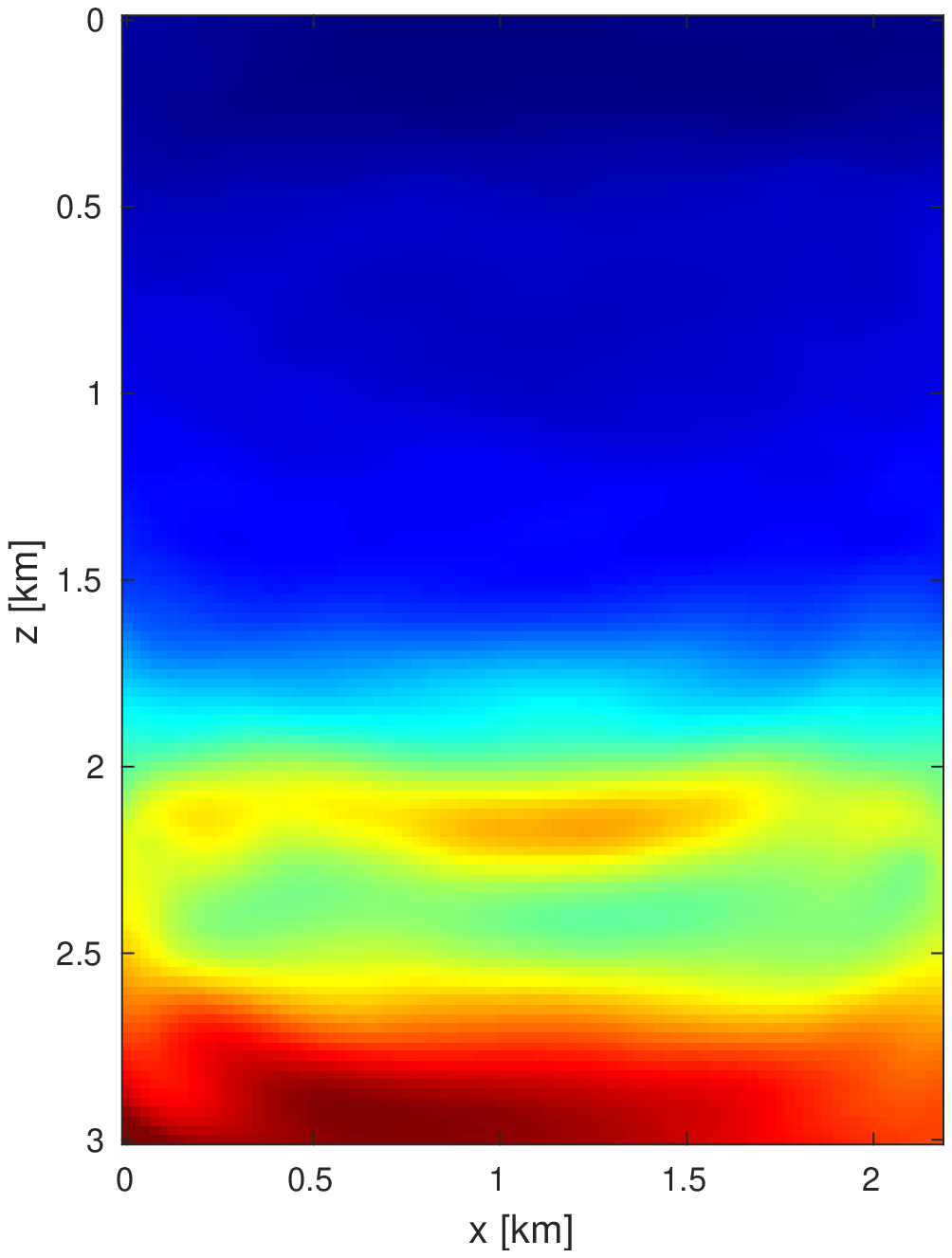} &
 \hspace{-1.0cm} \includegraphics[scale=0.28]{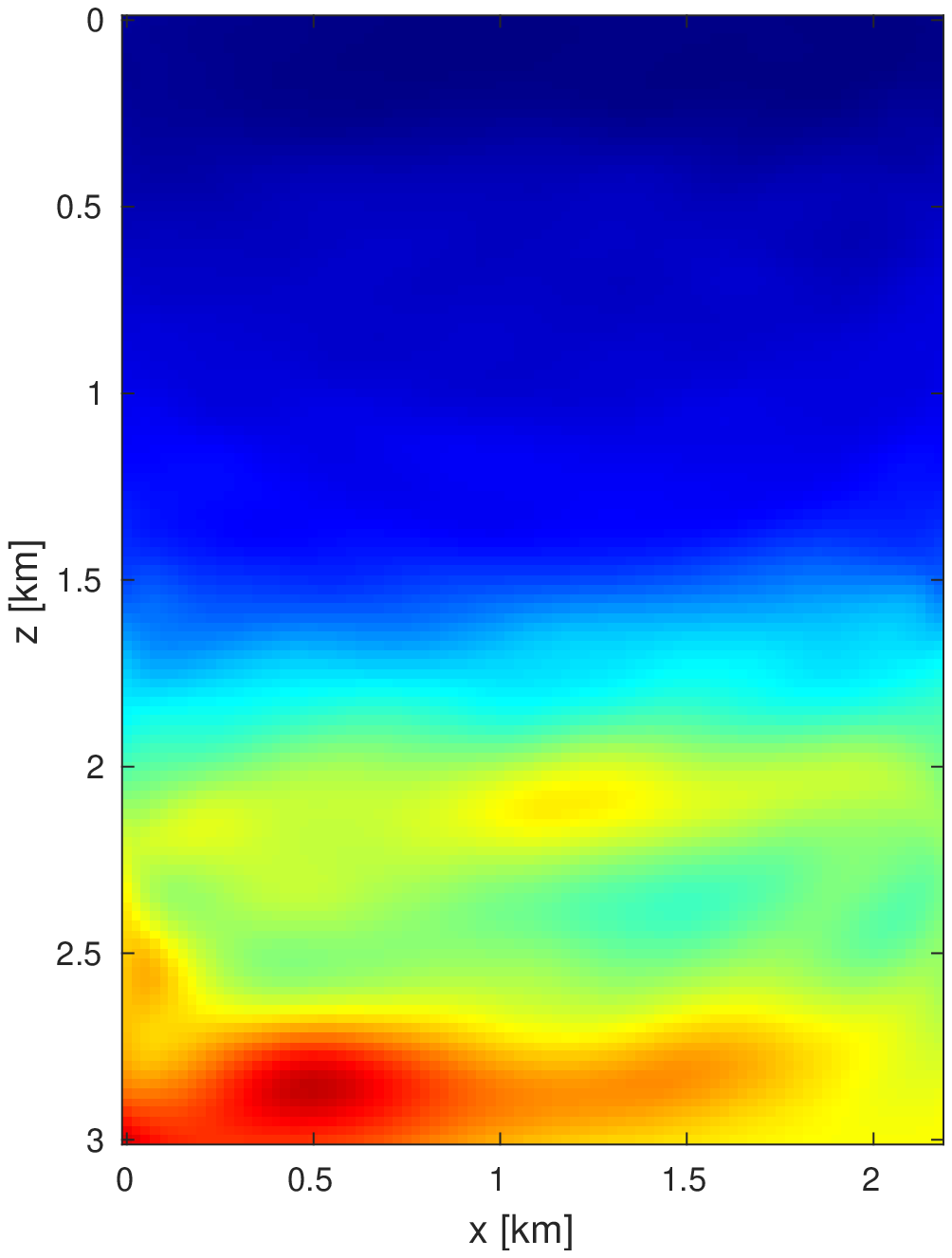} &
 \hspace{-1.0cm} \includegraphics[scale=0.28]{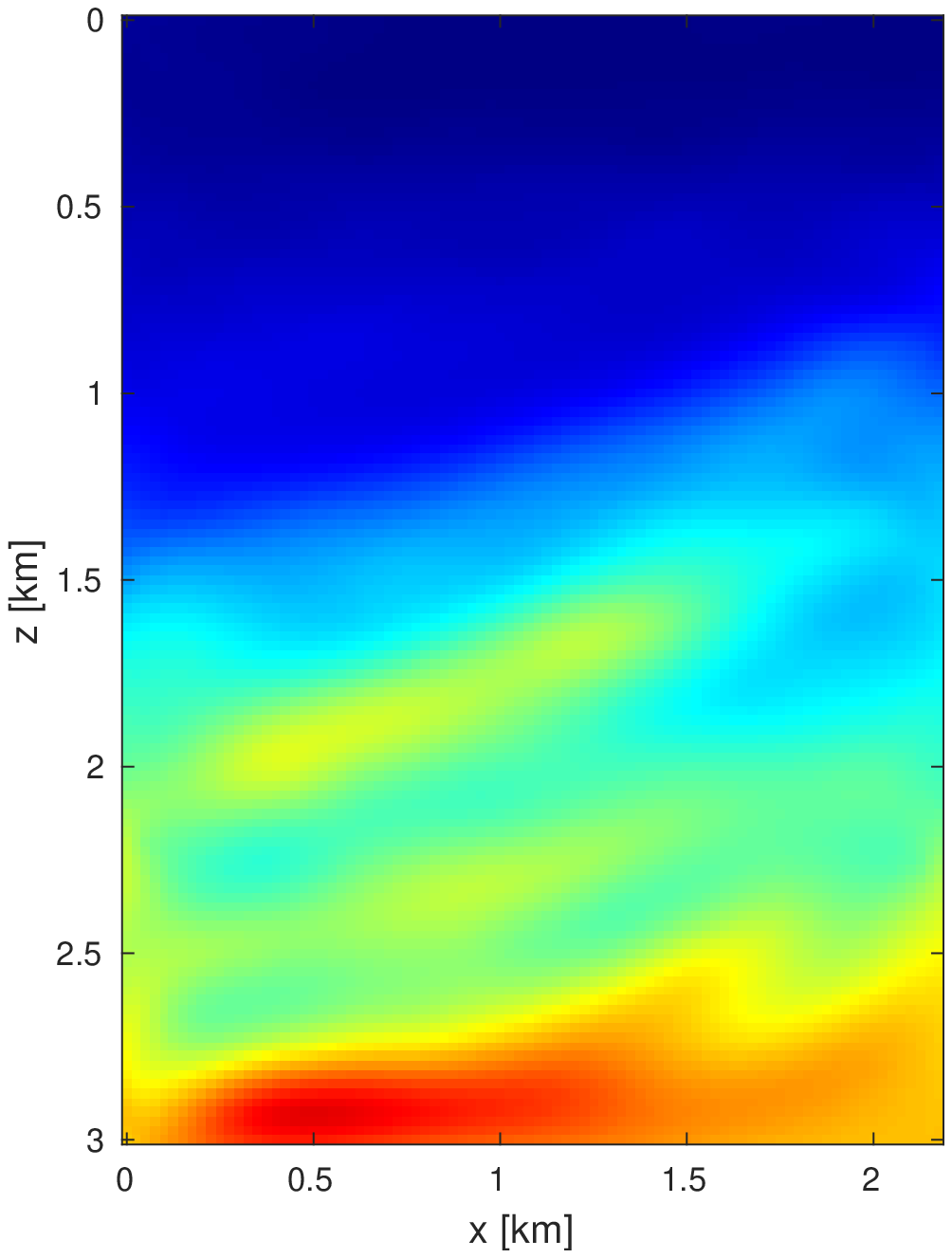} &
 \hspace{-1.0cm} \includegraphics[scale=0.28]{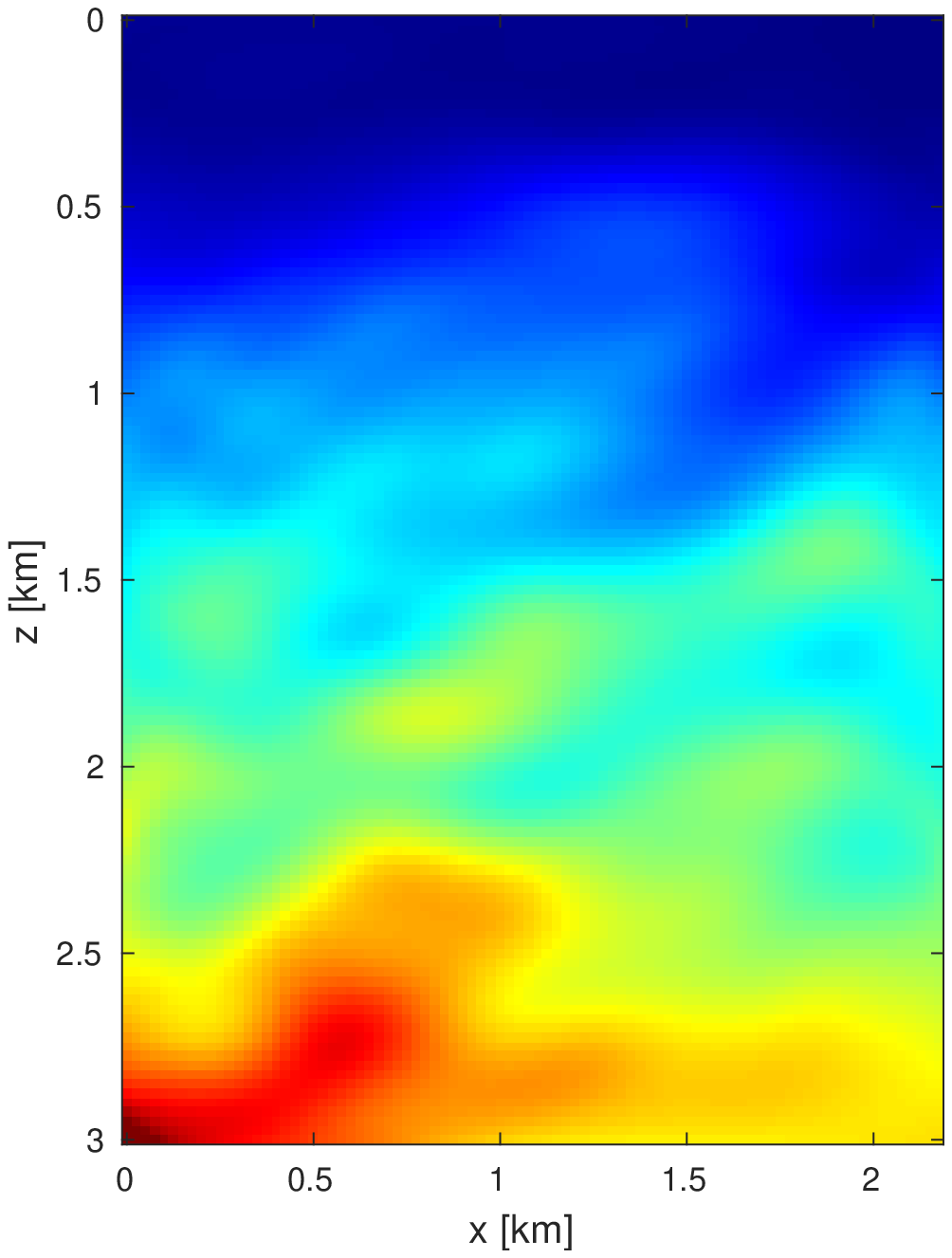} &
 \hspace{-1.0cm} \includegraphics[scale=0.28]{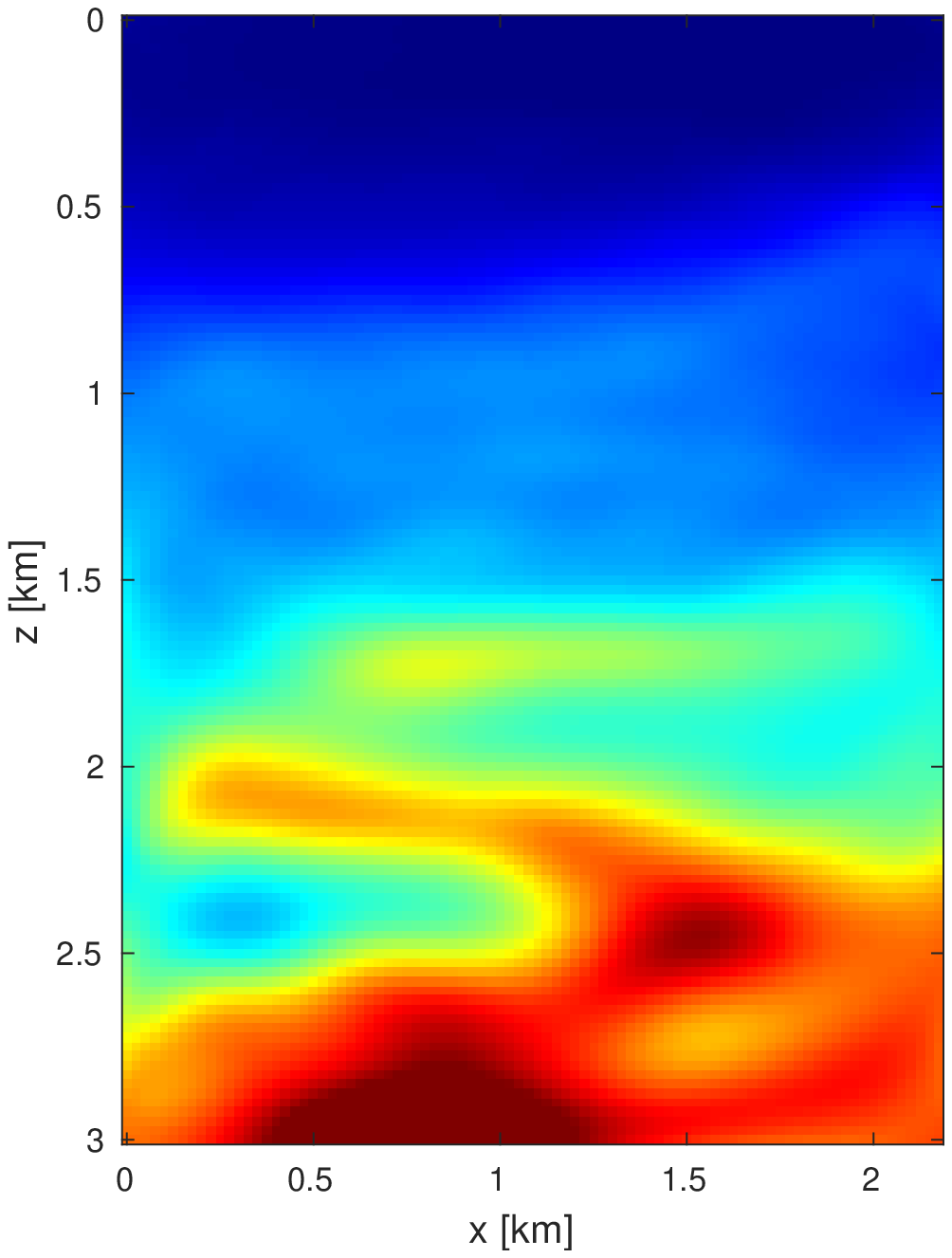}
\end{tabular}
\caption{{First row: the smooth Marmousi model divided into individual slices.
  Second row:  reconstructions using the non-optimised design parameters. Third 
  row: FWI reconstructions using the  $\mathcal{P}$ - $\alpha$ optimised designed parameters.
  Slices 1, 2, 3 and 5 were used for training, while slice 4 was used for testing.
  Colour maps are scaled to be in the same range; i.e., the colours correspond to the
  same values in each slice.}}\label{fig:reconstr}
\end{figure}

\begin{figure}[h!]
\begin{tabular}{cccccc}
 \hspace{-0.5cm}  & 
\hspace{-0.5cm} slice 1 (train) & 
\hspace{-0.5cm} slice 2 (train) & 
\hspace{-0.5cm} slice 3 (train) & 
\hspace{-0.5cm} slice 4 (test) & 
\hspace{-0.5cm} slice 5 (train)\\
 \hspace{-0.5cm}  \parbox[t]{2mm}{\multirow{3}{*}[7em]{\rotatebox[origin=c]{90}{unoptimised}}}  &
  \hspace{-0.5cm} \includegraphics[scale=0.27]{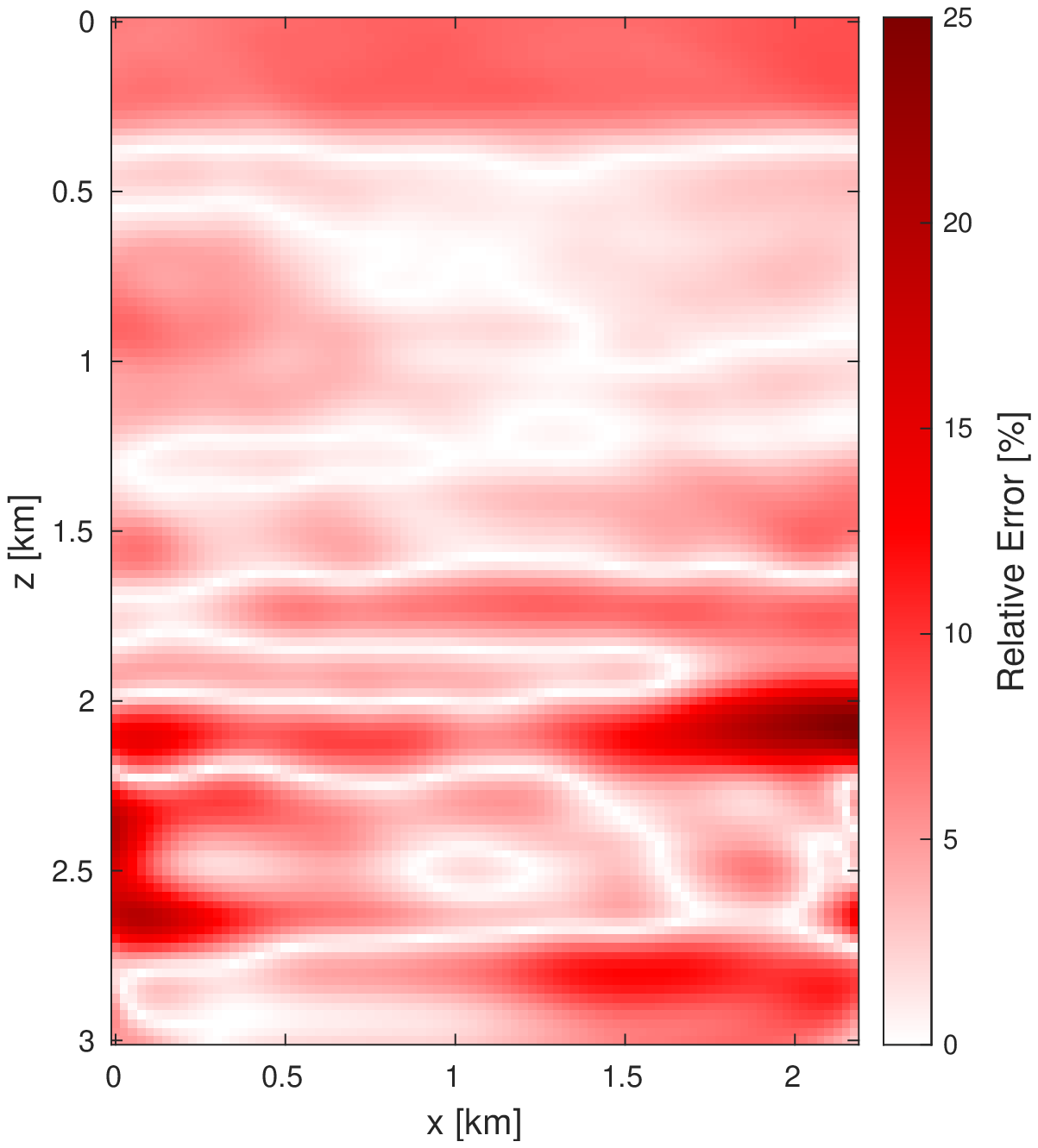} &
\hspace{-0.5cm} \includegraphics[scale=0.27]{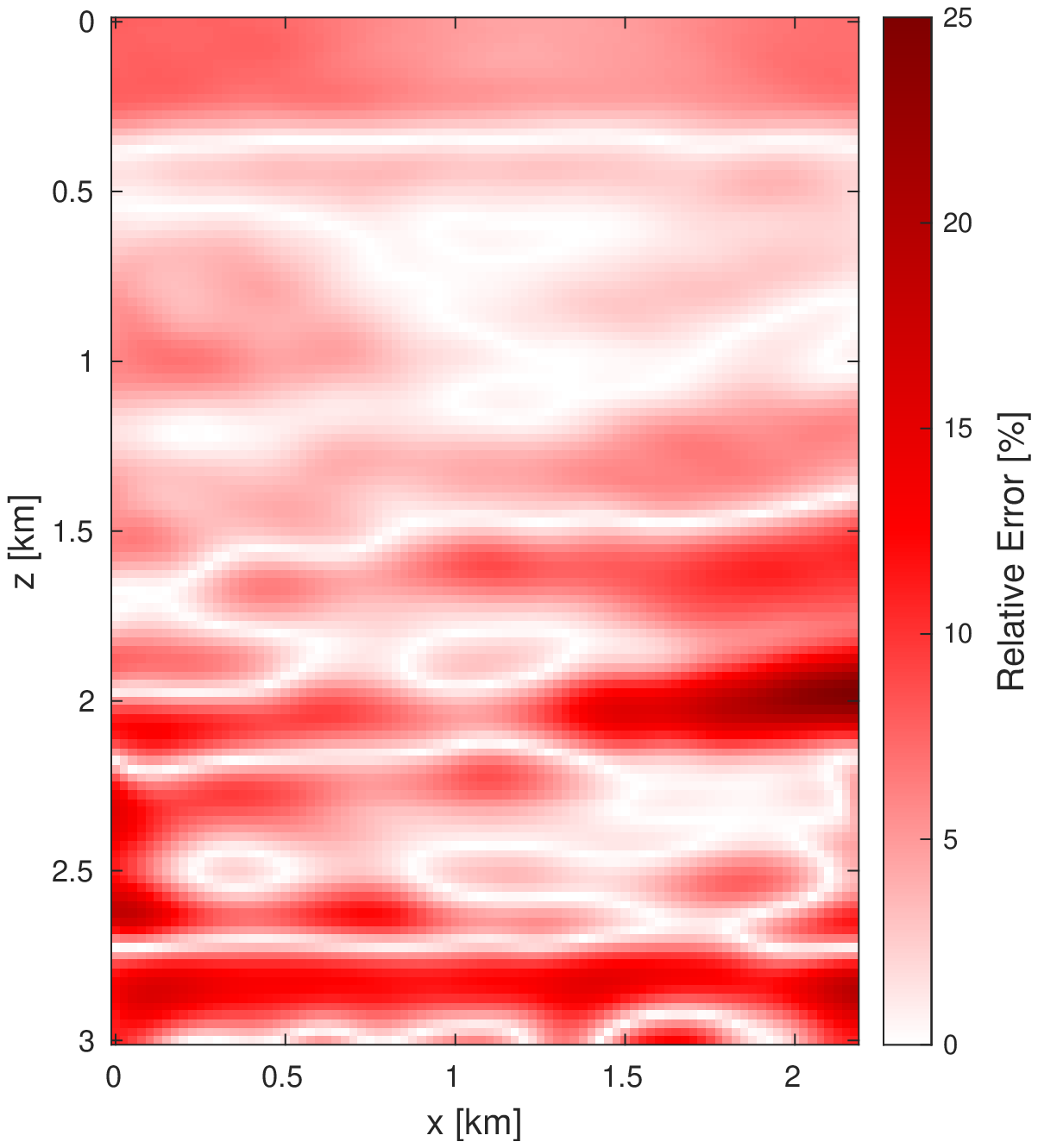} &
\hspace{-0.5cm} \includegraphics[scale=0.27]{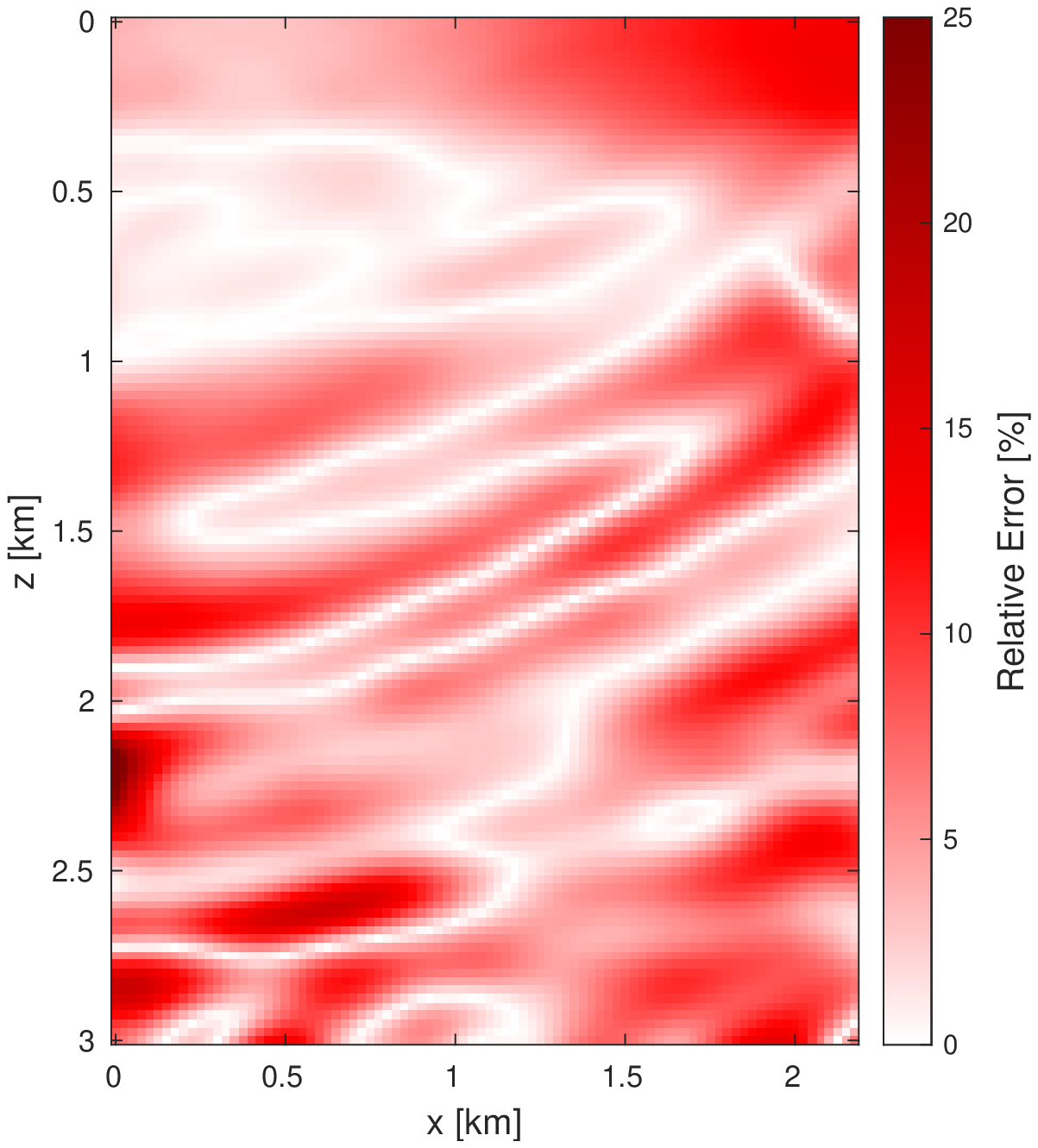} &
\hspace{-0.5cm} \includegraphics[scale=0.27]{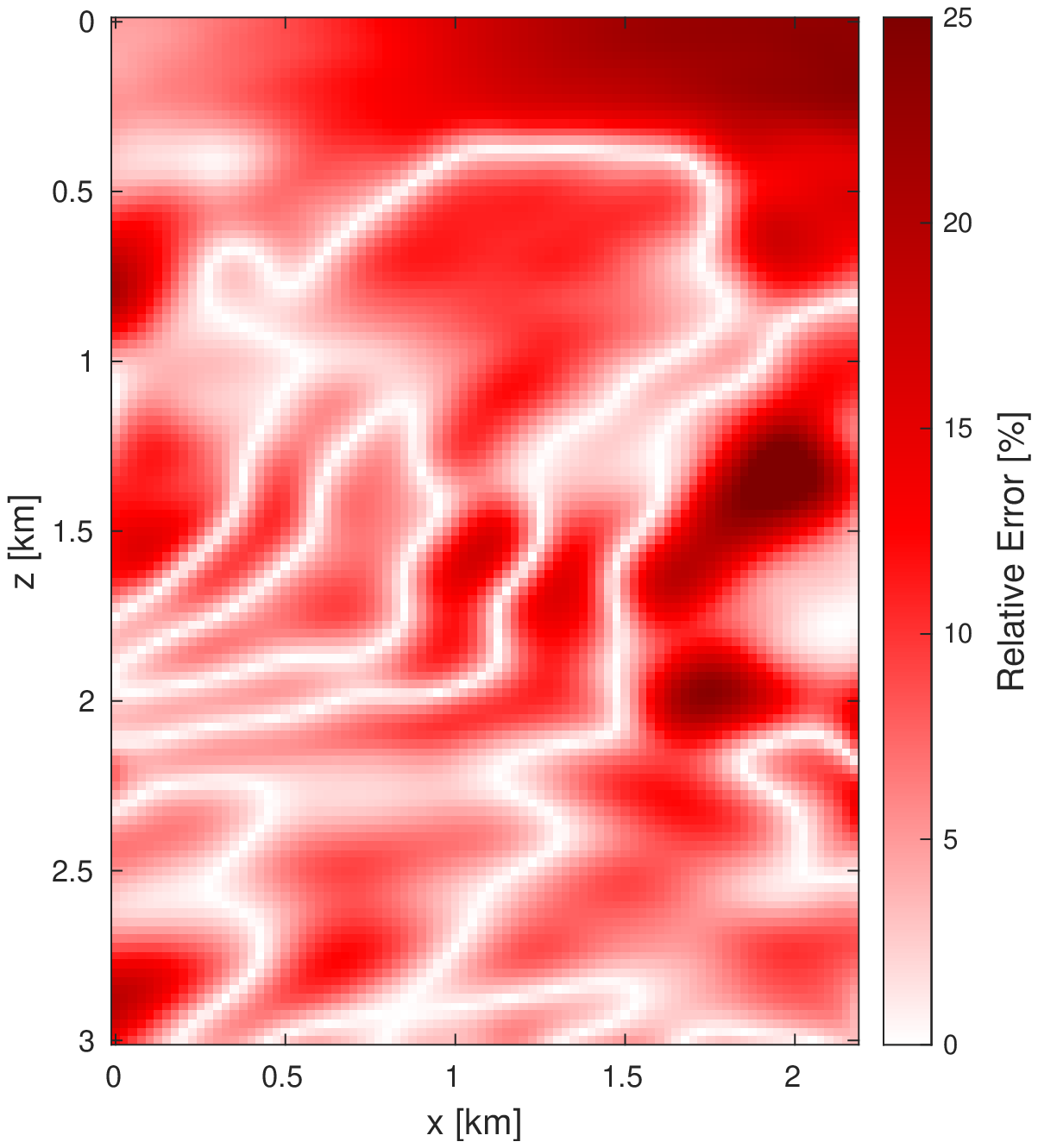} &
\hspace{-0.5cm} \includegraphics[scale=0.27]{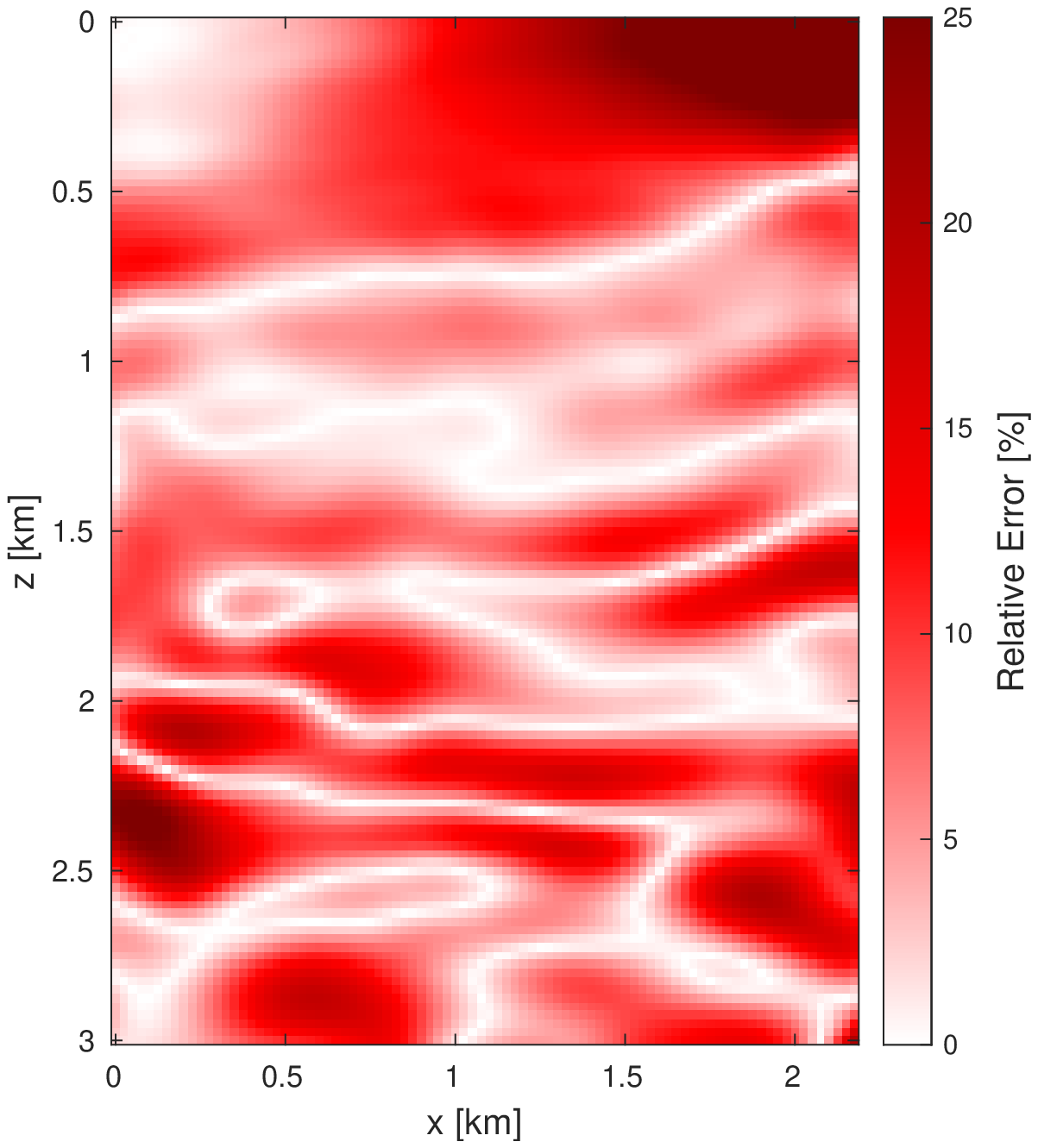}\\
  \hspace{-0.5cm}  \parbox[t]{2mm}{\multirow{3}{*}[7em]{\rotatebox[origin=c]{90}{$\mathcal{P},\alpha$-optimised}}}  &
 \hspace{-0.5cm} \includegraphics[scale=0.27]{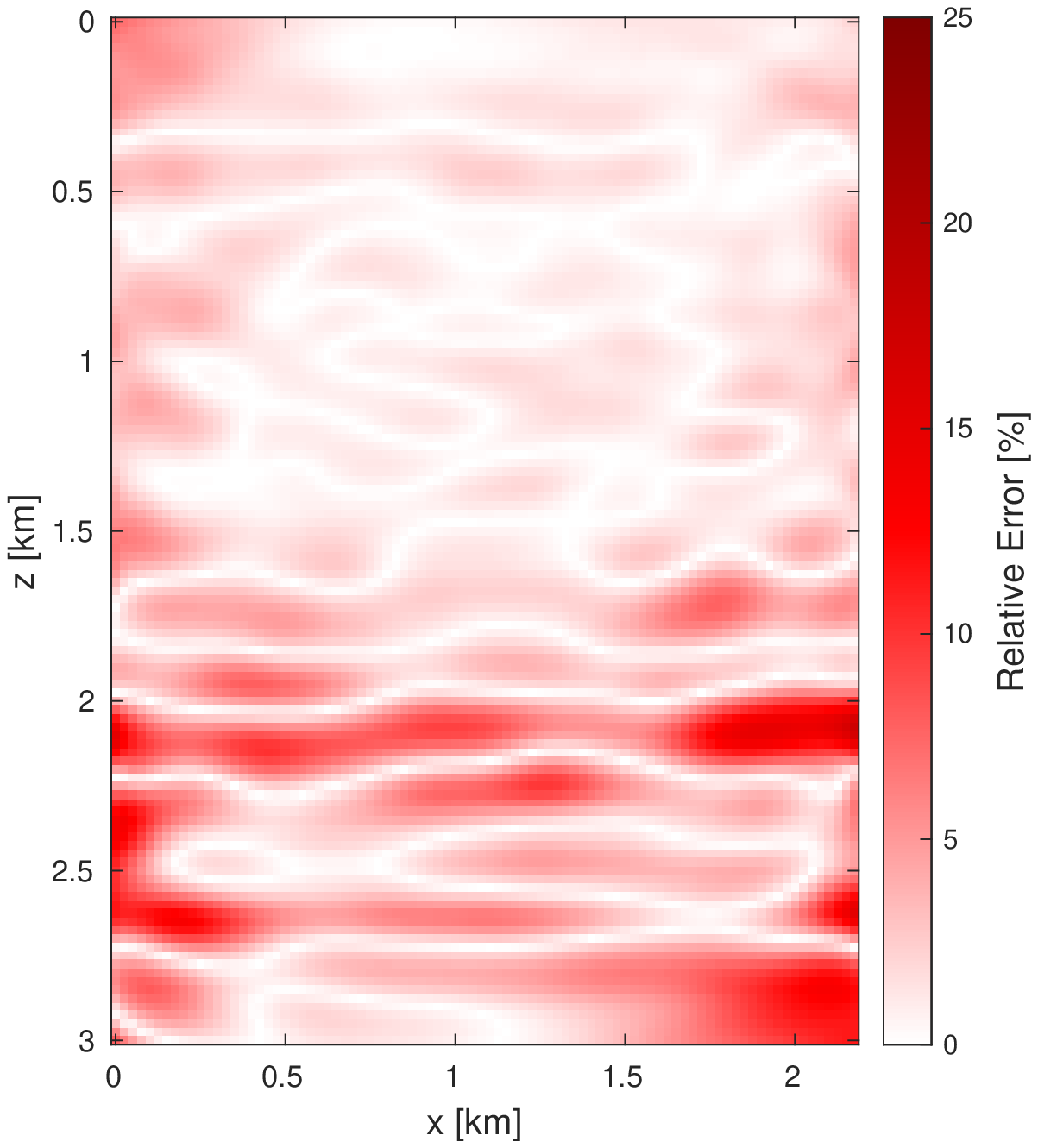} &
\hspace{-0.5cm} \includegraphics[scale=0.27]{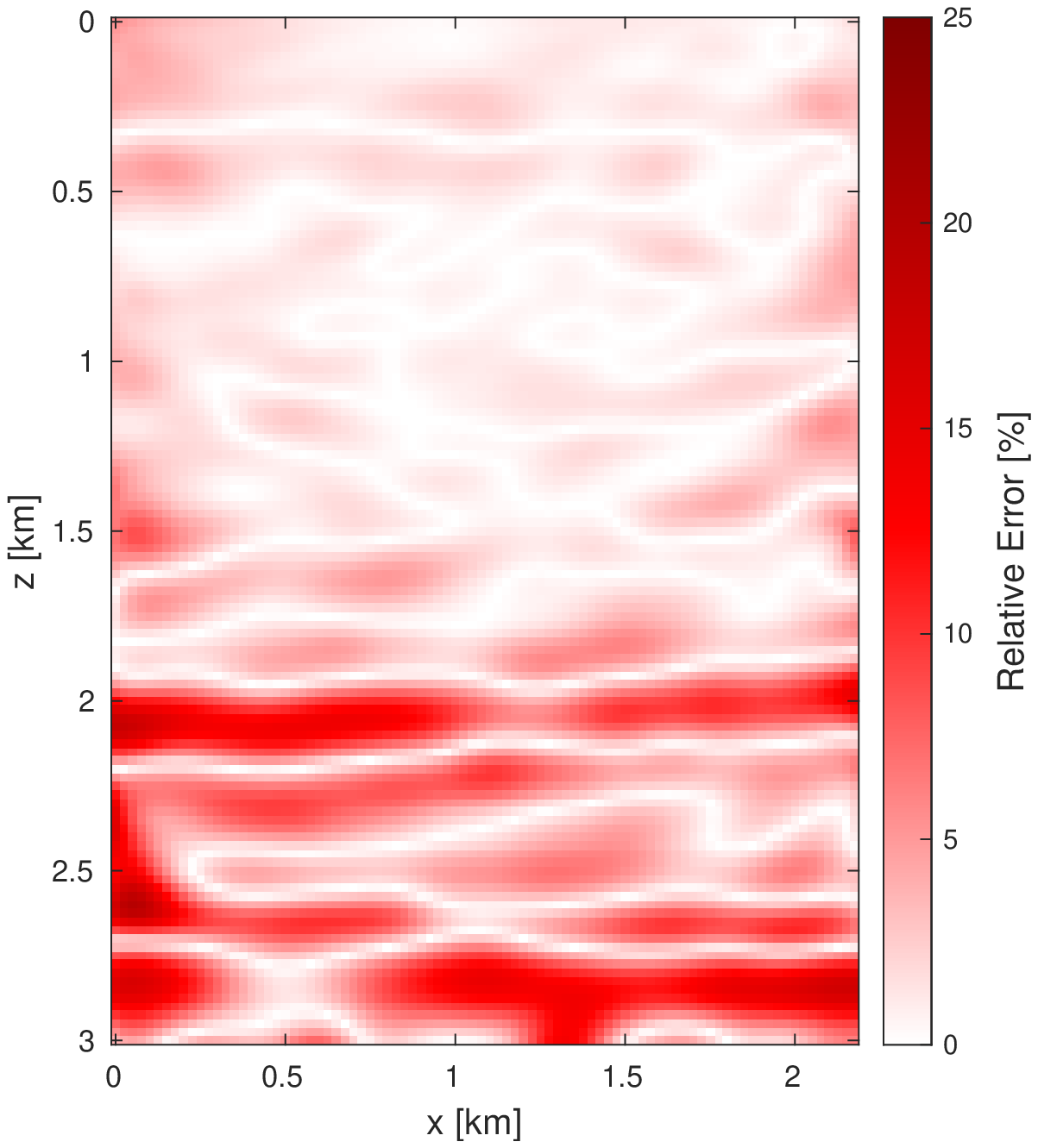} &
\hspace{-0.5cm} \includegraphics[scale=0.27]{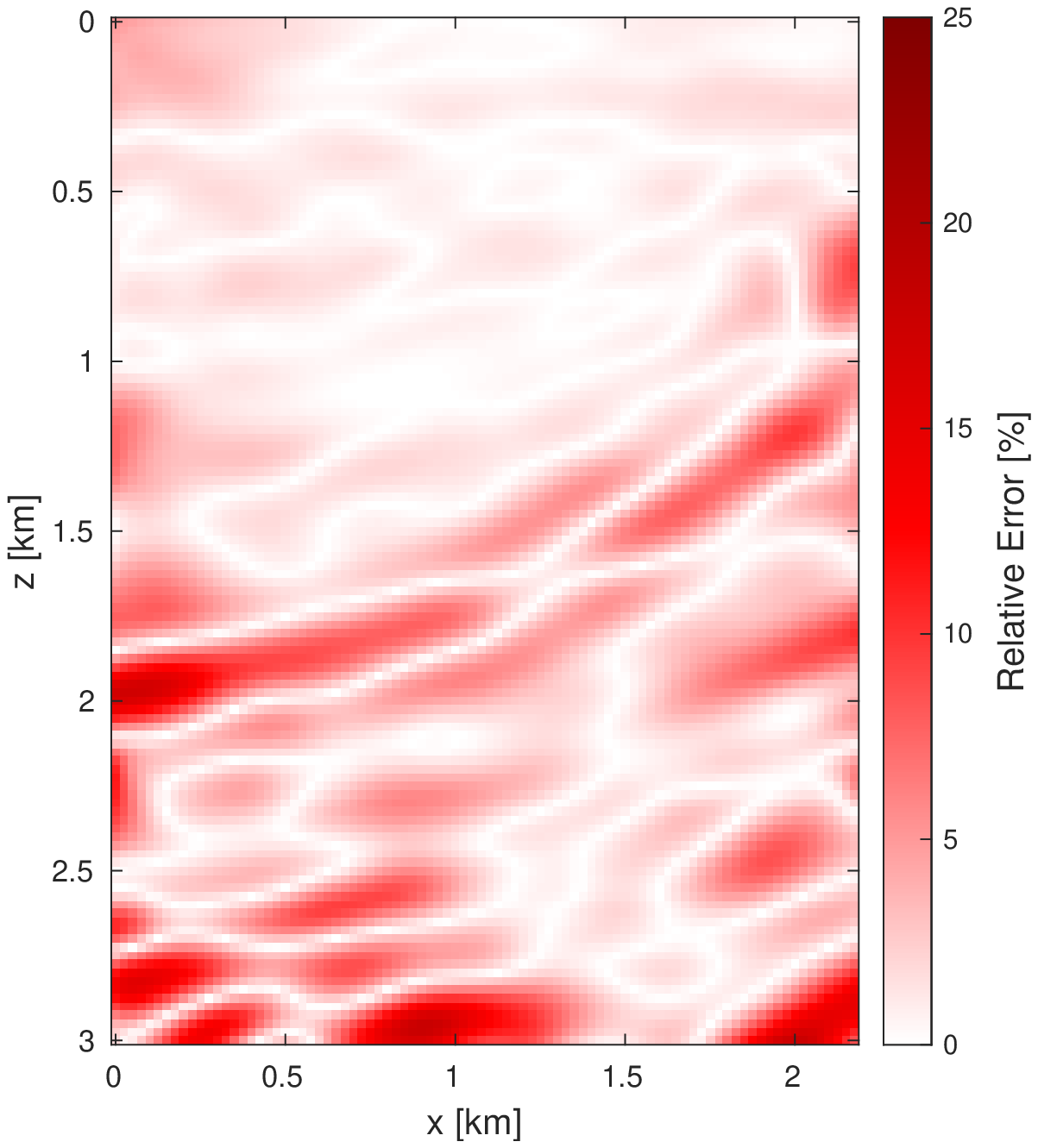} &
\hspace{-0.5cm} \includegraphics[scale=0.27]{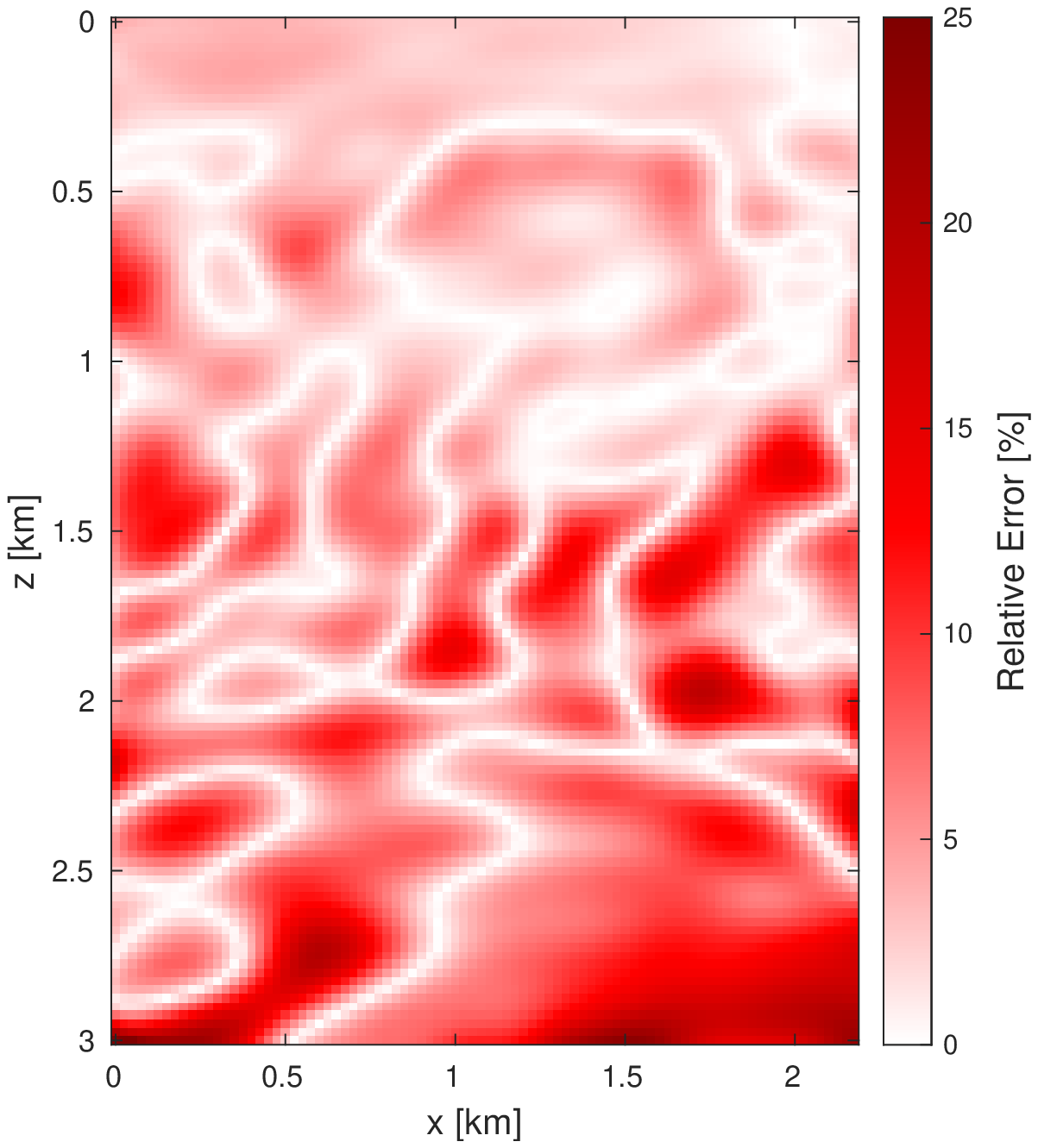} &
\hspace{-0.5cm} \includegraphics[scale=0.27]{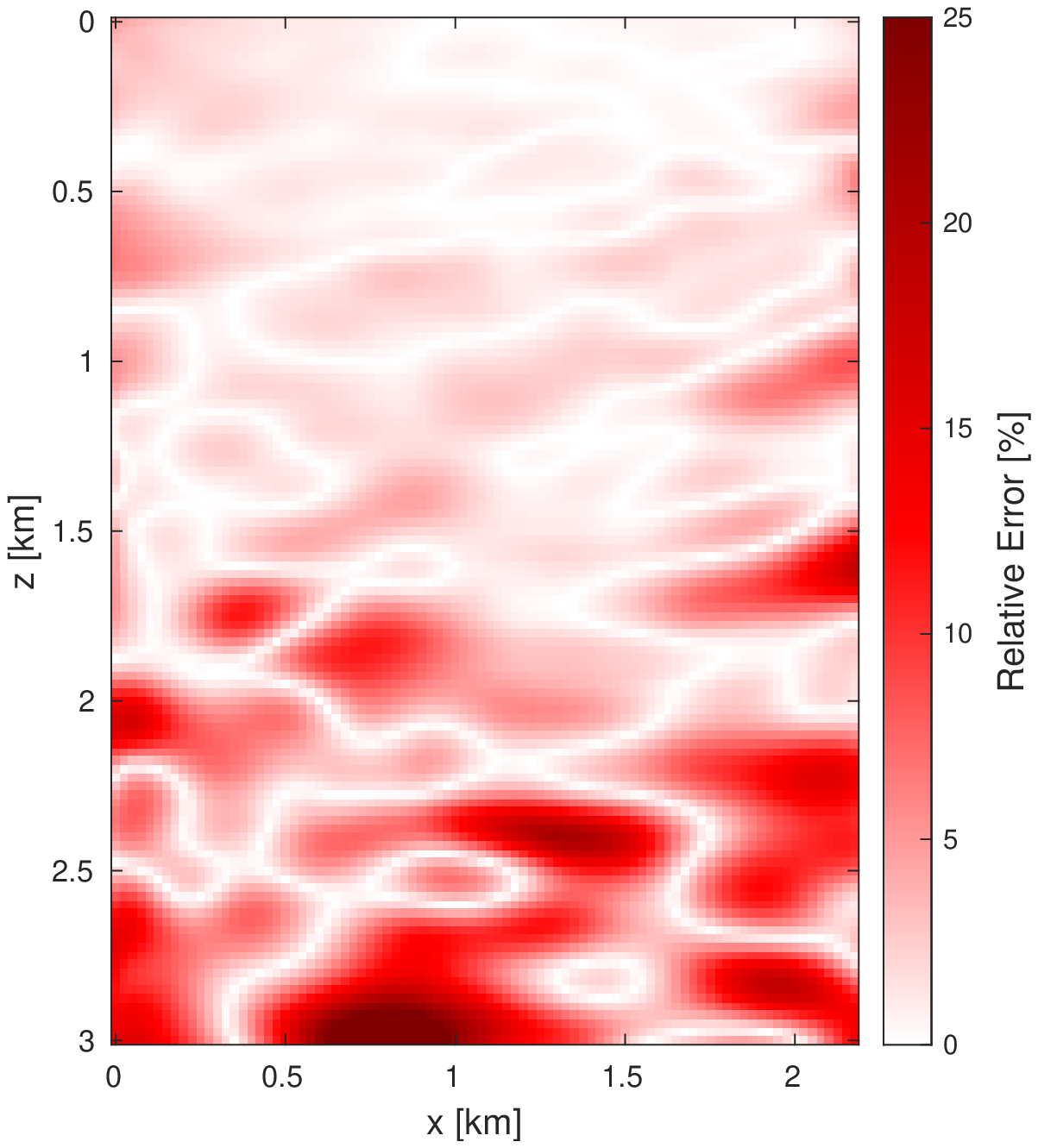}\\
\end{tabular}
\caption{{Absolute value of  the Mean Relative Percentage Error (MRE) defined in \eqref{relerr} for each of the scenarios in Figure 
  Slices 1, 2, 3 and 5 were used for training, while slice 4 was used for testing.}}\label{fig:relerr}
\end{figure}



{Examining the  unoptimised  results  for the training models,  we see    
that the large-scale structures are roughly  identifed.
However, the shapes and wavespeed values are not always correct.
For example, in Slices 1 and 2, the layer of higher wavespeed at depth approximately  2.1 km  has wavespeed values that are too low, and the thickness and length of the layer itself is far from correct. 
For Slice 5, the long thin arm ending at depth 2.0 km and width 0km is essentially missing.
For the testing model (Slice 4), the reconstruction  is relatively poor, for example the three `islands' at depth 1.25 - 2.0km are hardly visible.
After   optimisation,  all these issues are substantially improved on.
The images have appeared to `sharpen' up, particularly in the upper part of the domain, and the finer features of structures and boundaries between the layers have become more  evident. Examples of features which are now more visible include:  the thin strip of high wavespeed in slices 1 and 2 at depth approximately 2.1 km,  the two curved layers of higher wavespeed which sweep up to the right in Slice 3 and the three islands in slice 4 (mentioned above).   }


{When assessing these images we should bear in mind that this is a very underdetermined problem, we are reconstructing  10648 parameters using 5 sources, 5 sensors and 4 frequencies. So what is important is the improvement due to optimisation, rather than the absolute accuracy of the reconstructions.  

To further illustrate the improvemment obtained by optimisation, in Figure \ref{fig:relerr} we plot the quantity MRE (defined in  \eqref{relerr}) for  each of the cases  in Figure \ref{fig:reconstr}.
Since darker shading represents larger error, with white indicating zero error, we see more clearly here  the  conspicuous benefit of optimisation of the design parameters. }




{The optimised sensor positions $\mathcal{P}$ are displayed in Figures \ref{fig:Exp2StartFWI} (for $\cP$ optimisation only) and \ref{fig:Exp2Opt} (for $\cP,  \alpha$ optimisation).   (Both are  superimposed on Slice 2).  We note that the final sensor positions are very similar in the case of both $\cP$ and $\cP, \alpha$ optimisation. 
}


\begin{figure}[h!]
 \hspace*{0.75cm}
 \begin{subfigure}[c]{0.31\textwidth}
 \hspace*{-0.5cm}\includegraphics[scale=0.35]{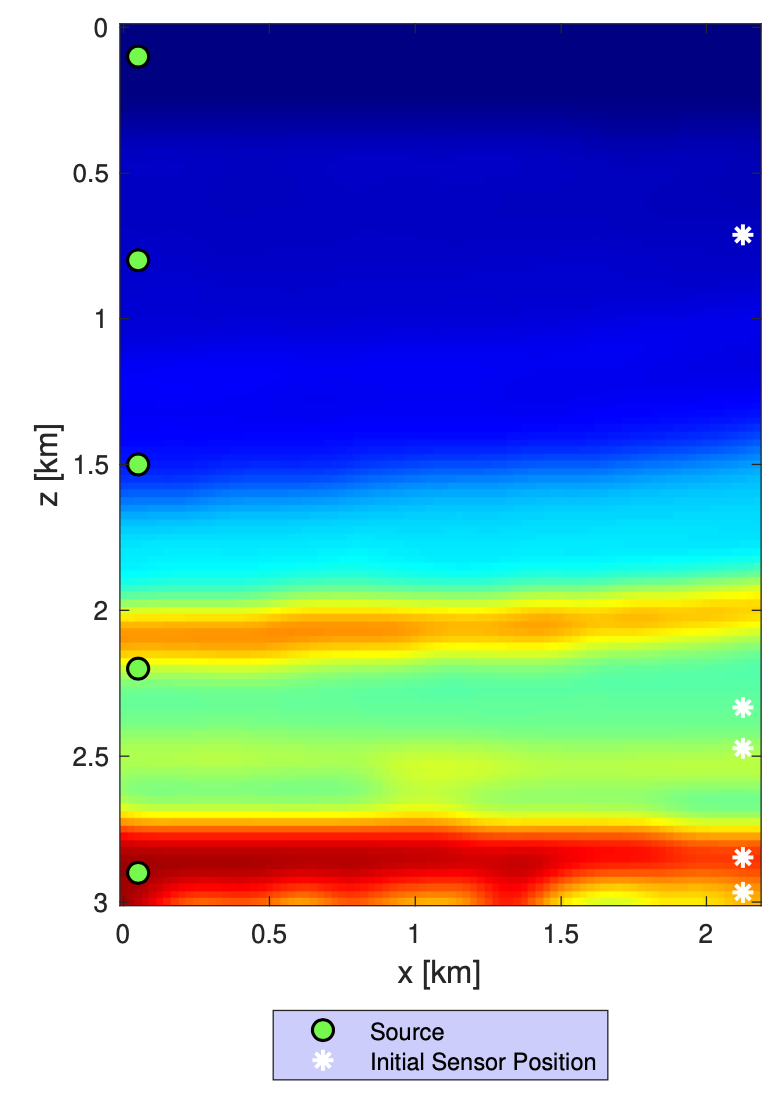}
               \caption{
\label{fig:Exp2Start}}
\end{subfigure}
\hspace{-0.5cm}
\begin{subfigure}[c]{0.31\textwidth}
 \hspace*{-0.1cm}\includegraphics[scale=0.35]{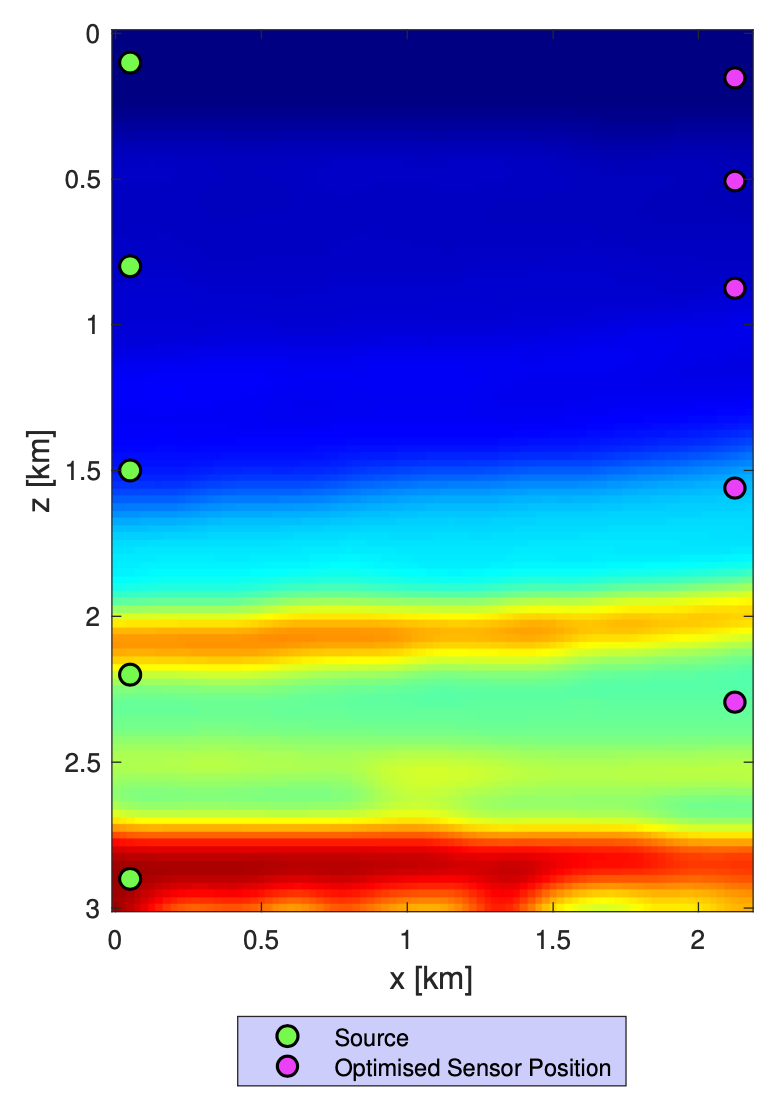}
\caption{
\label{fig:Exp2StartFWI}}
\end{subfigure}
\hspace{-0.1cm}
\begin{subfigure}[c]{0.31\textwidth}
\hspace*{-0.1cm}\includegraphics[scale=0.35]{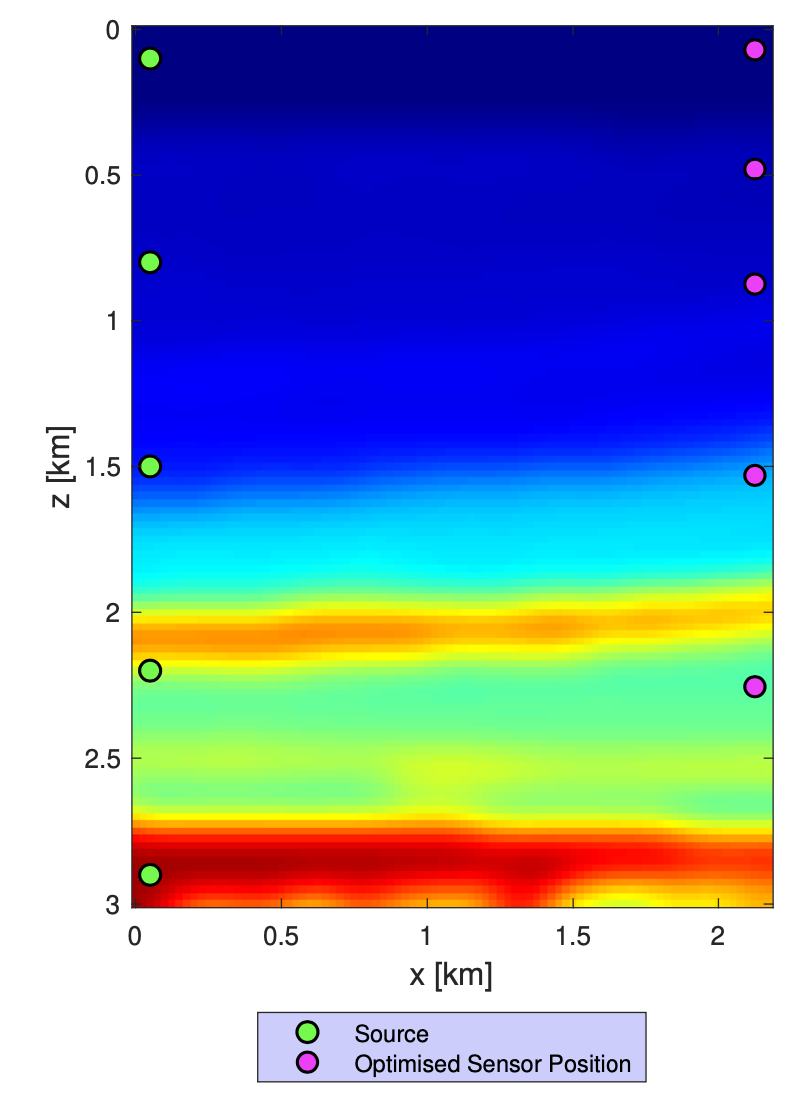}
\caption{
\label{fig:Exp2Opt}}
\end{subfigure} 
\caption{{(a) Source positions and {random}  non-optimised sensor positions; (b) Source positions and optimised sensor positions after $\cP$ optimisation (c) Source positions and optimised sensor positions after $\cP, \alpha$ optimisation.   Slices 1, 2, 3 and 5 were used  as training models, and slice 4
  as testing model.   Illustrations are presented on slice 2 for convenience.
  \label{fig:all5}}}
\end{figure} 

\subsection{Full cross-validation of the results.}

{To show the robustness of the results with respect to the learned design parameters when these are applied to a related (but different) model, we run our bilevel algorithm and  perform a full cross validation.}
{For each $i$ we choose as testing model the $i$th Marmousi slice, and we  train
  on the remaining slices,  repeating the experiment for each $i\in\{1,2,3,4,5\}$.  The results are presented in Table  \ref{tab:fullCV}. } 

\begin{table}[h!] 
  \begin{center}

\begin{tabular}{|| c || c | c || c | c ||c ||} 
 \hline
 Slice  & MRE ($\%$) & SSIM  &
MRE ($\%$) & SSIM & IF \\ [0.5ex] 
 \hline\hline
   & \multicolumn{2}{c||}{Starting parameters} & \multicolumn{2}{c||}{Optimised parameters} &\\\hline
    \multicolumn{6}{|c|}{Training = $\{2,3,4,5\}$, Testing = 1}\\\hline
    \textbf{1} &  \textbf{4.08}  & \textbf{0.82} & \textbf{2.78}  & \textbf{0.88} & \textbf{3.56}\\
    \hline
  2  & 4.45  & 0.76 &2.91&0.83&3.60\\
\hline
  3  &4.73  &0.78&  3.15 & 0.85 &3.95\\
\hline  
  4  &7.38 &0.67&4.36& 0.78 & 8.89\\
\hline
  5  & 7.31  & 0.68 &3.51 & 0.83 & 13.97\\ 
  \hline
  \hline
 \multicolumn{6}{|c|}{Training = $\{1,3,4,5\}$, Testing = 2}\\\hline
 1  & 4.06  & 0.82& 2.95 & 0.87 & 3.28 \\
 \hline
\textbf{2} &   \textbf{4.46}  & \textbf{0.76} & \textbf{3.27}  & \textbf{0.81} & \textbf{2.86}\\
 \hline
  3   &4.73 &0.78 &3.11 &0.86 &4.58\\
 \hline
  4  &7.37 &0.67& 4.91 &0.77 &7.88\\
 \hline
  5  &7.31  &0.68 &3.86 &0.83 &12.59\\ 
  \hline
  \hline
  \multicolumn{6}{|c|}{Training = $\{1,2,4,5\}$, Testing = 3}\\\hline
  1  &4.06 &0.82 &2.41&0.89 &4.68\\
 \hline
  2  &4.45  &0.76 &2.63 &0.84 &3.99\\
 \hline
\textbf{3} &  \textbf{4.75}  & \textbf{0.78} & \textbf{3.18}   & \textbf{0.85} & \textbf{3.56}\\
 \hline
  4  &7.38  &0.67 &3.86 &0.80 &9.50 \\
 \hline
  5  &7.31 &0.68 &3.27 &0.85 &15.17 \\ 
  \hline
  \hline
  \multicolumn{6}{|c|}{Training = $\{1,2,3,5\}$, Testing = 4}\\\hline
  1  &4.06 &0.82 & 2.45 &0.89 &4.92 \\
 \hline
  2  &4.44  &0.76 & 3.03 &0.82 &3.87\\
 \hline
  3  &4.73  & 0.78 &2.45 &0.86 &6.55\\
  \hline
  \textbf{4} &  \textbf{7.37}  & \textbf{0.67} & \textbf{4.92}  & \textbf{0.76} & \textbf{7.56}\\
  \hline
  5  &7.31  &0.68 &3.51 &0.86 & 18.33\\ 
  \hline
  \hline
 \multicolumn{6}{|c|}{Training = $\{1,2,3,4\}$, Testing = 5}\\\hline
 1 &4.06 & 0.82 &2.39 & 0.88 &4.93\\
 \hline
  2  &4.45  &0.76 & 2.92 &0.83 &3.73\\
 \hline
  3  &4.73  &0.78 &2.51& 0.85 &5.56\\
 \hline
  4  &7.38  & 0.67& 4.25& 0.79 &9.52\\ 
 \hline
  \textbf{5} &   \textbf{5.64}  &   \textbf{0.71} &   \textbf{3.55}  &   \textbf{0.85} &   \textbf{5.89}\\
  \hline
\end{tabular}

\end{center}
\caption{{Values of Mean Relative percentage error (MRE), SSIM and Improvement Factor (IF) for  
 full cross-validation of the bilevel learning algorithm, using $\cP, \alpha$ optimisation. Four slices of the smoothed Marmousi model are used for training, leaving one slice for testing. 
  {The test cases}   are highlighted in boldface. 
  The values in the fourth pane of this table coincide with  those reported in the first two and last three   columns of Table \ref{tab:jointOpt}.}} \label{tab:fullCV}
\end{table}

{The  results show that good design parameters for inverting data coming from an unknown model can potentially be obtained by training on a set of related models. Indeed, while the  training and testing models above have differences, they also  share some properties, such as the range of wavespeeds present and the fact that the wavespeed increases, on average, as the depth  increases. The results in  Table \ref{tab:fullCV} show that optimisation leads to a robust improvement in SSIM throughout and obtains Improvement Factors in the range 2.86-7.56 for all choices of training and testing regimes.    }

\section{{Conclusions and outlook}} 
\label{sec:conclusion} 
{In this paper we have  proposed  a new bilevel learning algorithm for computing optimal sensor positions and regularisation weight to be used with Tikhonov-regularised FWI. The numerical experiments  validate the proposed methods in that, on a specific cross-well test problem based on the Marmousi model, they  clearly show the benefit of jointly optimising sensor positions and regularisation weight versus using arbitrary values for these quantities. 

We list a number of important (and interrelated) open questions and extensions which can  be addressed by future research. (1) Assess the potential of the proposed method for 
geophysical problems, where (for example) rough surveys could be used to drive parameter optimisation for inversion from more extensive surveys. (2) Test our algorithm on more extensive collections of large-scale, multi-structural benchmark datasets, such as the recently developed OpenFWI \cite{openFWI}. (3) Apply our algorithm in situations where the lower level optimisation problem is FWI equipped with non-smooth regularisers, such as total (generalised) variation. (4) Extend the present bilevel optimisation algorithm to learn
parametric regularisation functionals to be employed within FWI. (5) Extend the present bilevel optimisation algorithm to also learn the optimal number of sensors,  the optimal number and positions of sources, as well as the optimal number of frequencies and their values.}
    
                         \noindent {\bf Acknowledgements.} \ We thank  Schlumberger Cambridge Research for financially supporting the PhD studentship of Shaunagh Downing within the SAMBa Centre for Doctoral Training at the University of Bath. The original concept for this research was proposed by Evren Yarman,  James Rickett and Kemal Ozdemir (all Schlumberger) and we thank them all for   many useful, insightful  and supportive discussions.
   We also thank Matthias Ehrhardt (Bath),  Romina Gaburro (Limerick), Tristan Van Leeuwen (Amsterdam), and 
                         St\'{e}phane Operto and Hossein Aghamiry (Geoazur, Nice) for helpful comments and discussions.

                         We gratefully acknowledge support from the UK Engineering and Physical
Sciences Research Council Grants EP/S003975/1 (SG, IGG, and
EAS),  EP/R005591/1 (EAS), and EP/T001593/1 (SG). This research made use of the Balena High Performance Computing (HPC) Service at the
University of Bath.

\bibliographystyle{siam}

\bibliography{bib}

\section{{Appendix  -- Details of the numerical implementation}}
\label{app:Details}

{In this appendix we give some details of the numerical implementation of our bilevel learning algorithm for learning optimal sensor positions and weighting parameter in FWI (see Figure \ref{fig:bilevel} for a schematic illustration).}

\subsection{Numerical forward model}
  \label{subsec:num_for} 

  To approximate the action of the forward operator $\mathscr{S}_{\bm,\omega}$ and its adjoint (defined in 
Definition   \ref{def:solnop}) we use a finite element method with quadrature. 
    Specifically, we write the problem \eqref{forward} satisfied by $u\in H^1(\Omega)$ (the Sobolev space of functions on $\Omega$ with one square-integrable  weak derivative) in weak form: 
\begin{align*} 
  a(u,v) : & = \int_\Omega \left(\nabla u \cdot \nabla \overline{v}
  - \omega^2 m u \overline{v}\right)
             - \ri \omega \int_{\partial \Omega} \sqrt{m} u \overline{v}
  \ =\  \int_\Omega f \overline{v} + \int_{\partial \Omega} f_b \overline{v} = : F(v) ,      
 \end{align*}
        {for all} $ v \in H^1(\Omega)$.
        This can be  discretised using the finite element method in  the space of  continuous piecewise polynomials of any degree $\geq 1$. Although high order methods are preferrable   for high-frequency problems,
        the Helmholtz problems which we deal with here are relatively low frequency and so here we employ
        linear finite elements with
        standard hat function basis on a uniform rectangular grid (subdivided into  triangles).
        With $h$ denoting mesh diameter the approximation space is      denoted  $V_h$. 
        The numerical solution $u_h\in V_h $ then satisfies
     $ a(u_h, v_h ) = F(v_h)$,   {for all}  $v_h \in V_h.$
Expressing  $u_h$ in terms of the hat-function basis $\{ \phi_j\} $ yields a  linear system in the form  
 \begin{align*}
  A(\bm, \omega) \mathbf{u} 
   = \mathbf{f} + \mathbf{f}_b,
 \end{align*}
        to be solved for the nodal values $\mathbf{u}$ of $u_h$. Here the matrix $A(\bm, \omega)$ takes the form
        \begin{align*}
         A(\bm, \omega) = S -\omega^2 M(\bm) - \ri \omega B(\bm) ,
         \end{align*}
        with  $S_{i,j}= \int_{\Omega} \nabla \phi_i \cdot \nabla \phi_j$, 
\begin{align}
M(\bm)_{i,j}=  \int_{\Omega} m\, \phi_i \cdot \phi_j, \quad   
B(\bm)_{i,j}= \int_{\partial \Omega} \sqrt{m} \,\phi_i \cdot \phi_j, 
\label{matrices}\end{align}
        \begin{align} f_j = \int_\Omega f \phi_j, \quad \text{and} \quad 
(f_b)_j = \int_{\partial \Omega} f_b \phi_j \label{rhss}\end{align}

To simplify this  further we approximate the integrals in  \eqref{matrices} and \eqref{rhss} by nodal
quadrature,  leading to approximations (again denoted $M$ and $B$) taking  the simpler diagonal form:  
      $$ M(\bm) = \mathrm{diag}\{d_k m_k\}, \quad B(\bm) = \mathrm{diag}\{b_k \sqrt{m_k}\}, $$
where $k = 1, \ldots M$ denotes a labelling of the nodes and $\bd, \bb \in \mathbb{R}^M_+$ are mesh-dependent vectors with   $b_k$ vanishing  at interior nodes. Moreover 
$$(\bff)_k = d_k f(x_k),   \quad  (\bff_b)_k = b_k f(x_k)  $$
are the  vectors of (weighted) nodal values of the functions
$f,f_b$. 
        Analagously, solving with $A(\bm, \omega)^* = S - \omega^2 D(\bm) + \ri \omega B(\bm)$ represents numerically the action of the adjoint solution operator $\mathscr{S}^*_{\bm, \omega}$. 

        All our computations in this paper are done on rectangular domains discretized with uniform rectangular meshes {(each subdivided into two triangles)} , in which case  $S$ corresponds  to the ``five point Laplacian'' arising in lowest order finite difference methods and
        $M,B$ are diagonal matrices,  analogous to  (but not the same as) those proposed in a
        finite difference context in  \cite{TvLcode,VaHe:16}.            

 \noindent {\bf Computing the wavefield.} \   When the source $s$ is a gridpoint, the wavefield  $\bu = \bu(\bm, s, \omega)$ (i.e. the approximation to $u(\bm, \omega,s)$ defined in   \eqref{defu}) is found by solving
 $$ A(\bm, \omega) \bu = \bfe_s, \quad   $$
where $ (\bfe_s)_k = 0 $ for $k \not = s$ and $(\bfe_s)_s = 1 $ (i.e. the standard basis vector centred on node  $s$). When $s$ is not a grid-point we still generate the  vector $\bf$ by inserting $f = \delta_s$ in the first integral in \eqref{rhss} and note that $f_b = 0$ in this case.

        Our implementation is in Matlab and the linear systems
  are factorized using the sparse direct (backslash) operator available there.
        Our code development for the lower  level problem  was
        influenced by \cite{TvLcode}.

             In the numerical implementation of the bilevel algorithm, the wavefields $u(\bm, \omega,s)$ and $u(\bm',\omega,s)$ were  computed on different grids before computing the misfit  \eqref{resdbi}, as is commonly done when avoiding  `inverse crimes'. This is done at both training and testing steps in Section \ref{sec:Numerical}.  We also  tested the bilevel algorithm with and without the addition of artificial noise in the misfit  $\beps$ and it was found that  adding noise made the upper level objective $\psi$ much less smooth. As a result, noise  was not included in the definition of $\phi$
        when the  design parameters were optimised. However,  noise was added to the synthetic data when the optimal design parameters were tested.

        \subsection{Quasi-Newton methods} \label{subsec:Quasi}

        At the lower level, the optimisation   is done using the L-BFGS method (Algorithms 9.1 and 9.2 in \cite{wright1999numerical}) with Wolfe Line Search.  Since the FWI runs for each training model are independent of eachother, we parallelise the lower-level over all training models. 
The upper-level optimisation is performed using a bounded version of the L-BFGS algorithm (namely L-BFGSb), chosen to ensure that the sensors stay within the domain we are considering.
  Our implementation of the variants of BFGS is based on \cite{granzow} and \cite{byrd1995limited}. More details are in  \cite[Section 5.4]{thesis}. 
 
        \subsection{Numerical restriction operator} \label{subsec:restrict} 

        In our implementation, the  restriction operator $\cR(\cP)$ defined in \eqref{Rv} is discretised as an  $N_r \times N$ matrix $R(\cP)$, where $N_r$ is the number of sensors and $N$ is the number of finite element nodes. For any nodal vector $\bu$ the product $R(\cP) \bu$ then contains 
        a vector of  approximations to the quantities $\{u_h(p): \, p \in \cP\}$.  The action of $R(\cP)$ 
        does not simply produce the values of $u_h$ at the points in $\cP$,  because this would not be sufficiently  smooth. In fact experiments with such a definition of $R(\cP)$ yielded generally poor results when used in the bilevel algorithm --  recall that the  upper-level gradient formula \eqref{gradALT1} involves the spatial  derivative of the  restriction operator.

        Instead, to obtain a sufficiently smooth dependence on the candidate sensor positions,  we use a ``sliding  cubic'' approximation defined as follows.  
        First, in one dimension (see Figures \ref{fig:Cub1}, \ref{fig:Cub2}), with  $p$ denoting
        the position of a sensor moving along a line, the value of the interpolant is found by cubic interpolation at the four nearest nodes. So as $p$ moves,  the nodal values used change. 
 In two dimensions, 
        we perform bicubic interpolation using 
the four closest nodes  in each direction.

\begin{figure}[h!]
\centering
\begin{tikzpicture}[scale=2, circ/.style={shape=circle, color=red, line width=0.75mm,, inner sep=5.5pt, draw, node contents=}]

\draw [->,gray] (0,0) -- (7,0) node [below] {\small $x$};

\foreach \n/\x in {0/0.5, 1/1.5, 2/2.5, 3/3.5 4/4.5 5/5.5}
{
   \fill (\x,0) circle (1.6pt) node [below=7pt] {$x_{\n}$};
}
\fill (4.5,0) circle (1.6pt) node [below=7pt] {$x_{4}$};
\fill (5.5,0) circle (1.6pt) node [below=7pt] {$x_{5}$};
\fill (6.5,0) circle (1.6pt) node [below=7pt] {$\hdots$};
\fill[red] (2,0) circle (1.3pt) node [below=7pt] {$p$};
\draw node (c1) at (0.5, 0) [circ];
\draw node (c2) at (1.5, 0) [circ];
\draw node (c3) at (2.5, 0) [circ];
\draw node (c4) at (3.54, 0) [circ];
\end{tikzpicture}
\caption{Points used for cubic interpolant for sensor $p$ in interval $[x_1, x_2]$ \label{fig:Cub1}}
\end{figure}
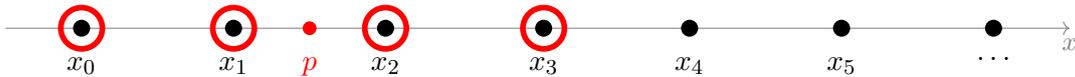

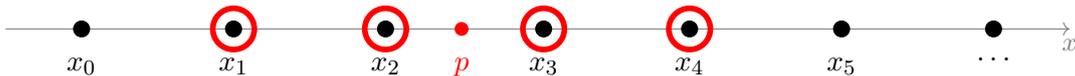
\begin{figure}[h!]
\centering
\begin{tikzpicture}[scale=2, circ/.style={shape=circle, color=red, line width=0.75mm,, inner sep=5.5pt, draw, node contents=}]

\draw [->,gray] (0,0) -- (7,0) node [below] {\small $x$};

\foreach \n/\x in {0/0.5, 1/1.5, 2/2.5, 3/3.5 4/4.5 5/5.5}
{
   \fill (\x,0) circle (1.6pt) node [below=7pt] {$x_{\n}$};
}
\fill (4.5,0) circle (1.6pt) node [below=7pt] {$x_{4}$};
\fill (5.5,0) circle (1.6pt) node [below=7pt] {$x_{5}$};
\fill (6.5,0) circle (1.6pt) node [below=7pt] {$\hdots$};
\fill[red] (3,0) circle (1.3pt) node [below=7pt] {$p$};
\draw node (c1) at (4.5, 0) [circ];
\draw node (c2) at (1.5, 0) [circ];
\draw node (c3) at (2.5, 0) [circ];
\draw node (c4) at (3.54, 0) [circ];
\end{tikzpicture}
\caption{Points used for cubic interpolant for sensor $p$ in interval $[x_2, x_3]$\label{fig:Cub2}}
\end{figure}
  
\subsection{Bilevel Frequency Continuation}
\label{sect:FC}
It is well-known that (in frequency domain FWI), the  objective function $\phi$ is less oscillatory  for lower frequencies than higher (see, e.g., Figure \ref{fig:4}), but that a range of frequencies are required to reconstruct a range of different sized features in the image. 
  Hence  frequency continuation (i.e.,  optimising for lower frequencies first to obtain a
  good starting guess for higher frequencies)  
  is a standard technique for helping to  avoid spurious  local minima  (see, e.g., \cite{sirgue2004efficient}).
  Here we present a   bilevel version of the frequency-continuation approach that uses interacting frequency continuation on both the upper and lower levels, thus reducing  
  the risk of converging to spurious stationary points at either  level.  
  
  As motivation we consider  the following simplified  illustration. Figure \ref{fig:FCEx}
  shows the training model used  (i.e. $N_{m'}=1$), with three sources (given by the green dots) and three sensors (given by the red dots).   Here  $L =  1.25$ km, and $m = 1/c^2$ with maximum wavespeed $c$ varying between
  $2$ km/s and $2.1$ km/s.  The sensors are   constrained on a vertical line
and are placed   symmetrically   about the centre point. Here  there is only one  
  optimisation variable -- the distance $\Delta\in [0,L]$ depicted in Figure  \ref{fig:FCEx}.
For each of 1251 equally spaced values of   $\Delta \in [0,L]$,  
  (giving 1251 configurations of sensors) 
we   compute the value of  the upper level objective function   $\psi$ in \eqref{SOobj} and plot it in Figure \ref{fig:FCPsiHighLow}. This is repeated for two values of $\omega = \pi$ (blue dashed line) and  $\omega = 11 \pi$ (continuous red line). 
  We see that for $\omega = \pi$ the one local minimum is also the global minimum, but for
  $\omega = 11 \pi$ there are several local minima, but the global minimum is close to the global minimum of $\omega = \pi$. This illustration  shows the potential for  bilevel frequency continuation.    

  \begin{center}
  \begin{figure}[h!]
\begin{subfigure}[t]{0.38\textwidth}
\hspace*{0.5cm}
\includegraphics[scale=0.56]{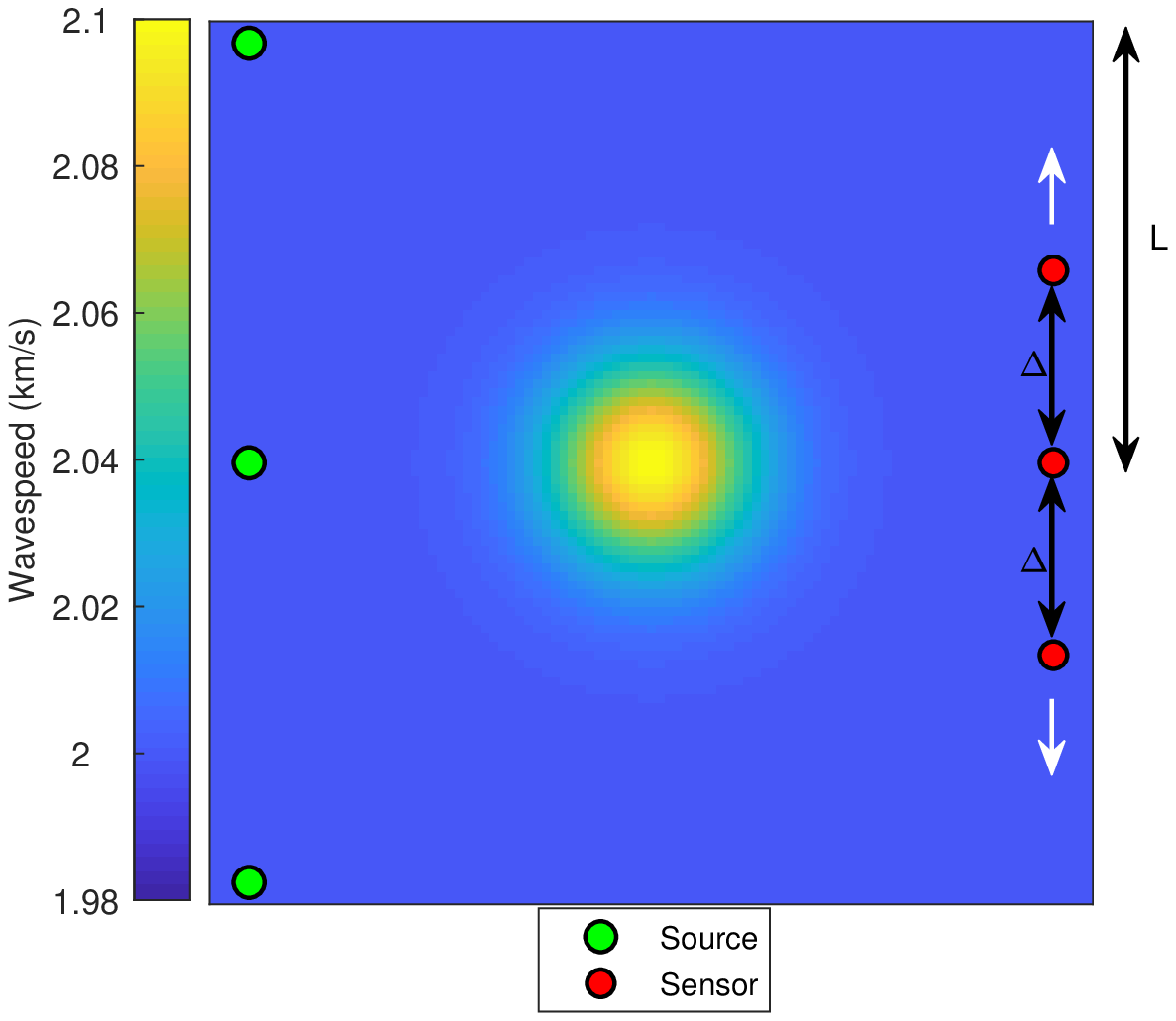}
\caption[Setup for $\psi$ plotting]{Setup for  plots of $\psi$.  \label{fig:FCEx}} %
\end{subfigure}
\hspace{1.9cm}
\begin{subfigure}[t]{0.28\textwidth}
\includegraphics[scale=0.56]{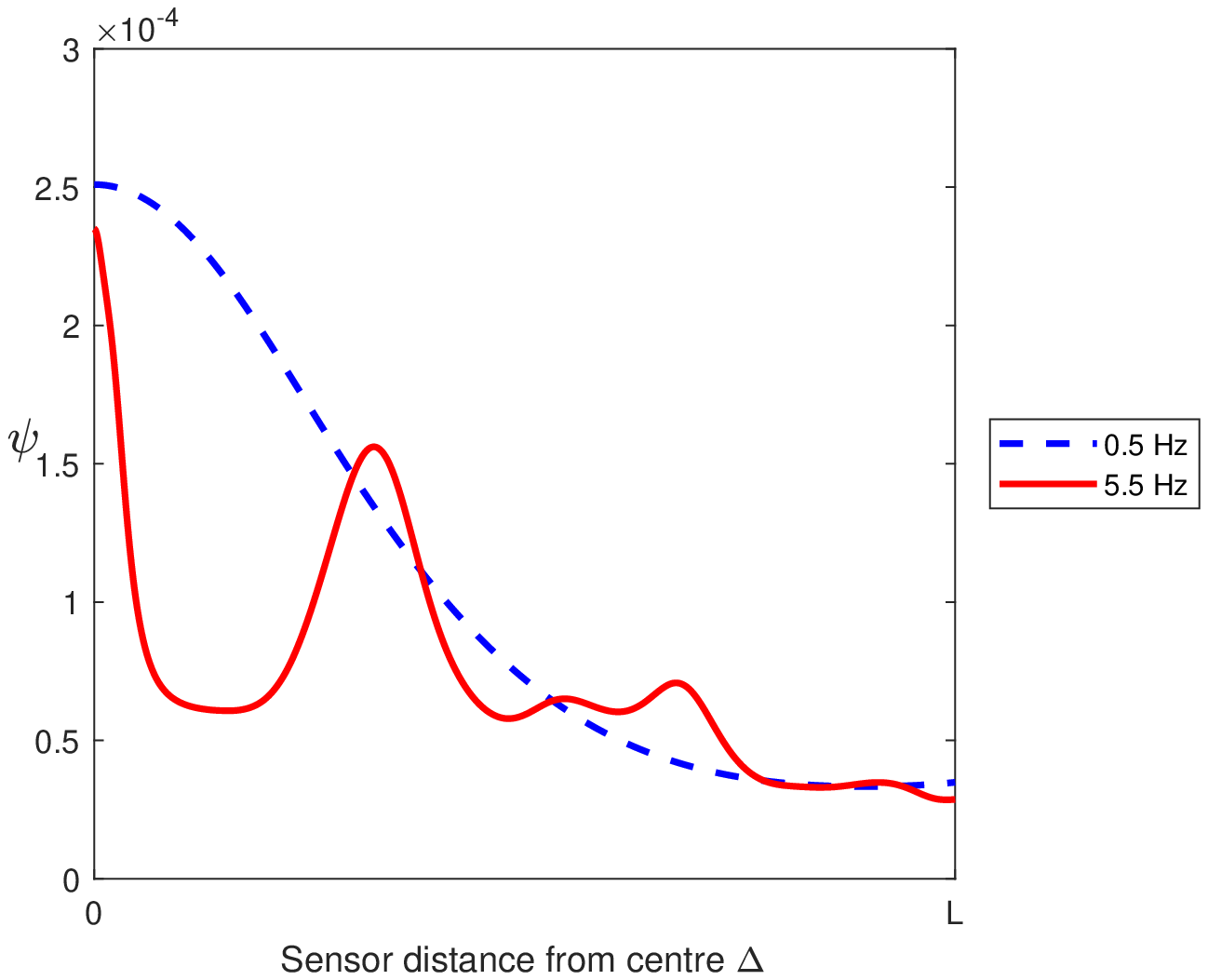}
\caption{Plots of  $\psi$. \label{fig:FCPsiHighLow}}
\end{subfigure}
\caption{\label{fig:4}}
\end{figure}
\end{center}
  
\vspace{-1.1cm} 
  These observations suggest  the following algorithm in which we arrange
  frequencies into groups of increasing size and solve  the  bilevel problem for each group using the solution obtained for the previous group.
  We summarise this in  Algorithm \ref{allg:fcb1} which, for simplicity, is presented
  for one training model only; {the notation `Bilevel Optimisation Algorithm($g_k$)'
  means solving the bilevel optimisation problem sketched in Figure \ref{fig:bilevel} on the $k$th frequency group $g_k$.} 

\smallskip 
\begin{algorithm}[h!]
\setstretch{1.15}
\caption{Bilevel Frequency Continuation} 
\begin{algorithmic}[1]
\State \textit{Inputs:}  $\cP_0$, $\bm_0$, frequencies $\{\omega_1< \omega_2<...<\omega_{N{\omega}}\} \in \mathcal{W}$,  $\bm'$
\State  Group frequencies into $N_f$ groups $\{g_1$, $g_2$,$\ldots$, $g_{N_f}\}$ 
\State \textbf{for} $k=1$ \textbf{to} $N_f$ \textbf{do}
\State \indent $[\cP_{\min}, \bm^{\rm FWI}] \gets$ Bilevel Optimisation Algorithm($g_k$)
\State \indent $\cP_0 \gets \cP_{\min}$
\State \indent $\bm_0 \gets \bm^{\rm FWI}$
\State \textbf{end for}
\State \textit{Output:} $\cP_{\min}$
\end{algorithmic}
\label{allg:fcb1}
\end{algorithm}

\smallskip 
\begin{remark}
  Algorithm \ref{allg:fcb1} is written for the  optimisation of  sensor positions $\cP$ only.
  The experiments in \cite[Section 5.1]{thesis} indicate that the objective function $\psi$  does not become  more oscillatory with respect to
    $\alpha$ for higher frequencies,
  and so the bilevel frequency-continuation approach is not required for optimising $\alpha$. Thus we recommend the user to begin optimising $\alpha$ alongside $\cP$ only in the final frequency group, starting with a reasonable initial guess for $\alpha$, to keep iteration numbers low. If one does not have a reasonable starting guess for $\alpha$, it may be beneficial to begin optimising $\alpha$ straight away in the first frequency group. 
\end{remark}


We illustrate the performance  of Algorithm \ref{allg:fcb1} in  Figure \ref{fig:FCsteps}.
  Each row of  Figure \ref{fig:FCsteps} shows a plot of the upper-level objective function $\psi$ for the problem setup in Figure \ref{fig:FCEx}, starting at a low frequency on row one, and increasing to progressively higher frequencies/frequency groups. In Subfigure (a) we represent a typical starting guess for the parameter to be optimised, $\Delta$, by an open red circle. Here  $\psi$  has one minimum, and the optimisation method finds it  straightforwardly -- see the full red circle in Subfigure (b).
We then progress through  higher frequency groups, using  the solution at the previous step as a starting guess for the next, allowing eventually convergence to the global minimum of the highest frequency group and avoiding the spurious local minima. 
  \begin{figure}
  \centering
  \begin{adjustbox}{minipage=\textwidth,scale=0.80}
\begin{subfigure}[b]{0.48\textwidth}
\hspace*{-1.1cm}
 \tikz[remember picture]\node[inner sep=0pt,outer sep=0pt] (a){\includegraphics[scale=0.55]{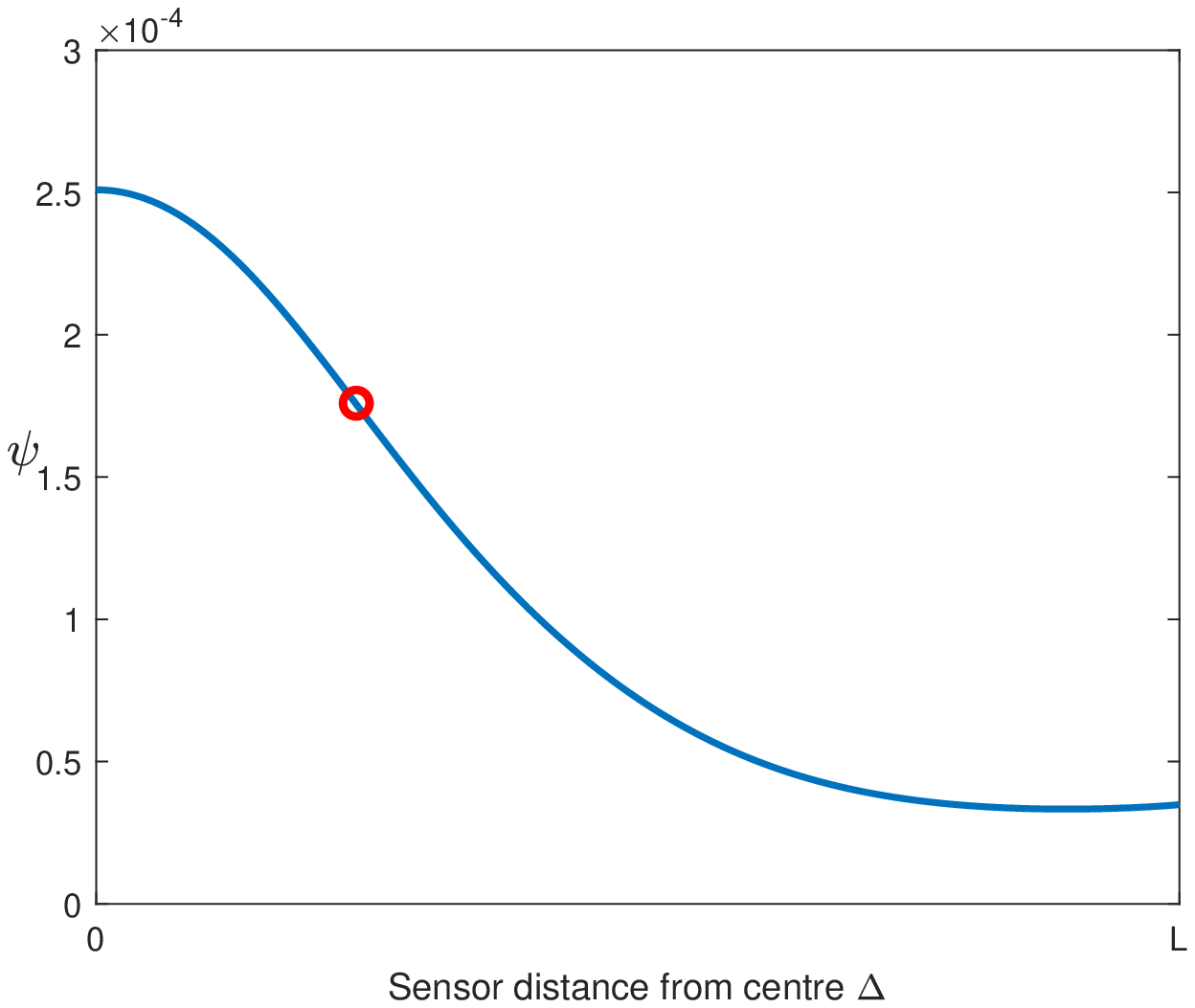}};
\caption{ 0.5 Hz}
\end{subfigure}
\hspace{1.1cm}
\begin{subfigure}[b]{0.48\textwidth}
\tikz[remember picture]\node[inner sep=0pt,outer sep=0pt] (b){\includegraphics[scale=0.55]{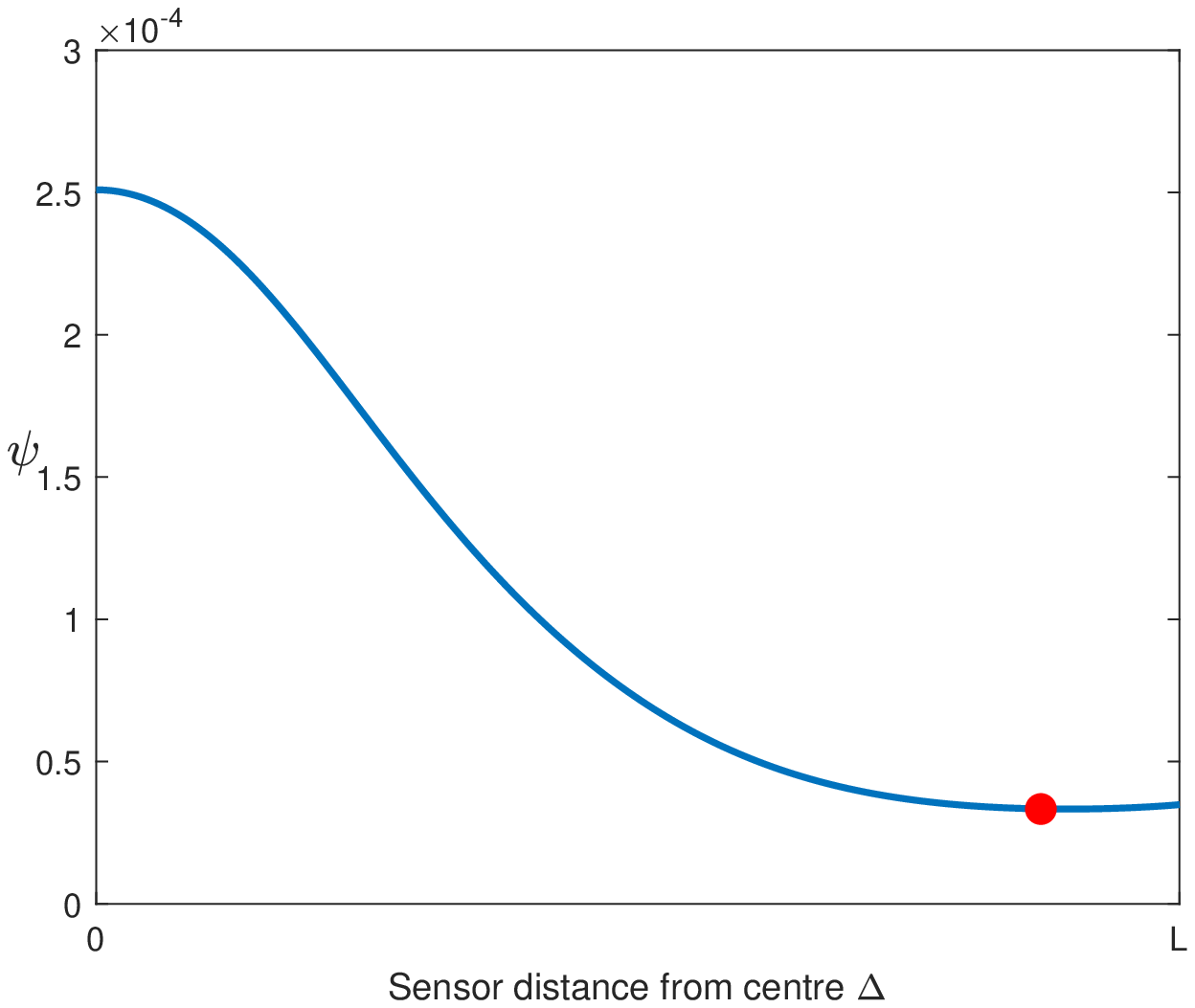}};
\caption{0.5 Hz }
\end{subfigure}\\
\begin{subfigure}[b]{0.48\textwidth}
\hspace*{-1.1cm}
 \tikz[remember picture]\node[inner sep=0pt,outer sep=0pt] (c){\includegraphics[scale=0.55]{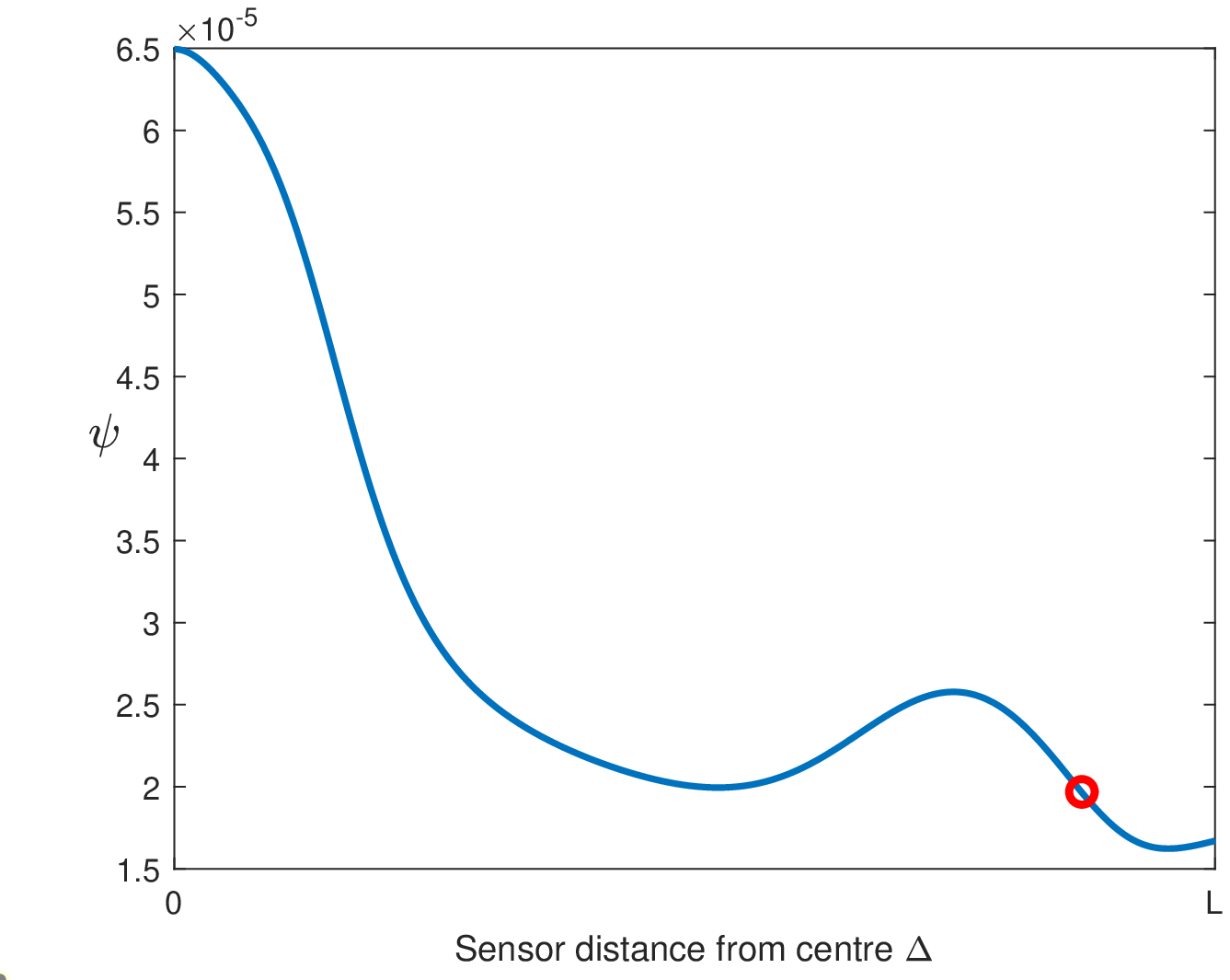}};
\caption{1 and 2 Hz group}
\end{subfigure}
\hspace{1.1cm}
\begin{subfigure}[b]{0.48\textwidth}
\tikz[remember picture]\node[inner sep=0pt,outer sep=0pt] (d){\includegraphics[scale=0.55]{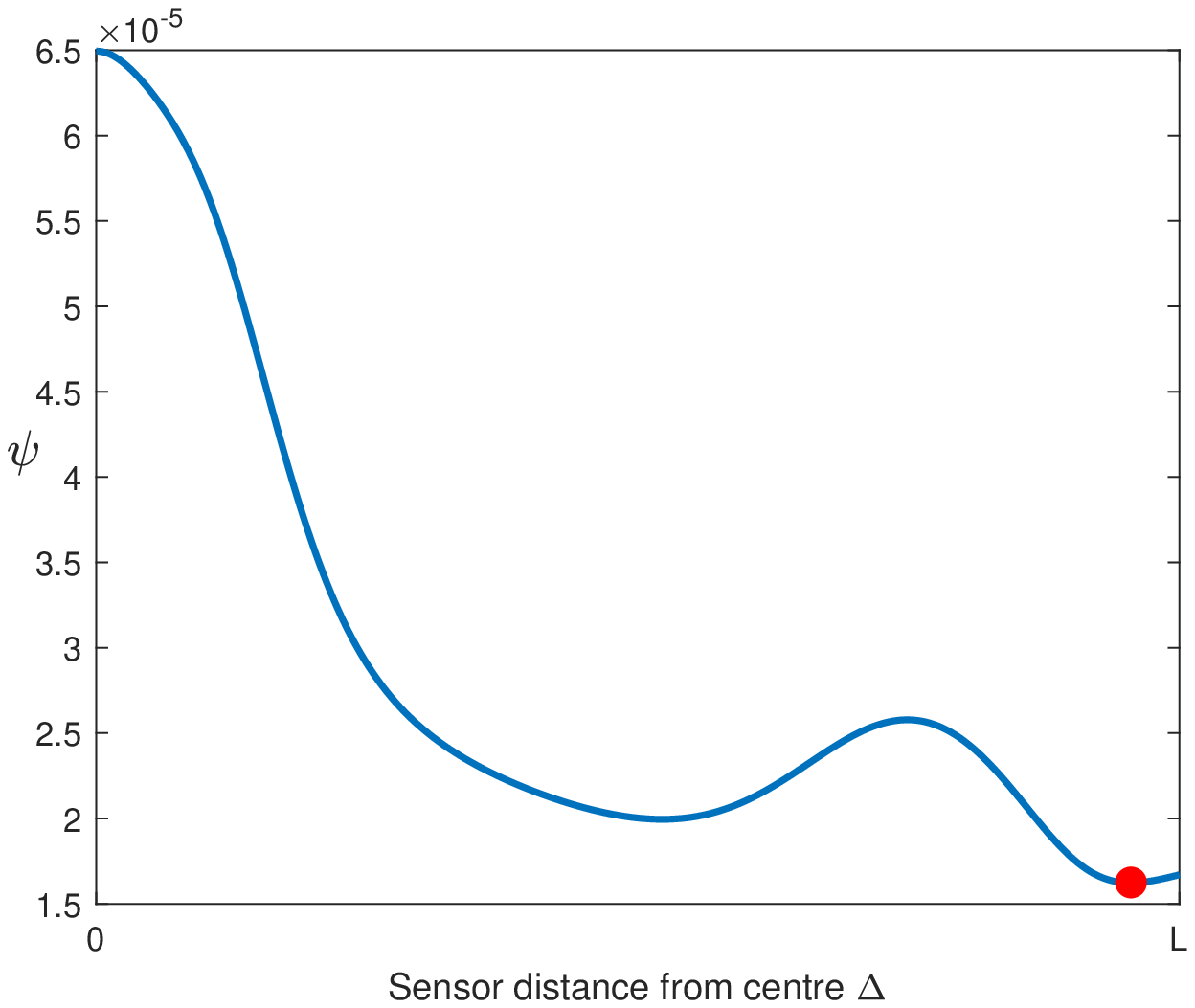}};
\caption{1 and 2 Hz group}
\end{subfigure}\\
\begin{subfigure}[b]{0.48\textwidth}
\hspace*{-1.1cm}
 \tikz[remember picture]\node[inner sep=0pt,outer sep=0pt] (e){\includegraphics[scale=0.55]{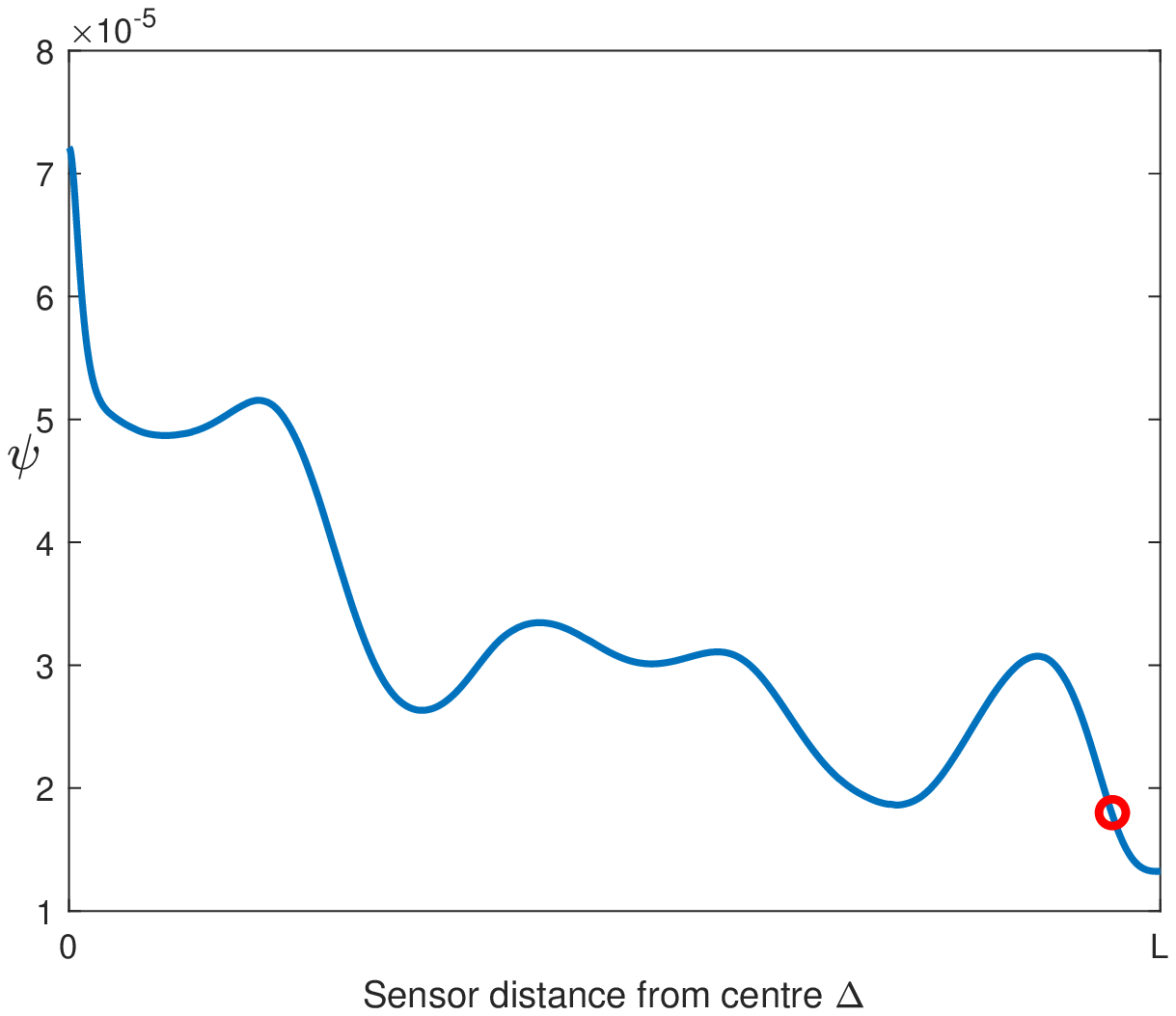}};
\caption{4 and 5 Hz group }
\end{subfigure}
\hspace{1.1cm}
\begin{subfigure}[b]{0.48\textwidth}
\tikz[remember picture]\node[inner sep=0pt,outer sep=0pt] (f){\includegraphics[scale=0.55]{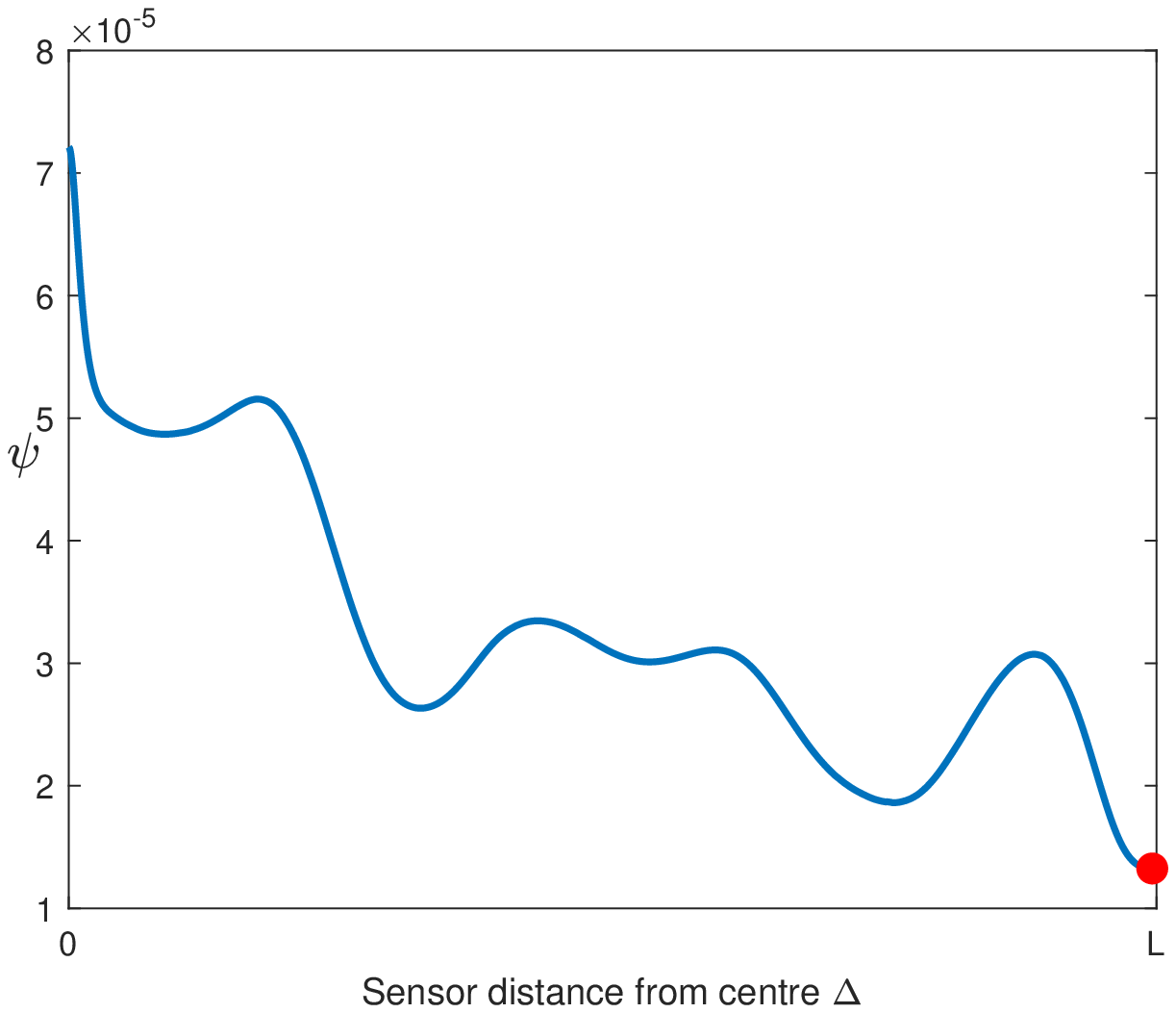}};
\caption{4 and 5 Hz group }
\end{subfigure}
\tikz[remember picture,overlay]\draw[line width=1pt,-stealth,black] ([xshift=1mm]a.east) -- ([xshift=-1mm]b.west)node[midway,above,text=black,font=\LARGE\bfseries\sffamily] {};
\tikz[remember picture,overlay]\draw[line width=1pt,-stealth,black] ([xshift=1mm]c.east) -- ([xshift=-1mm]d.west)node[midway,above,text=black,font=\LARGE\bfseries\sffamily] {};
\tikz[remember picture,overlay]\draw[line width=1pt,-stealth,black] ([xshift=1mm]e.east) -- ([xshift=-1mm]f.west)node[midway,above,text=black,font=\LARGE\bfseries\sffamily] {};
\tikz[remember picture,overlay]\draw[line width=1pt,-stealth,black] ([xshift=1mm]b.south west) -- ([xshift=-1mm]c.north east)node[midway,above,text=black,font=\LARGE\bfseries\sffamily] {};
\tikz[remember picture,overlay]\draw[line width=1pt,-stealth,black] ([xshift=1mm]d.south west) -- ([xshift=-1mm]e.north east)node[midway,above,text=black,font=\LARGE\bfseries\sffamily] {};
\caption[Illustration of the bilevel frequency continuation approach.]{Plots of the upper-level objective function in an illustration of the bilevel frequency continuation approach. The symbol \tikz{ \draw[red](0,0) circle (0.7ex);} denotes a starting guess, and the symbol \tikz{ \draw[red, fill=red](0,0) circle (0.7ex);} denotes a minimum.\label{fig:FCsteps}}
 \end{adjustbox}
\end{figure}

\subsection{Preconditioning the Hessian }
\label{sect:PCG}

 While the solutions of Hessian systems are not required in the
quasi-Newton method for the lower-level problem, such solutions are required  to compute the gradient of
  the upper level objective function (see \eqref{rho}).   We solve these systems using a preconditioned conjugate gradient iteration, without explicilty forming the Hessian.
  As  explained in Section \ref{sec:appendix}, 
  ``adjoint-state'' type arguments   can be applied  to efficiently compute  matrix-vector multiplications with $H$; see also \cite[Section 3.2]{metivier2017full}. 
  In this section we discuss preconditioning techniques for the system \eqref{rho}.
  Our proposed preconditioners are:
\begin{itemize}
\item \textbf{Preconditioner 1:} $$P_1^{-1}, \quad \text{where} \quad P_1 =H(\bm^{\rm FWI}(\cP_0,\alpha_0, \bm'), \cP_0, \alpha_0),$$
  i.e., the full Hesssian at some chosen design parameters $\cP_0$ and $\alpha_0$.

  During the bilevel algorithm the design parameters $\cP, \alpha$ (and hence $\bm^{\rm FWI}(\cP, \alpha, \bm')$)  may  move away from the initial choice $\cP_0, \alpha_0$ and $\bm^{\rm FWI}(\cP_0,\alpha_0, \bm')$ 
  and  the preconditioner may need to be recomputed
 using updated $\cP, \alpha$ to ensure its effectiveness.  
 Here we  recompute the preconditioner at the beginning of each new  frequency group.
The cost of computing this preconditioner, is relatively high -- costing   $M$ Helmholtz  solves, for each source, frequency, and training model.   
 \item \textbf{Preconditioner 2:}
\begin{equation}\label{P2}
P_2^{-1}, \quad \text{where} \quad P_2=\Gamma(\alpha_0, \mu).
\end{equation}
 This is cheap to compute as no PDE solves are required. In addition, this preconditioner is independent of the sensor positions $\cP$, training models $\bm'$, FWI reconstructions $\bm^{\rm FWI}$  and frequency. When optimising sensor positions alone,  the preconditioner therefore only needs to be computed once at the beginning of the bilevel algorithm. Even
 when optimising $\alpha$, this preconditioner does not need to be recomputed since it turns out to remain
 effective even when $\alpha$ is no longer near its initial guess. 
\end{itemize}
  
  Although we write the preconditioners above as the inverses of certain matrices, these are not
  computed in practice, rather the Cholesky factorisation of the relevant matrix is computed and used  to compute the action of the inverse.  %


  To test the preconditioners we consider the solution of \eqref{rho} in the following situation.   We take a training model  and configuration of sources and sensors as  shown in Figure \ref{fig:preconGT} and compute $\bm^{\rm FWI}$,  using synthetic data, avoiding  an inverse crime by computing the data and solution using different grids. 
  
\begin{figure}[h!]
\hspace*{2.5cm}
\includegraphics[scale=0.72]{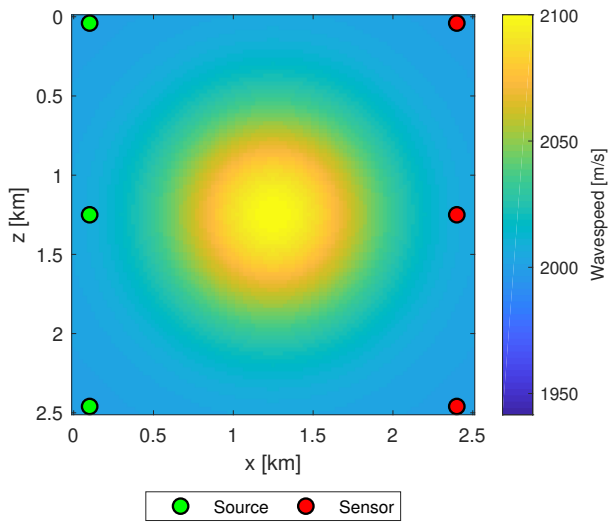}
\vspace{-0.75cm}
\caption{Ground truth model and acquisition setup used for preconditioining experiments. \label{fig:preconGT}}
\end{figure}

    We consider the CG/PCG method to have converged if  
$\|\textbf{r}_n\|_2/\|\textbf{r}_0\|_2\leq 10^{-6}$, where $\textbf{r}_n$ denotes the residual at the $n$th iteration and $\textbf{r}_0$ denotes the initial residual.

  \par The aim of this experiment is to demonstrate the reduction in 
the number of iterations for PCG to converge, compared to the number of CG iterations (denoted $N_i$ here). 
We denote the number of iterations taken using Preconditioner 1 as $N_i^{P_1}$ and the number of iterations taken using Preconditioner 2 as $N_i^{P_2}$. As we have explained, the preconditioner $P_1$ depends on the sensor positions. Therefore we test two versions of $P_1$ -- one where the sensor positions $\cP_0$ are close to the current sensor positions $\cP$ (i.e. close to those shown in Figure \ref{fig:preconGT}) and one where the sensor positions $\cP_0$ are far from $\cP$. We denote these preconditioners as $P_{1_{near}}$ and $P_{1_{far}}$ 
 and their iterations counts as $N_i^{{P_1}_{near}}$ and $N_i^{{P_1}_{far}}$, respectively.  We display these `near' and `far' sensor setups in Figure \ref{fig:precongoodbad} (a) and (b), respectively. 
\begin{figure}[h!]
\hspace*{0.25cm}
\begin{subfigure}[t]{0.48\textwidth}
\hspace{0.2cm}
\includegraphics[scale=0.48]{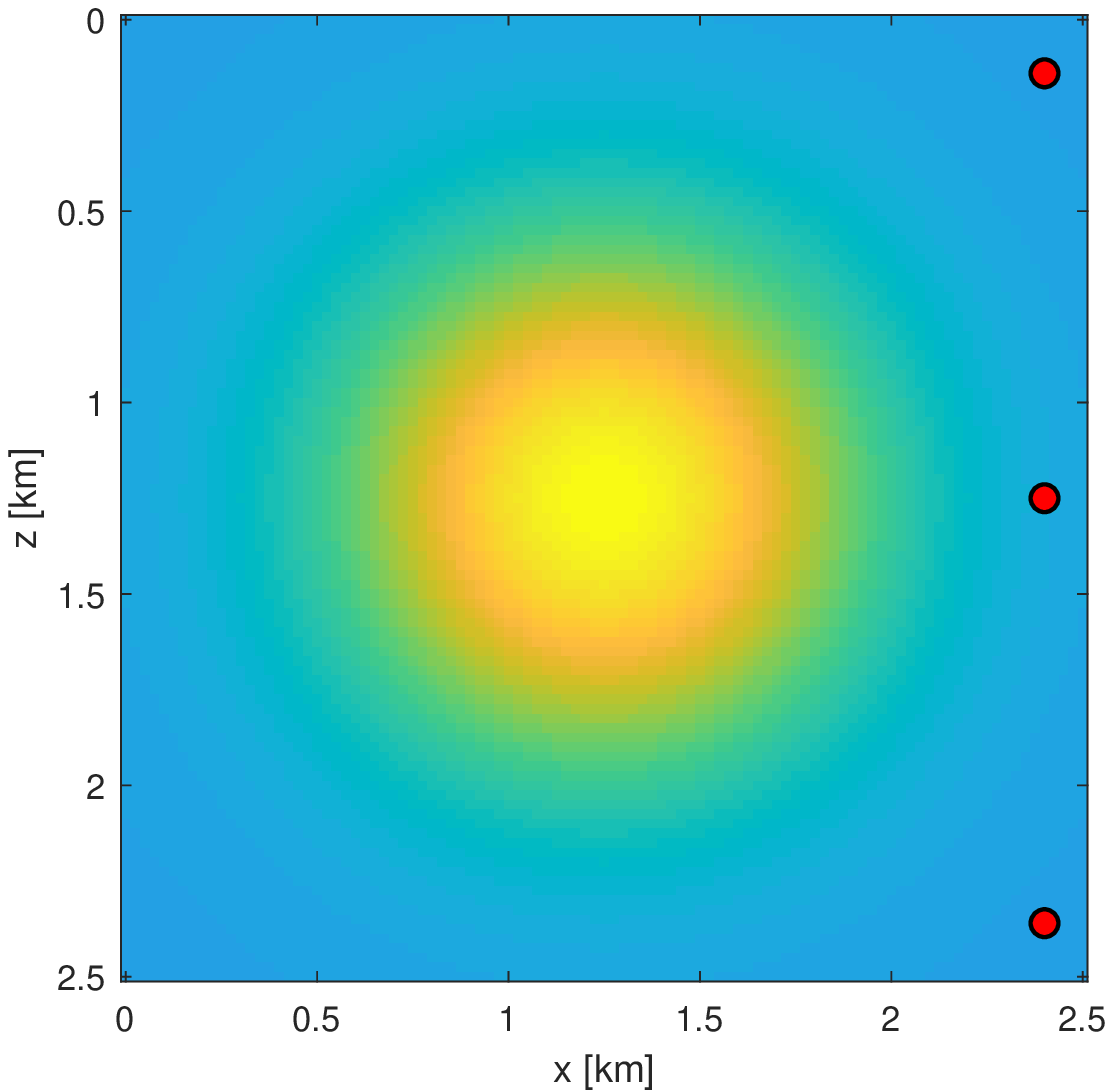}
\caption{Sensor positions `near' $\bp$ }
\end{subfigure}
\hspace{-0.9cm}
\begin{subfigure}[t]{0.48\textwidth}
\hspace*{0.15cm}
\includegraphics[scale=0.48]{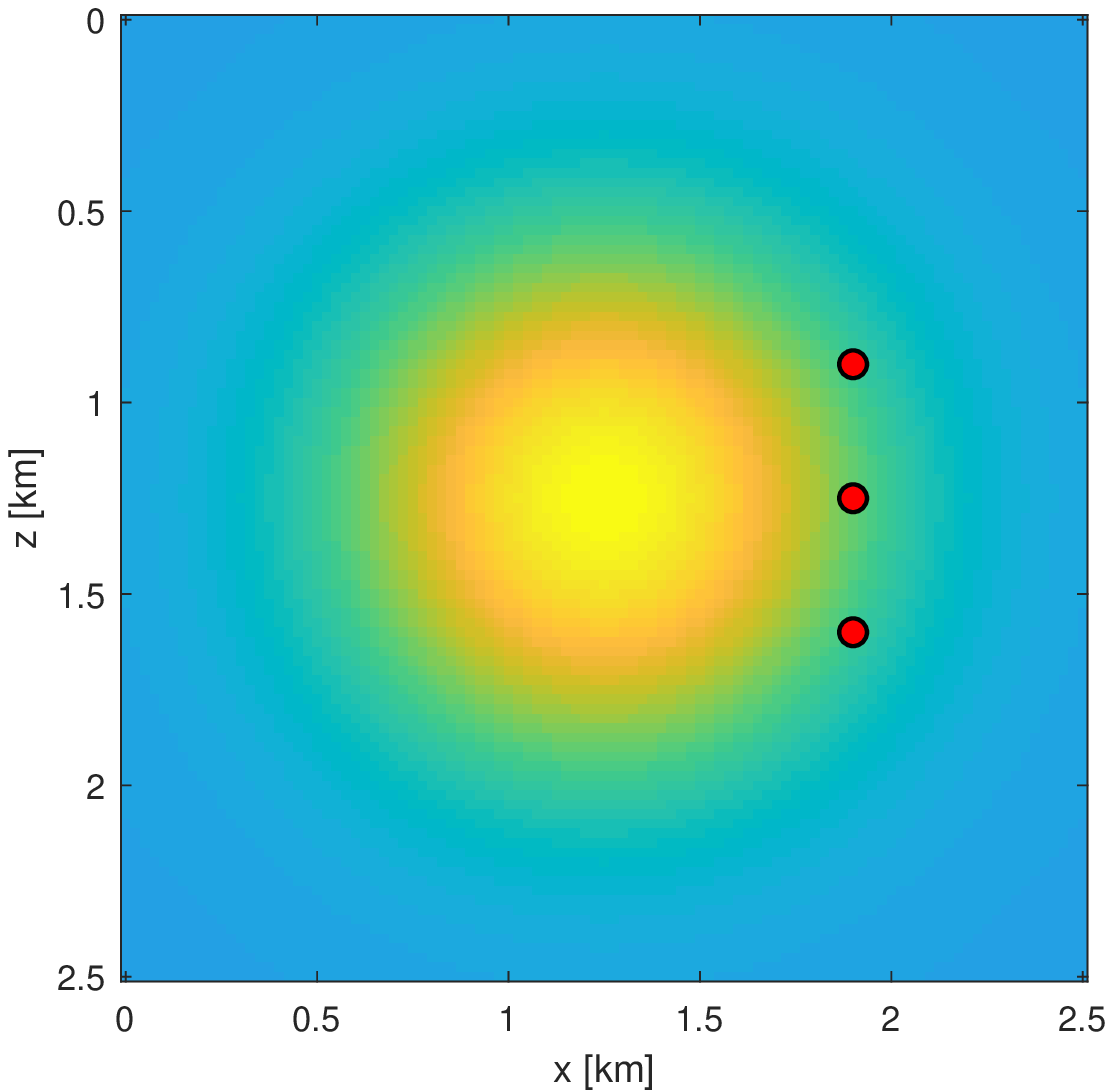}
\caption{Sensor positions `far' from $\bp$   }
\end{subfigure}
\caption{Sensor positions used to compute the preconditioner $P_1$. This preconditioner is used to solve \eqref{rho} for the problem shown in Figure \ref{fig:preconGT}. \label{fig:precongoodbad}}
\end{figure}
\par We vary the regularisation  parameter $\alpha$ and record the resulting number of CG/PCG iterations taken to solve \eqref{rho}. 
Table \ref{tab:alphaP} shows the number of iterations needed to solve \eqref{rho} when using the PCG method, as well as the percentage reduction in iterations (computed as the reduction in iterations divided by the original non-preconditioned number of iterations, expressed as a percentage and rounded to the nearest whole number). We see that preconditioner $P_1$ is very effective at reducing the number of iterations when the sensors $\cP_0$ are close to $\cP$. The number of iterations are reduced by between 85-96$\%$. When $\cP_0$ is not close to $\cP$ however, $P_1$ is not as effective. In this case, the PCG method is even worse than the CG method when $\alpha$ is small, but improves as $\alpha$ is increased, reaching approximately a 89$\%$ reduction in number of iterations at best. This motivates the update in this preconditioner as the sensors move far from their initial positions. The preconditioner $P_2$ produces a more consistent reduction in the number of iterations, ranging from 
71-91$\%$. 
\begin{table}[h!]
\begin{center}
\setstretch{1.15}
\begin{tabular}{|| c |c|| c| c || c | c || c |c ||} 
 \hline  
 $\alpha$ & $N_i$ & $N_i^{{P_1}_{near}}$& $\%$ &$N_i^{{P_1}_{far}}$& $\%$ & $N_i^{P_2}$& $\%$ \\ [0.5ex] 
 \hline\hline
 0.5 & 153  & 21 &  86 $\%$ & 181 & -18 $\%$  &36   & 76  $\%$ \\ 
 \hline
 1 & 132 & 17 &  87 $\%$  &  136 & - 3 $\%$ & 35 &  73 $\%$\\
 \hline
 5 & 127 & 11  &  91 $\%$ & 62 &  51 $\%$ & 29 &   77 $\%$\\
 \hline
 10 & 137 & 9 &  93 $\%$ & 46 &   66 $\%$ & 26 &  81 $\%$\\
 \hline
 20 & 143 & 8 &  94 $\%$  &32 &  78 $\%$  &21 &  85 $\%$\\ 
  \hline
 50 & 158  & 7 &   96 $\%$ &  24 &  85 $\%$ & 17&   89 $\%$\\ 
  \hline
 100 &162  & 6 &  96 $\%$ & 16 &    90 $\%$  &14 &  91 $\%$\\ [1ex] 
 \hline
\end{tabular}
\end{center} 
\caption{Effect of varying Tikhonov regularisation weight $\alpha$ on solving \eqref{rho} using PCG, using two versions of the preconditioner $P_1$ and the preconditioner $P_2$. The convex parameter is constant at $\mu=10^{-8}$. \label{tab:alphaP}}
\end{table}
\par These iteration counts must then be considered in the context of the overall cost of solving the bilevel problem in \cite[Section 5.2.2.1]{thesis}. In this cost analysis we show that, in general, $P_2$ is a more cost effective preconditioner than $P_1$ when $M$ is large.

\subsection{Parallelisation}
  \label{sect:par}
 Examination of Algorithm  \ref{alg:box} reveals that  its parallelisation   over training models is straightforward.   For each  $\bm' \in \mathcal{M}'$, the lower-level solutions $\bm^{\rm FWI}(\bp, \bm')$
  can be computed  independently. 
  Then, from the loop beginning at Step 2 of Algorithm \ref{alg:box}, the main work in computing the gradient of $\psi$ is also independent of $\bm'$, 
  with only the finally assembly of the gradient (by \eqref{sograda} \eqref{gradALT}) having  to take place outside of this parallelisation.  The algorithm was parallelised using the \verb+parfor+ function in Matlab. \cite[Section 5.3]{thesis} demonstrated, using strong and weak scaling, that the problem scaled well using up to $N_{m'}$ processes.

\section{Appendix  --  Computations with the gradient and Hessian of $\phi$}\label{sec:appendix}

  \subsection{Gradient of  $\phi$}
  \label{subsec:gradphi}  Recall the formula for the   gradient $\nabla \phi$ given in \eqref{graddirect_here}. 
 Since we use a variant of the BFGS quasi-Newton algorithm to minimise $\phi$, the cost of this algorithm is dominated by the computation of  $\phi$ and $\nabla \phi$. 
  While efficient methods for computing   these two quantities (using an `adjoint-state' argument) are known, we state them again briefly here  
                         since (i) we do not know references where this procedure is written down at the PDE
                         (non-discrete) level  and (ii) the development  motivates our  approach for computing $\nabla \psi$ given in Section \ref{sect:gradient}.
  A review of the adjoint-state method, and an alternative derivation of the gradient from a Lagrangian persepctive, are provided in \cite{plessix2006review}, see also \cite{metivier2017full}.
  
\begin{theorem}[\textbf{Formula for $\nabla \phi$}] \label{th:grad} \ 
  For  models $\bm, \bm'$, sensor positions  $\cP$, regularisation
  parameter $\alpha$ and  each $k = 1, \ldots , M$, 
       \begin{align}   
       \frac{\partial \phi }{\partial m_k}(\bm,\cP, \alpha, \bm')  
         \ & = \
         -  \Re \sum_{s \in \mathcal{S}} \sum_{\omega \in \mathcal{W}} \bigg( \cG_{\bm, \omega} (\beta_k u(\bm, \omega,s) ),
          \lambda(\bm, \cP, \omega, s, \bm')   \bigg)_ {\Omega \times \partial \Omega}  \noindent + \Gamma(\alpha,\mu) \bm , 
\label{I4}
     \end{align}
       where, for each  $\omega \in \mathcal{W}$  and $s\in \cS$, $\lambda $ is the adjoint solution:    
       \begin{align} 
        \lambda(\bm,\cP, \omega,s, \bm') =       \mathscr{S}_{\bm, \omega}^* \left( \begin{array}{l} \mathcal{R}(\cP)^*\beps(\bm, \cP, \omega,s, \bm')\\0 \end{array} \right).  
\label{lambda}       
       \end{align}          
 \end{theorem}
\begin{proof}
Here we adopt the convention in Notation \ref{not:conv} and  use \eqref{I1} to write \eqref{graddirect_here} as 
\begin{align}
\frac{\partial \phi }{\partial m_k}(\bm, \cP, \alpha) \ & = \ - \Re  \left\langle \mathcal{R}(\cP) \frac{\partial u }{\partial m_k}(\bm) , \beps(\bm, \cP) \right\rangle + \Gamma(\alpha, \mu)\bm  \nonumber \\
 \ &= \ - \Re \left(\frac{\partial u }{\partial m_k}(\bm) , \mathcal{R}(\cP)^* \beps(\bm, \cP) \right)_\Omega + \Gamma(\alpha, \mu) \bm.    
\label{I2}  
\end{align}
Using  \eqref{oneMgrad2a}, the first term on the right-hand side of \eqref{I2} becomes
\begin{align}
 & - 
                                                     \Re \bigg( (1,0) \, \mathscr{S}_{\bm, \omega} \cG_{\bm,\omega} (\beta_k u(\bm)), \,  
    \mathcal{R}(\cP)^* \beps(\bm, \cP) \bigg)_\Omega \nonumber \\
  & \quad \quad 
   = \  - 
    \Re \left( \mathscr{S}_{\bm, \omega} \cG_{\bm,\omega} (\beta_k u(\bm)), \,  
  \left(\begin{array}{l}
    \mathcal{R}(\cP)^* \beps(\bm, \cP)  \\ 0\end{array}\right)\right)_{\Omega\times \partial \Omega}\nonumber \\
  & \quad \quad = \  - 
      \Re \left( \cG_{\bm,\omega} (\beta_k u(\bm)), \,   \lambda(\bm, \cP) \right)_{\Omega\times \partial \Omega},
  \label{I5}
\end{align}
where we used \eqref{Green}. Combining \eqref{I2} and \eqref{I5}  yields the result.
\end{proof}
Thus, given  $\bm, \cP, \alpha, \bm'$  and assuming  $u(\bm',s,\omega)$  is known for all $s, \omega$, to find  $\nabla \phi$, we need only two Helmholtz  solves for each  $s$ and $\omega$, namely a  
solve to obtain $u(\bm, \omega,s)$ and an  adjoint
solve to obtain $\lambda(\bm, \cP, \omega, s, \bm')$.
 Algorithm  \ref{alg:box1}  presents the steps involved. 

\smallskip 
\begin{algorithm}[h!]
\setstretch{1.15}
\caption{Algorithm for Computing $\phi, \ \nabla \phi$} 
\begin{algorithmic}[1]
\State \textit{Inputs:}  $\cP$, 
     $\mathcal{W}$,  $\cS$, $\bm, \bm'$, 
       $\alpha, \mu$ and $\textbf{u}(\bm', s, \omega) = (1,0)\mathscr{S}_{\bm', \omega}(\delta_s)$
\State \textbf{for} $\omega \in \mathcal{W}$, $s \in \mathcal{S}$ \textbf{do}
\State \indent Compute the primal wavefield   $\textbf{u}(\bm, s, \omega) = (1,0)\mathscr{S}_{\bm, \omega}(\delta_s)$
\State \indent Compute $\beps(\bm, \cP, \omega,s,\bm')$ by \eqref{resdbi}
\State \indent Compute the adjoint wavefield  $\boldsymbol{\lambda}(\bm,\cP, \omega,s, \bm')$ from \eqref{lambda}
\State \textbf{end for}
\State Compute $\phi$ from \eqref{newphi}, \eqref{resdbi}. 
\State Compute $\nabla \phi$ from \eqref{I4}
\end{algorithmic}
\label{alg:box1}
\end{algorithm}

\medskip

\noindent{\bf Computing the gradient. } \  
  Using the discretisation scheme in Section \ref{subsec:num_for} and  restricting to the computation of a single term in  the sum on the right-hand side,  the numerical analogue of \eqref{I4} is as follows.
  
The vector $\blambda \in \mathbb{C}^M$ representing $\lambda(\bm, \cP, \omega, s, \bm')$ is computed by solving
$$ A(\bm, \omega)^* \blambda = \sum_{\bp\in \cP} \eps_{\bp}\bfe_{\bp}  \ , \quad \text{where} \quad \eps_{\bp} = (R(\cP) \left( \bu(\bm',\omega,s) - \bu(\bm, \omega,s)\right))_{\bp}, $$
and $R(\cP)$ denotes the sliding cubic approximation (as described in Section \ref{subsec:restrict}) and $\bfe_{\bp}$ represents the numerical realisation of the delta function situated at $\bp$. Then the inner product in \eqref{I4} is approximated by   \begin{align*}
&-\Re \left(\omega^2 d_k u_k \overline{\lambda_k}\right), &\text{if} \quad 
k \quad \text{is an interior node},\\
&- \Re \left( (\omega^2 d_k u_k  + b_k \frac{\ri \omega}{2 \sqrt{m_k}} u_k) \overline{\lambda_k}\right),   &\text{if} \quad k \quad \text{is a boundary node}. 
  \end{align*}

\subsection{ Matrix-vector multiplication with the Hessian}
\label{subsec:MVHess} Differentiating \eqref{graddirect_here} with respect to $m_j$,  we  obtain the    
 Hessian $H$  of $\phi$,  
\begin{align*}
  H(\bm, \cP, \alpha, \bm') =H^{(1)}(\bm, \cP)+H^{(2)}(\bm, \cP,\bm')+ \Gamma(\alpha, \mu), 
\end{align*}
with $H^{(1)}$ and $H^{(2)}$ defined by
\begin{align}
\left(H^{(1)}(\bm, \cP)\right)_{j, k}&\ =\ \Re \sum_{s \in \mathcal{S}} \sum_{\omega \in \mathcal{W}}  \left\langle \mathcal{R}(\cP) \frac{\partial u }{\partial m_j}(\bm, \omega, s) \, , \,    \mathcal{R}(\cP) \frac{\partial u }{\partial m_{k}}(\bm, \omega, s) \right \rangle , 
\label{h1}\\
\left(H^{(2)}(\bm, \cP, \bm')\right)_{j,k}&\ =\ - \Re \sum_{s \in \mathcal{S}} \sum_{\omega \in \mathcal{W}}    \left\langle  \mathcal{R}(\cP) \frac{\partial^2 u }{\partial m_k\partial m_{j} }(\bm, \omega, s)\, , \,  \beps(\bm, \cP, \omega, s, \bm') \right\rangle . 
\label{h2}
\end{align}
Observe that  $H^{(1)}$ is symmetric  positive semidefinite, while 
                                            $H^{(2)}$ is symmetric  but possibly indefinite.
                                            
In the following two lemmas we obtain efficient formulae for computing $H^{(1)} \tbm$ and $H^{(2)} \tbm$   for any $\tbm \in \mathbb{R}^M$. These make use of `adjoint state' arguments. Analogous formulae in the discrete case are given in \cite{metivier2017full}.                   Before we begin, for any  $\tbm=(\tm_1,\dots,\tm_M) \in \mathbb{R}^M$,
we define 
\begin{align} \label{deftm} \tm = \sum_k \tm_k \beta_k. \end{align}                              
 
  \begin{lemma}[Adjoint-state formula for multiplication by $H^{(1)}(\bm, \cP)$]\label{thm:H1}\ \ 
    For any $\bm, \bm' \in \mathbb{R}^M$, $\omega \in \mathcal{W}$ and $s \in \cS$, let
    $u(\bm, \omega,s)$ be the wavefield defined by \eqref{defu},   and set 
\begin{align} \label{defv}  v(\bm, \omega, s, \tbm) = (1 ,0)\,  \mathscr{S}_{\bm, \omega}\,  \cG_{\bm, \omega} \left(\begin{array}{l} \tm u(\bm, \omega,s)\\ \tm u(\bm, \omega,s)\vert_{\partial \Omega}\end{array}  \right) .  \end{align} Then, for each 
$j = 1,\ldots,M$, 
\begin{align} 
    (H^{(1)} (\bm, \cP, \bm') \tbm)_j & \quad =     \Re  \sum_{s \in \mathcal{S}}
 \sum_{\omega \in \mathcal{W}}  \left(\cG_{\bm, \omega}\left(\begin{array}{l} \beta_j u(\bm, \omega, s) \\ \beta_j u(\bm, \omega, s)\vert_{\partial \Omega} \end{array} \right), \mathscr{S}_{\bm, \omega}^* \left(\begin{array}{l}\cR(\cP)^* \cR(\cP) v(\bm,  \omega,s, \tbm)\\0\end{array} \right)  \right)_{\Omega\times \partial \Omega} . 
 \label{H1v}
 \end{align}
    \end{lemma} 

                                      \begin{proof}
                                        Using the convention in Notation \ref{not:conv},
                                        the definition of $\tm$,  the linearity of $\mathscr{S}_{\bm, \omega}$  and $\cG_{\bm, \omega}$ and then   \eqref{oneMgrad2a}, we can write  
                                        \begin{align}  v(\bm, \tbm) = (1,0) \, \sum_{k}  \,  \mathscr{S}_{\bm,\omega} \, \cG_{\bm, \omega}  \left( \begin{array}{l} \beta_k u(\bm)  \\ \beta_k u(\bm)\vert_{\partial \Omega}\end{array} \right) \, \tm_{k} 
                                          = \sum_{k}  \frac{\partial u }{\partial m_k}(\bm)\, \tm_{k}.         \label{I8} 
    \end{align}
Then, using \eqref{h1}, \eqref{I8}, and then \eqref{I1},     we obtain
 \begin{align*}
   (H^{(1)}(\bm, \cP, \bm') \tbm)_j &= \Re \sum_k \left\langle \mathcal{R}(\cP) \frac{\partial u }{\partial m_j}(\bm)\, ,\,    \mathcal{R}(\cP) \frac{\partial u }{\partial m_{k}}(\bm) \right \rangle  \tm_k\\
                                    &= \Re  \left\langle \mathcal{R}(\cP) \frac{\partial u }{\partial m_j}(\bm) \, ,   \,  \mathcal{R}(\cP) v(\bm, \bm')\right\rangle = \Re  \left( \frac{\partial u }{\partial m_j}(\bm) \, ,   \,  \mathcal{R}(\cP)^* \mathcal{R}(\cP) v(\bm, \bm')\right)_{\Omega}.
 \end{align*}
 Then substituting for $\partial u / \partial m_j $ using \eqref{oneMgrad2a}, and  proceeding analogously to  \eqref{I5}, we have
   \begin{align*}
     (H^{(1)}(\bm, \cP, \bm') \tbm)_j & = \Re
    \left( 
    \cG_{\bm, \omega}\left(\begin{array}{l} \beta_j u(\bm)\\ \beta_j u(\bm) \vert_{\partial \Omega}\end{array}  \right) \, , \, \mathscr{S}_{m,\omega}^* \left(\begin{array}{l} \mathcal{R}(\cP)^* \mathcal{R}(\cP) v(\bm, \bm')\\
      0 \end{array} \right) \right)_{\Omega \times \partial \Omega}.  
                            \end{align*}
  Recalling Notation \ref{not:conv}, this completes the proof. 
\end{proof}

    This lemma shows that (for each $\omega,s$) 
   computing $H^{(1)}\tbm$ requires only three Helmholtz solves, namely those required to compute  $u$, $v$  and, in addition,  the second argument in the inner products  \eqref{H1v}. 
    Computing $H^{(2)}\tbm$
     is a bit more complicated.  For this we need the following  formula for the second  derivatives of $u$ with respect to the model.  
       \begin{align} 
    \left( \begin{array}{l} \frac{\partial^2 u}{\partial m_j \partial m_k}\\ \frac{\partial^2 u}{\partial m_j \partial m_k}\vert_{\partial \Omega}\end{array} \right)  & =    \mathscr{S}_{\bm, \omega} \, \left[ \cG_{\bm, \omega} \left(\begin{array}{l} u_{j,k}\\ u_{j,k}\vert_{\partial \Omega} \end{array} \right) - \left(
\begin{array}{l} 0\\ \left(\frac{\ri \omega}{4 m^{3/2}}\right) \beta_j\beta_k u\vert_{\partial \Omega} \end{array} \right)\right],                                                                                             
  \label{Jac_exp} 
  \end{align}
  where   \begin{align}\label{defujk}  u_{j,k} (\bm,\omega, s)=  \beta_j \frac{\partial u}{\partial m_k}(\bm,\omega, s) + \beta_k \frac{\partial u}{\partial m_j}(\bm,\omega, s).  \
                                                           \end{align}
The formulae \eqref{Jac_exp}, \eqref{defujk} are  obtained by writing out \eqref{oneMgrad2a} explicitly and differentiating with respect to $m_j$.   
Analogous second derivative  terms appear  in, e.g., \cite{pratt1998gauss} and \cite{metivier2012second}. However these are presented  in the context of a forward problem consisting of solution of a linear algebraic system and so the detail of the  PDEs being  solved at each step is less explicit than here.  
  \begin{lemma}[Adjoint-state formula for $H^{(2)}(\bm, \cP, \bm') \,  \tbm$]\label{thm:H2}
\noindent For each  $j = 1, \ldots , M$, 
  \begin{align}  
    \left(H^{(2)}(\bm,\cP,\bm') \tbm \right)_{j}           & = -  \Re
                                      \sum_{s \in \mathcal{S}} \sum_{\omega \in \mathcal{W}}
                                                             \left[     \left(\cG_{\bm,\omega} \left(\begin{array}{l}\beta_j v(\bm, \omega,s,\bm')\\ \beta_j v(\bm, \omega,s,\bm') \vert_{\partial \Omega} \end{array} \right), \lambda(\bm, \cP,\omega,s,\bm')\right)_{\Omega \times \partial \Omega}\right. \nonumber 
    \\
&     \quad \quad \quad \quad +
         \left.                                                       \left(\cG_{\bm,\omega}
         \left(\begin{array}{l} \beta_j u(\bm, \omega,s) \\ \beta_j u(\bm, \omega,s)\vert_{\partial \Omega} \end{array} \right), \mathscr{S}_{\bm,\omega}^* \left( \tm \cG_{\bm, \omega} \lambda(\bm, \cP,\omega,s,\bm')\right) \right)_{\Omega \times \partial \Omega}\right. \nonumber \\
    &  \quad \quad \quad \quad \left. -  \left( \frac{\ri \omega}{4 m^{3/2}} \beta_j \tm u (\bm, \omega,s) \, , \, \lambda(\bm, \cP, \omega,s,\bm')  \right)_{\partial \Omega}  \right],
 \label{I6}
 \end{align}  
 where $\lambda$ is defined in (\ref{lambda}) and $v$ is defined in \eqref{defv}. 
\end{lemma}
 \begin{proof}
   Using Notation \ref{not:conv}  
 and  \eqref{I1}, we can write $(H^{(2)}(\bm, \cP))_{j,k}$ in (\ref{h2}) as 
\begin{align}
\left(H^{(2)}(\bm, \cP)\right)_{j,k}=  - \Re  \left\langle  \mathcal{R}(\cP) \frac{\partial^2 u }{\partial m_k\partial m_{j} }(\bm),  \beps(\bm, \cP) \right\rangle   
& =  - \Re
           \left( \frac{\partial^2 u}{\partial m_k \partial m_{j}} (\bm) , \mathcal{R}(\bp)^* \beps (\bm, \cP) \right)_\Omega . 
           \label{I7}
\end{align}
Then, substituting (\ref{Jac_exp}) into \eqref{I7} and recalling \eqref{lambda}, we obtain
\begin{align}  
 (H^{(2)} (\bm, \cP) )_{j,k}        & = - \Re \left( \cG_{\bm, \omega} \left(\begin{array}{l} u_{j,k}\\ u_{j,k}\vert_{\partial \Omega} \end{array} \right) - \left(
\begin{array}{l} 0\\ \left(\frac{\ri \omega}{4 m^{3/2}}\right) \beta_j\beta_k u\end{array} \right)\, , \, \lambda(\bm, \cP)    \right)_{\Omega \times \partial \Omega} .  \label{I10}  
\end{align}
Before proceeding with \eqref{I10} we first note that,  by \eqref{defujk}, \eqref{deftm}, and then \eqref{I8}, we have     
\begin{align}\label{I11}
  \sum_k u_{j,k} \tm_k = \beta_j \sum_k \frac{\partial u}{\partial m_k} (\bm) \tm_k  + \tm \frac{\partial u} {\partial m_j}(\bm) 
  = \ \beta_j v(\bm, \bm')  + \tm \frac{\partial u}{\partial m_j} (\bm) . 
  \end{align}
Then,  by \eqref{I10}, \eqref{I11} and linearity of $\cG_{\bm, \omega}$,
\begin{align}   (H^{(2)} (\bm, \cP)\tbm)_{j} 
                  \ & = \  - \Re \left(\cG_{\bm,\omega} \left(\begin{array}{l} \beta_j v(\bm, \bm')\\ \beta_j v(\bm, \bm')\vert_{\partial \Omega} \end{array}  \right) , \lambda(\bm, \cP)\right)_{\Omega \times \partial \Omega}\nonumber  \\
                    & \quad  - \Re \left(\cG_{\bm, \omega} \left(\begin{array}{l} \tm \frac{\partial u}{\partial m_j}(\bm)\\\tm \frac{\partial u}{\partial m_j}(\bm)\vert_{\partial \Omega}\end{array}\right) \, , \, \lambda(\bm, \cP) \right)_{\Omega \times \partial \Omega}\nonumber  \\
  & \quad + \Re \left(\left( \begin{array}{l} 0\\ \frac{\ri \omega}{4 m^{3/2}} \beta_j\tm u\vert_{\partial \Omega} \end{array} \right) \, , \, \lambda(\bm, \cP)\right)_{\Omega \times \partial \Omega}.
\label{three_terms}\end{align}
The first and third terms in \eqref{three_terms} correspond to the first and third terms
in \eqref{I6}. The second term in \eqref{three_terms} can be written
\begin{align*}
- \Re \left( \left(\begin{array}{l} \frac{\partial u}{\partial m_j}(\bm) \\ \frac{\partial u}{\partial m_j}(\bm)\vert_{\partial \Omega} \end{array} \right)\, , \, \tm \cG_{\bm,\omega}^* \lambda(\bm, \cP)\right)_{\Omega \times \partial \Omega}
= - \Re \left( \mathscr{S}_{\bm, \omega} \cG_{\bm, \omega} \left(\begin{array}{l} \beta_k u(\bm) \\ \beta_k u (\bm)\vert_{\partial \Omega} \end{array} \right)\, , \, \tm \cG_{\bm,\omega}^* \lambda(\bm, \cP)\right)_{\Omega \times \partial \Omega}
\end{align*}
(where we also used \eqref{oneMgrad2a}), and this corresponds to the second term in \eqref{I6}, completing the proof.   
 \end{proof} 
                         
\par                       As noted in \cite[Section 3.3]{metivier2017full}, the solution of a Hessian system with a matrix-free conjugate
gradient algorithm requires the solution to $2 + 2 N_i$ PDEs, where $N_i$ is the number of conjugate gradient iterations performed.

                         \begin{discussion}[Cost of matrix-vector multiplication with $H$]
\label{disc:Hess}                           Lemmas \ref{thm:H1} and \ref{thm:H2} show that, to compute the product $H(\bm, \cP, \bm') \tbm$ for any $\tbm \in \mathbb{R}^M$ the Helmholtz solves required are (i) computation of $u$ in \eqref{defu}, (ii) computation of $\lambda $ in \eqref{lambda}, (iii) computation of $v$ in \eqref{defv}, and finally (iv) the more-complicated adjoint solve
\begin{align*}
z := \mathscr{S}_{\bm, \omega}^* \left[
 \left( \begin{array}{l} \cR(\cP)^*\cR(\cP) v(\bm,\omega,s,\tbm)\\0 \end{array} \right) -
   \tm \cG_{\bm, \omega} \lambda(\bm, \cP,\omega,s,\bm') 
  \right] 
\end{align*}
The remainder of the calculations in \eqref{H1v} and \eqref{I6} require only inner products and no Helmholtz solves. 
                         \end{discussion}

\end{document}